\documentclass[12pt, reqno]{amsart}
\usepackage{txfonts}    %{article} was 12pt latex e
\usepackage{amssymb}
\usepackage{eucal}
\usepackage{amsmath}
\usepackage{amsthm}
\usepackage{amscd}
\usepackage[dvips]{color}
\usepackage{multicol}
\usepackage[all,pdf]{xy}        %xypic macro for latex2.09
\usepackage{graphicx}
\usepackage{color}
\usepackage{colordvi}
\usepackage{xspace}
\usepackage{rotating}
\usepackage{tikz}
\usepackage{appendix}

\usepackage{ifpdf}
\ifpdf
    \usepackage[colorlinks,final,backref=page,hyperindex]{hyperref}
\else
    \usepackage[colorlinks,final,backref=page,hyperindex,hypertex]{hyperref}
\fi

\newcommand {\emptycomment}[1]{} %to remove paragraphs

\newcommand{\nc}{\newcommand}
\newcommand{\delete}[1]{}

\nc{\mlabel}[1]{\label{#1}}  % Use this to suppress names
\nc{\mcite}[1]{\cite{#1}}  % Use this to suppress names
\nc{\mref}[1]{\ref{#1}}  % Use this to suppress names
\nc{\meqref}[1]{\eqref{#1}} % Use this to suppress names
\nc{\mbibitem}[1]{\bibitem{#1}} % Use this to show number
\delete{
\nc{\mlabel}[1]{\label{#1}  % Use the next two lines to show names
{\hfill \hspace{1cm}{\bf{{\ }\hfill(#1)}}}}
\nc{\mcite}[1]{\cite{#1}{{\bf{{\ }(#1)}}}}  % Use this lines to show names
\nc{\mref}[1]{\ref{#1}{{\bf{{\ }(#1)}}}}  % Use this lines to show names
\nc{\meqref}[1]{\eqref{#1}{{\bf{{\ }(#1)}}}} % Use this lines to show names
\nc{\mbibitem}[1]{\bibitem[\bf #1]{#1}} % Use this to show name
}

\makeatletter
\@namedef{subjclassname@2020}{\textup{2020} Mathematics Subject Classification}
\makeatother
\newtheorem{thm}{Theorem}[section]
\newtheorem{lem}[thm]{Lemma}
\newtheorem{cor}[thm]{Corollary}
\newtheorem{pro}[thm]{Proposition}
\theoremstyle{definition}
\newtheorem{defi}[thm]{Definition}
\newtheorem{ex}[thm]{Example}
\newtheorem{rmk}[thm]{Remark}

\nc{\tred}[1]{\textcolor{red}{#1}}
\nc{\tblue}[1]{\textcolor{blue}{#1}}
\nc{\tgreen}[1]{\textcolor{green}{#1}}
\nc{\tpurple}[1]{\textcolor{purple}{#1}}
\nc{\btred}[1]{\textcolor{red}{\bf #1}}
\nc{\btblue}[1]{\textcolor{blue}{\bf #1}}
\nc{\btgreen}[1]{\textcolor{green}{\bf #1}}
\nc{\btpurple}[1]{\textcolor{purple}{\bf #1}}
\nc{\zis}[1]{\textcolor{purple}{#1}}

\nc{\cm}[1]{\textcolor{red}{Chengming:#1}}
\nc{\li}[1]{\textcolor{blue}{#1}}
\nc{\lir}[1]{\textcolor{blue}{Li:#1}}
\nc{\sy}[1]{\textcolor{purple}{Siyuan:#1}}

\nc{\name}[1]{{\bf #1}}
\nc{\acd}{associative crossed datum\xspace}
\nc{\acds}{associative crossed data\xspace}

\nc{\mmod}[1]{{\ (\mathrm{mod}~{#1})}}

\nc{\qcl}{QCL\xspace} %quasiclassical limit\xspace}
\nc{\qcls}{QCLs\xspace}

\nc{\bfK}{\mathbf{K}}
\nc{\dera}{{d_1}}
\nc{\derb}{{d_2}}

\nc{\vspa}{\vspace{-.1cm}}
\nc{\vspb}{\vspace{-.2cm}}
\nc{\vspc}{\vspace{-.3cm}}
\nc{\vspd}{\vspace{-.4cm}}
\nc{\vspe}{\vspace{-.5cm}}

\nc{\twovec}[2]{\left(\begin{array}{c} #1 \\ #2\end{array} \right )}
\nc{\threevec}[3]{\left(\begin{array}{c} #1 \\ #2 \\ #3 \end{array}\right )}
\nc{\twomatrix}[4]{\left(\begin{array}{cc} #1 & #2\\ #3 & #4 \end{array} \right)}
\nc{\threematrix}[9]{{\left(\begin{matrix} #1 & #2 & #3\\ #4 & #5 & #6 \\ #7 & #8 & #9 \end{matrix} \right)}}
\nc{\twodet}[4]{\left|\begin{array}{cc} #1 & #2\\ #3 & #4 \end{array} \right|}

\nc{\rk}{\mathrm{r}}

\nc{\tforall}{\text{ for all }}

\nc{\svec}[2]{{\tiny\left(\begin{matrix}#1\\
#2\end{matrix}\right)\,}}  % column vector
\nc{\ssvec}[2]{{\tiny\left(\begin{matrix}#1\\
#2\end{matrix}\right)\,}} % subscript column vector

\nc{\typeI}{local cocycle $3$-Lie bialgebra\xspace}
\nc{\typeIs}{local cocycle $3$-Lie bialgebras\xspace}
\nc{\typeII}{double construction $3$-Lie bialgebra\xspace}
\nc{\typeIIs}{double construction $3$-Lie bialgebras\xspace}

\nc{\bia}{{$\mathcal{P}$-bimodule ${\bf k}$-algebra}\xspace}
\nc{\bias}{{$\mathcal{P}$-bimodule ${\bf k}$-algebras}\xspace}

\nc{\rmi}{{\mathrm{I}}}
\nc{\rmii}{{\mathrm{II}}}
\nc{\rmiii}{{\mathrm{III}}}
\nc{\pr}{{\mathrm{pr}}}

\nc{\OT}{constant $\theta$-}
\nc{\T}{$\theta$-}
\nc{\IT}{inverse $\theta$-}
\nc{\rad}{\mathrm{ad}}
\nc{\bt}{\tilde{T}}
\nc{\bi}{\tilde{id}}
\nc{\asi}{ASI\xspace}
\nc{\qadm}{$Q$-admissible\xspace}
\nc{\aybe}{AYBE\xspace}
\nc{\admset}{\{\pm x\}\cup (-x+K^{\times}) \cup K^{\times} x^{-1}}
\nc{\hr}{\hat{r}}
\nc{\dualrep}{gives a dual representation\xspace}
\nc{\admt}{admissible to\xspace}

\nc{\opa}{\cdot_A}
\nc{\opb}{\cdot_B}

\nc{\post}{positive type\xspace}
\nc{\negt}{negative type\xspace}
\nc{\invt}{inverse type\xspace}

\nc{\pll}{\beta}
\nc{\plc}{\epsilon}

\nc{\ass}{{\mathit{Ass}}}
\nc{\lie}{{\mathit{Lie}}}
\nc{\comm}{{\mathit{Comm}}}
\nc{\dend}{{\mathit{Dend}}}
\nc{\zinb}{{\mathit{Zinb}}}
\nc{\tdend}{{\mathit{TDend}}}
\nc{\prelie}{{\mathit{preLie}}}
\nc{\postlie}{{\mathit{PostLie}}}
\nc{\quado}{{\mathit{Quad}}}
\nc{\octo}{{\mathit{Octo}}}
\nc{\ldend}{{\mathit{ldend}}}
\nc{\lquad}{{\mathit{LQuad}}}

 \nc{\adec}{\check{;}} \nc{\aop}{\alpha}
\nc{\dftimes}{\widetilde{\otimes}} \nc{\dfl}{\succ} \nc{\dfr}{\prec}
\nc{\dfc}{\circ} \nc{\dfb}{\bullet} \nc{\dft}{\star}
\nc{\dfcf}{{\mathbf k}} \nc{\apr}{\ast} \nc{\spr}{\cdot}
\nc{\twopr}{\circ} \nc{\tspr}{\star} \nc{\sempr}{\ast}
\nc{\disp}[1]{\displaystyle{#1}}
\nc{\bin}[2]{ (_{\stackrel{\scs{#1}}{\scs{#2}}})}  %binomial coeff
\nc{\binc}[2]{ \left (\!\! \begin{array}{c} \scs{#1}\\
    \scs{#2} \end{array}\!\! \right )}  %binomial coeff
\nc{\bincc}[2]{  \left ( {\scs{#1} \atop
    \vspace{-.5cm}\scs{#2}} \right )}  %binomial coeff
\nc{\sarray}[2]{\begin{array}{c}#1 \vspace{.1cm}\\ \hline
    \vspace{-.35cm} \\ #2 \end{array}}
\nc{\bs}{\bar{S}} \nc{\dcup}{\stackrel{\bullet}{\cup}}
\nc{\dbigcup}{\stackrel{\bullet}{\bigcup}} \nc{\etree}{\big |}
\nc{\la}{\longrightarrow} \nc{\fe}{\'{e}} \nc{\rar}{\rightarrow}
\nc{\dar}{\downarrow} \nc{\dap}[1]{\downarrow
\rlap{$\scriptstyle{#1}$}} \nc{\uap}[1]{\uparrow
\rlap{$\scriptstyle{#1}$}} \nc{\defeq}{\stackrel{\rm def}{=}}
\nc{\dis}[1]{\displaystyle{#1}} \nc{\dotcup}{\,
\displaystyle{\bigcup^\bullet}\ } \nc{\sdotcup}{\tiny{
\displaystyle{\bigcup^\bullet}\ }} \nc{\hcm}{\ \hat{,}\ }
\nc{\hcirc}{\hat{\circ}} \nc{\hts}{\hat{\shpr}}
\nc{\lts}{\stackrel{\leftarrow}{\shpr}}
\nc{\rts}{\stackrel{\rightarrow}{\shpr}} \nc{\lleft}{[}
\nc{\lright}{]} \nc{\uni}[1]{\tilde{#1}} \nc{\wor}[1]{\check{#1}}
\nc{\free}[1]{\bar{#1}} \nc{\den}[1]{\check{#1}} \nc{\lrpa}{\wr}
\nc{\curlyl}{\left \{ \begin{array}{c} {} \\ {} \end{array}
    \right .  \!\!\!\!\!\!\!}
\nc{\curlyr}{ \!\!\!\!\!\!\!
    \left . \begin{array}{c} {} \\ {} \end{array}
    \right \} }
\nc{\leaf}{\ell}       % number of leafs
\nc{\longmid}{\left | \begin{array}{c} {} \\ {} \end{array}
    \right . \!\!\!\!\!\!\!}
\nc{\ot}{\otimes} \nc{\sot}{{\scriptstyle{\ot}}}
\nc{\otm}{\overline{\ot}}
\nc{\ora}[1]{\stackrel{#1}{\rar}}
\nc{\ola}[1]{\stackrel{#1}{\la}}%${\Bbb Z}$
\nc{\pltree}{\calt^\pl}
\nc{\epltree}{\calt^{\pl,\NC}}
\nc{\rbpltree}{\calt^r}
\nc{\scs}[1]{\scriptstyle{#1}} \nc{\mrm}[1]{{\rm #1}}
\nc{\dirlim}{\displaystyle{\lim_{\longrightarrow}}\,}
\nc{\invlim}{\displaystyle{\lim_{\longleftarrow}}\,}
\nc{\mvp}{\vspace{0.5cm}} \nc{\svp}{\vspace{2cm}}
\nc{\vp}{\vspace{8cm}} \nc{\proofbegin}{\noindent{\bf Proof: }}
\nc{\proofend}{$\blacksquare$ \vspace{0.5cm}}
\nc{\freerbpl}{{F^{\mathrm RBPL}}}
\nc{\sha}{{\mbox{\cyr X}}}  %used to be \cyr
\nc{\ncsha}{{\mbox{\cyr X}^{\mathrm NC}}} \nc{\ncshao}{{\mbox{\cyr
X}^{\mathrm NC,\,0}}}
\nc{\shpr}{\diamond}    %Shuffle product
\nc{\shprm}{\overline{\diamond}}    %Shuffle product
\nc{\shpro}{\diamond^0} %Shuffle product
\nc{\shprr}{\diamond^r}  %product on controlled trees
\nc{\shpra}{\overline{\diamond}^r}
\nc{\shpru}{\check{\diamond}} \nc{\catpr}{\diamond_l}
\nc{\rcatpr}{\diamond_r} \nc{\lapr}{\diamond_a}
\nc{\sqcupm}{\ot}
\nc{\lepr}{\diamond_e} \nc{\vep}{\varepsilon} \nc{\labs}{\mid\!}
\nc{\rabs}{\!\mid} \nc{\hsha}{\widehat{\sha}}
\nc{\lsha}{\stackrel{\leftarrow}{\sha}}
\nc{\rsha}{\stackrel{\rightarrow}{\sha}} \nc{\lc}{\lfloor}
\nc{\rc}{\rfloor}
\nc{\tpr}{\sqcup}
\nc{\nctpr}{\vee}
\nc{\plpr}{\star}
\nc{\rbplpr}{\bar{\plpr}}
\nc{\sqmon}[1]{\langle #1\rangle}
\nc{\forest}{\calf}
\nc{\altx}{\Lambda_X} \nc{\vecT}{\vec{T}} \nc{\onetree}{\bullet}
\nc{\Ao}{\check{A}}
\nc{\seta}{\underline{\Ao}}
\nc{\deltaa}{\overline{\delta}}
\nc{\trho}{\tilde{\rho}}

\nc{\rpr}{\circ}
\nc{\dpr}{{\tiny\diamond}}
\nc{\rprpm}{{\rpr}}

\nc{\mmbox}[1]{\mbox{\ #1\ }} \nc{\ann}{\mrm{ann}}
\nc{\Aut}{\mrm{Aut}} \nc{\can}{\mrm{can}}
\nc{\twoalg}{{two-sided algebra}\xspace}
\nc{\colim}{\mrm{colim}}
\nc{\Cont}{\mrm{Cont}} \nc{\rchar}{\mrm{char}}
\nc{\cok}{\mrm{coker}} \nc{\dtf}{{R-{\rm tf}}} \nc{\dtor}{{R-{\rm
tor}}}

\nc{\depth}{{\mrm d}}
\nc{\Div}{{\mrm Div}} \nc{\End}{\mrm{End}} \nc{\Ext}{\mrm{Ext}}
\nc{\Fil}{\mrm{Fil}} \nc{\Frob}{\mrm{Frob}} \nc{\Gal}{\mrm{Gal}}
\nc{\GL}{\mrm{GL}} \nc{\Hom}{\mrm{Hom}} \nc{\hsr}{\mrm{H}}
\nc{\hpol}{\mrm{HP}} \nc{\id}{\mrm{id}} \nc{\im}{\mrm{im}}
\nc{\incl}{\mrm{incl}} \nc{\length}{\mrm{length}}
\nc{\LR}{\mrm{LR}} \nc{\mchar}{\rm char} \nc{\NC}{\mrm{NC}}
\nc{\mpart}{\mrm{part}} \nc{\pl}{\mrm{PL}}
\nc{\ql}{{\QQ_\ell}} \nc{\qp}{{\QQ_p}}
\nc{\rank}{\mrm{rank}} \nc{\rba}{\rm{RBA }} \nc{\rbas}{\rm{RBAs }}
\nc{\rbpl}{\mrm{RBPL}}
\nc{\rdef}{\mrm{def}} \nc{\rdiv}{{\rm div}} \nc{\rtf}{{\rm tf}}
\nc{\rtor}{{\rm tor}} \nc{\res}{\mrm{res}} \nc{\SL}{\mrm{SL}}
\nc{\Spec}{\mrm{Spec}} \nc{\tor}{\mrm{tor}} \nc{\Tr}{\mrm{Tr}}
\nc{\mtr}{\mrm{sk}}
\nc{\rbw}{\rm{RBW }} \nc{\rbws}{\rm{RBWs }} \nc{\rcot}{\mrm{cot}}
\nc{\rest}{\rm{controlled}\xspace}

\nc{\ab}{\mathbf{Ab}} \nc{\Alg}{\mathbf{Alg}}
\nc{\Dend}{\mathbf{DD}} \nc{\bfk}{{\bf k}} \nc{\bfone}{{\bf 1}}
\nc{\rb}{\mathrm{RB}}

\nc{\BA}{{\mathbb A}} \nc{\CC}{{\mathbb C}} \nc{\DD}{{\mathbb D}}
\nc{\EE}{{\mathbb E}} \nc{\FF}{{\mathbb F}} \nc{\GG}{{\mathbb G}}
\nc{\HH}{{\mathbb H}} \nc{\LL}{{\mathbb L}} \nc{\NN}{{\mathbb N}}
\nc{\QQ}{{\mathbb Q}} \nc{\RR}{{\mathbb R}} \nc{\BS}{{\mathbb{S}}} \nc{\TT}{{\mathbb T}}
\nc{\VV}{{\mathbb V}} \nc{\ZZ}{{\mathbb Z}}

\nc{\calao}{{\mathcal A}} \nc{\cala}{{\mathcal A}}
\nc{\calc}{{\mathcal C}} \nc{\cald}{{\mathcal D}}
\nc{\cale}{{\mathcal E}} \nc{\calf}{{\mathcal F}}
\nc{\calfr}{{{\mathcal F}^{\,r}}} \nc{\calfo}{{\mathcal F}^0}
\nc{\calfro}{{\mathcal F}^{\,r,0}} \nc{\oF}{\overline{F}}
\nc{\calg}{{\mathcal G}} \nc{\calh}{{\mathcal H}}
\nc{\cali}{{\mathcal I}} \nc{\calj}{{\mathcal J}}
\nc{\call}{{\mathcal L}} \nc{\calm}{{\mathcal M}}
\nc{\caln}{{\mathcal N}} \nc{\calo}{{\mathcal O}}
\nc{\calp}{{\mathcal P}} \nc{\calq}{{\mathcal Q}} \nc{\calr}{{\mathcal R}}
\nc{\calt}{{\mathcal T}} \nc{\caltr}{{\mathcal T}^{\,r}}
\nc{\calu}{{\mathcal U}} \nc{\calv}{{\mathcal V}}
\nc{\calw}{{\mathcal W}} \nc{\calx}{{\mathcal X}}
\nc{\CA}{\mathcal{A}}

\nc{\fraka}{{\mathfrak a}} \nc{\frakB}{{\mathfrak B}}
\nc{\frakb}{{\mathfrak b}} \nc{\frakd}{{\mathfrak d}}
\nc{\oD}{\overline{D}}
\nc{\frakF}{{\mathfrak F}} \nc{\frakg}{{\mathfrak g}}
\nc{\frakm}{{\mathfrak m}} \nc{\frakM}{{\mathfrak M}}
\nc{\frakMo}{{\mathfrak M}^0} \nc{\frakp}{{\mathfrak p}}
\nc{\frakS}{{\mathfrak S}} \nc{\frakSo}{{\mathfrak S}^0}
\nc{\fraks}{{\mathfrak s}} \nc{\os}{\overline{\fraks}}
\nc{\frakT}{{\mathfrak T}}
\nc{\oT}{\overline{T}}
\nc{\frakX}{{\mathfrak X}} \nc{\frakXo}{{\mathfrak X}^0}
\nc{\frakx}{{\mathbf x}}
\nc{\frakTx}{\frakT}      %All rooted trees, correspond to \ncsha(X)
\nc{\frakTa}{\frakT^a}    % rooted trees for \ncsha(A)
\nc{\frakTxo}{\frakTx^0}   % rooted trees for \ncshao(X)
\nc{\caltao}{\calt^{a,0}}   % rooted trees for \ncshao(A)
\nc{\ox}{\overline{\frakx}} \nc{\fraky}{{\mathfrak y}}
\nc{\frakz}{{\mathfrak z}} \nc{\oX}{\overline{X}}

\font\cyr=wncyr10

\nc{\al}{\alpha}
\nc{\lam}{\lambda}
\nc{\lr}{\longrightarrow}

\newcommand*\mycirc[1]{%
    \begin{tikzpicture}
        \node[draw,circle,inner sep=1pt] {#1};
\end{tikzpicture}}

\begin{document}

\title[Solving the Poisson Yang-Baxter equation via deformation-to-quasiclassical-limit]{Solving the Poisson Yang-Baxter equation via deformation-to-quasiclassical-limits}

\author{Siyuan Chen}
\address{Chern Institute of Mathematics and LPMC, Nankai University, Tianjin 300071, China}
\email{1120210010@mail.nankai.edu.cn}

\author{Chengming Bai}
\address{Chern Institute of Mathematics and LPMC, Nankai University, Tianjin 300071, China}
\email{baicm@nankai.edu.cn}

\author{Li Guo}
\address{Department of Mathematics and Computer Science, Rutgers University, Newark, NJ 07102, USA}
\email{liguo@rutgers.edu}

\begin{abstract}
A fundamental construction of Poisson algebras is to derive them as the quasiclassical limits (\qcls) of associative algebra deformations of commutative associative algebras.
This paper lifts this process to the level of classical Yang-Baxter type equations. Solutions of the Poisson
Yang-Baxter equation (PYBE) in Poisson algebras is obtained by
scalarly deforming solutions of the associative Yang-Baxter
equation (AYBE) in commutative associative algebras.
Inspired by the characterization of solutions of
various classical Yang-Baxter
type equations by $\calo$-operators, we introduce the notions of
deformations of (bi)module algebras and scalar deformations of the corresponding $\mathcal O$-operators, from which the \qcls give $\calo$-operators for Poisson algebras, which in turn provide solutions of the PYBE.
Furthermore, as the \qcls of tridendriform algebra deformations of commutative tridendriform algebras, post-Poisson algebras produce deformations-QCLs of $\calo$-operators for Poisson algebras, thus offering explicit solutions of the PYBE.
\end{abstract}

\subjclass[2020]{
    13D10, % Deformations and infinitesimal methods in com  mutative ring theory [See also 14B10,   14B12,  14D15, 32Gxx]
    16W60, % Valuations, completions, formal power series   and related constructions (associative rings and algebras) [See also 13Jxx]
    17B38, % Yang-Baxter equations and Rota-Baxter operators
    17B63, % Poisson algebras
    53D55,  % Deformation quantization, star products
    13N15   % Derivations and commutative rings
}

\keywords{deformation; quasiclassical limit;
    $\mathcal{O}$-operator; classical Yang-Baxter equation;
    post-Poisson algebra, dendriform algebra, tridendriform algebra, derivation}

\maketitle

\vspace{-1.5cm}

\tableofcontents

\vspace{-1.7cm}

%%%%%%%%%%%%%%%%%%%%%%%%%%%%%%%%%%%%%%%%%%%%%%%%%%%%%%%%%%%%%%%%%%%%%%%%%%%%%%%%

\section{Introduction}
The purpose of this paper is to lift the construction of Poisson algebras as the quasiclassical limits (QCLs) of associative algebra deformations of commutative algebras to the level of the Poisson Yang-Baxter equation (PYBE), namely constructing solutions of the PYBE as the QCLs of deformed solutions of the associative Yang-Baxter equation (AYBE) from solutions in commutative algebras. To achieve this goal, we employ the following two strategies. 
    \begin{enumerate}
\item We adapt the $\calo$-operator approach of the classical Yang-Baxter equation (CYBE) and AYBE to the PYBE, deriving solutions of the PYBE in a \qcl Poisson algebra from \qcl $\calo$-operators for Poisson algebras (Theorem~\mref{pro-w1});
\item We further adapt the Zinbiel algebra (namely commutative dendriform algebras) construction of $\calo$-operators to pre-Poisson algebras, constructing \qcl $\calo$-operators for Poisson algebras from \qcl pre-Poisson algebras (Theorem~\mref{pro-diaid}).
    \end{enumerate}
As a consequence, deformations of Zinbiel algebras provide solutions of the PYBE (Corollary~\mref{cor:summary}). Moreover, our general approach also shows that deformations of commutative tridendriform algebras yield solutions of the PYBE (Theorem~\mref{pro-w2}).

\vspb
\subsection{Deformations of algebras and quasiclassical limits}
The deformation theory of associative algebras was initiated by
Gerstenhaber \mcite{G1}. For an associative algebra $(A,\cdot)$, its formal deformation is an associative product $\cdot_{h}$ on
$A[[h]]$ such that
\[\vspb x\cdot_{h}y=x\cdot y+\sum_{s=1}^{\infty}x\cdot_{s} y~h^s,\quad  x,y\in A,
\vspa
\]
where $x\cdot_{s}y\in A$. When $(A,\cdot)$ is commutative, then
with the bracket $[,]$ defined by
\vspb
$$[x,y]=x\cdot_1
y-y\cdot_{1}x,
\vspa
$$
the triple $(A,[,],\cdot)$ is called the {\bf quasiclassical limit
    (QCL)} of the deformation.~\footnote{Quasiclassical limit also bears the name classical limit or
    semiclassical limit. We follow the convention of \mcite{ES}.} The
triple $(A,[,],\cdot)$ is a Poisson algebra, henceforth called the {\bf \qcl Poisson algebra} of the deformation. The (topological)
algebra $(A[[h]],\cdot_{h})$ is called a quantization of the
Poisson algebra $(A,[,],\cdot)$, as shown in the diagram
\vspb
    \begin{equation}
        \xymatrix{
            *\txt{\small commutative algebra}\ar[rr]^{\txt{\tiny deformation}}
            &&*\txt{\small associative algebra}\ar@<2pt>[rr]^{\txt{ \tiny
                    \qcl}}
            &&*\txt{\small Poisson algebra} \ar@<2pt>[ll]^{\txt{\tiny quantization}}
        }
        \mlabel{eq:defqcl}
    \end{equation}
The quantization procedure
establishes a transition from Poisson bracket to commutator, which corresponds to the transition from classical mechanics to quantum mechanics. The problem of quantizing the Poisson algebra of smooth
functions on a Poisson manifold was resolved in the well-known
work of Kontsevich \mcite{Ko}.

The study for deformation and QCL has been extended to dendriform
algebras and tridendriform algebras, both of which are associative
algebras whose products split into several operations. Dendriform
algebras were introduced by Loday \cite{Lo} in his study of
algebraic $K$-theory and tridendriform algebras were introduced by
Loday and Ronco \mcite{LR} in their study of the Koszul duality
and polytopes. The QCL of the dendriform  algebra deformation of a
Zinbiel algebra, which is the commutative version of a dendriform
algebra, was shown by Aguiar \mcite{A2} to be a pre-Poisson
algebra. It is further shown by Ospel, Panaite and
Vanhaecke~\mcite{OPV} that the QCL corresponding to the
tridendriform algebra deformation of a commutative tridendriform
algebra is a post-Poisson algebra, a notion introduced in
\mcite{BBGN} from an operadic point of view and in \mcite{NB}
related to Poisson bialgebras.
See \cite{Do-PBW,LBS,LS} for the applications of pre-Poisson algebras and post-Poisson algebras.
\vspb

\subsection{Classical Yang-Baxter type equations and $\mathcal{O}$-operators}
Our purpose is to lift the deformation-to-\qcl process for Poisson
algebras in \meqref{eq:defqcl} to the level of the Poisson Yang-Baxter equation (PYBE). The PYBE is a combination of the
CYBE in Lie algebras and the
AYBE in associative algebras. As is well-known, the CYBE has
its importance as the classical limit of the quantum Yang-Baxter
equation and as an important notion in integral systems. As a
fundamental device to find solutions of the CYBE, the
$\calo$-operator on a Lie algebra was introduced by
Kupershmidt~\mcite{Ku} and was shown to characterize solutions of the CYBE. In general, all
$\calo$-operators give rise to solutions of the
CYBE~\mcite{Bai0}. Furthermore, pre-Lie algebras, in addition to
their independent interest, naturally produce $\calo$-operators
and hence solutions of the CYBE. This relations can be
depicted by the diagram
\begin{equation}
\xymatrix{ \txt{\small pre-Lie algebras} \ar[rr] && \txt{\small $\calo$-operators on Lie algebras} \ar@<2pt>[rr] && \txt{\small solutions of\\\small the CYBE}\ar@<2pt>[ll]  }
\mlabel{eq:opcybe}
\end{equation}
On the associative side, the AYBE was introduced by Aguiar \mcite{A1,A3} and by Polishchuk \mcite{P1} in the form with spectral
parameters. The AYBE is an analog of the well-known CYBE.
Numerous applications of the AYBE can be found in  bialgebra theory \mcite{A3}, mirror symmetry
\mcite{P1}, algebraic geometry \mcite{P2}, integrable systems\mcite{MS} and magnetization dynamics \mcite{AZ}.
The AYBE can also be studied along the line of the CYBE as in \meqref{eq:opcybe}. There $\calo$-operators are also called relative Rota-Baxter operators~\mcite{BGN4} and generalized Rota-Baxter operators~\mcite{U}, have been studied extensively recently~\mcite{BGN2,D}, in the contexts of Poisson structures, deformations and cohomology.

The PYBE arose from the study of Poisson bialgebras
\mcite{NB}. Since the Poisson algebra is the quasiclassical limit
of an associative deformation, provided with suitable solutions of
the AYBE in the commutative algebra and in the associative
deformation, we may expect that the solution of the PYBE in the
QCL Poisson algebra can be obtained. Similarly, the
$\mathcal{O}$-operator on a Poisson algebra has been applied to
study the PYBE \mcite{LBS, NB}.

\subsection{Solutions of the PYBE as \qcls of scalarly deformed solutions of the AYBE}\label{sss}
This study extends the deformation-\qcl construction of Poisson algebras as illustrated in \meqref{eq:defqcl} to the level of PYBE. Inspired by the operator approach to CYBE outlined in \meqref{eq:opcybe}, we are also led to the deformation-\qcl constructions of $\calo$-operators and of pre-Poison algebras and post-Poisson algebras.

Our overall approach is illustrated in the following diagram where the horizontal direction follows \meqref{eq:defqcl} and the vertical direction is inspired by \meqref{eq:opcybe}. A more precise summary can be found in the diagrams  (\ref{diagw1}) and (\ref{maindiag}).
\vspb
$$
%{\small \xymatrixcolsep{5pc}\xymatrixrowsep{2.5pc}
\xymatrix{
\mycirc{3}  \txt{\small Zinbiel algebra\\ (pre-commutative algebra)}
    \ar[rr]^{\txt{\tiny deformation}}\ar[d]_{\txt{\tiny pre-structure}}^{\txt{\tiny realization}}
    && *\txt{\small dendriform algebra\\ (pre-associative algebra)}\ar[rr]^{\txt{ \tiny
            \qcl}}\ar[d]_{\txt{\tiny pre-structure}}^{\txt{\tiny realization}}
    && *\txt{\small pre-Poisson algebra}\ar[d]_{\txt{\tiny pre-structure}}^{\txt{\tiny realization}}\\
\hspace{-.5cm}\mycirc{2} \txt{\small $\mathcal{O}$-operator\\ on commutative algebra}\ar[rr]^{\txt{\tiny deformation}}\ar@<2pt>[d]_{\txt{\tiny operator}}^{\txt{\tiny approach}}
    &&*\txt{\small $\mathcal{O}$-operator\\ on associative algebra}\ar[rr]^{\txt{ \tiny
            \qcl}}\ar@<2pt>[d]_{\txt{\tiny operator}}^{\txt{\tiny approach}}
    &&*\txt{\small $\mathcal{O}$-operator\\ on Poisson algebra}\ar@<2pt>[d]_{\txt{\tiny operator}}^{\txt{\tiny approach}}\\
\hspace{-1cm}\mycirc{1}  \txt{\small solution of AYBE in \\ commutative algebra}\ar@<2pt>[u]\ar[rr]^{\txt{\tiny deformation}}
    &&*\txt{\small solution  of AYBE in\\ associative algebra}\ar@<2pt>[u]\ar[rr]^{\txt{ \tiny
            \qcl}}
    &&*\txt{\small solution of PYBE in\\ Poisson algebra}\ar@<2pt>[u]
}
$$

We next follow this diagram to summarize our approach on solving the PYBE in three steps.

\subsubsection{Step 1: Solving PYBE by deforming solutions of AYBE in commutative algebras}
Here we take the direct approach of treating the solutions of the equations in tensor forms. Started from a solution of the AYBE in a commutative associative algebra with invariant
symmetric part, if its scalar deformation is a solution of the AYBE in the deformed associative algebra with invariant
symmetric part, then we obtain a solution of the PYBE in the QCL Poisson algebra (Theorem~\ref{inv-aybe}). But it is challenging to determining whether the scalar deformation of a solution of the AYBE is a solution of the AYBE in a given associative deformation.

\subsubsection{Step 2: Solving PYBE by deforming $\calo$-operators for commutative algebras}
Motivated by the $\calo$-operator characterizations of solutions of the AYBE and PYBE~\mcite{BGN2,NB},
we find characterization of the scalarly deformed solutions of the AYBE in the deformed associative algebras by scalar deformations of $\calo$-operators on commutative algebras.
For the latter $\calo$-operators, we deform an associative
algebra and its bimodule algebra, and then take a scalar deformation of the associated $\mathcal{O}$-operator which is still an $\mathcal{O}$-operator for the deformed associative algebras. In this case, the original $\mathcal{O}$-operator is also an $\mathcal{O}$-operator associated to the \qcl module Poisson algebra (Theorem \mref{invrb}). Moreover, the deformation-QCL process for solutions of classical Yang-Baxter type equations is equivalently presented as the deformation-QCL process for $\calo$-operators (Theorem~\ref{pro-w1}).

On the other hand, $\calo$-operators for associative algebras and Poisson algebras give rise to skew-symmetric solutions of the AYBE and PYBE, respectively~\mcite{BGN2,LBS}.
We further show that scalar deformations of $\calo$-operators give rise to skew-symmetric solutions of the PYBE (Theorem~\ref{pro-skews}). Even though, it is still difficult to determine whether the scalar
deformation of an $\calo$-operator is an $\calo$-operator
associated to a given bimodule algebra deformation, the operator approach has the advantage that the $\calo$-operators can be realized by the pre- or post-structures, leading us to the next and final step.

\subsubsection{Step 3: Solving PYBE by deforming commutative pre-structures and post-structures}
According to~\mcite{BGN4}, a tridendriform algebra is a bimodule algebra for which the identity map is an $\calo$-operator. We show that the scalar deformation
of this identity $\calo$-operator is again an $\calo$-operator associated to the corresponding bimodule algebra deformation
(Theorem~\ref{pro-diaid}), thus providing instances of scalarly deformed $\calo$-operators. Then combining with Step 2, we obtain solutions of the PYBE in the \qcls Poisson algebras (Corollary~\mref{cor:summary} and Theorem~\ref{pro-w2}).

\subsection{Organization of the paper}

With the main ideas outlined above, the paper is structured as follows, beginning with the more fundamental concepts that are of interest in their own right.

In Section \mref{S3}, we give the notion of a deformation of a
bimodule algebra and determine the corresponding QCL. Then we show
that a bimodule algebra deformation with a scalar deformation of
an $\mathcal{O}$-operator provides an $\mathcal{O}$-operator
associated to the QCL.

In Section \mref{S2}, we apply bimodule algebra deformations with scalar deformations of $\mathcal{O}$-operators to tridendriform deformations of commutative tridendriform algebras.
We give an explicit deformation formula via derivations.

In Section \mref{s:ybe}, we apply bimodule algebra deformations with scalar deformations of $\mathcal{O}$-operators
to establish the links between certain solutions of the AYBE and
the PYBE, which are illustrated by commutative diagrams.
\smallskip

\noindent
{\bf Notations.} Throughout this paper, we fix a ground field $\bfk$ of characterization $0$ for vector spaces, algebras and tensor products. We also fix the power series ring $\mathbf{K}=\bfk[[h]]$ as the ground algebra for topological modules, topological algebras and topological tensor products. For a vector space $A$ with an operation $\circ$, the linear maps $L_\circ, R_\circ: A\rightarrow {\rm End}_{\bf k}(A)$ are respectively defined by
\vspa
$$L_\circ(x)y:=x\circ y,\;\;R_\circ(x)y:=y\circ x,\;\; x,y\in A.
\vspa
$$
Let $\mathbb{N}$ and $\mathbb{Z}^{+}$ denote the set of natural numbers and the set of positive integers respectively. 
\vspb

\section{Bimodule algebra deformations and scalar deformations of $\mathcal{O}$-operators}\mlabel{S3}
This section introduces the notions
of a bimodule algebra deformation and the corresponding
quasiclassical limit (QCL). The scalar deformation of an
$\calo$-operator on a commutative associative algebra is also
introduced to give an $\calo$-operator on the Poisson algebra as
the QCL.

\subsection{Deformations of bimodule algebras and the corresponding quasiclassical limits}
\mlabel{ss:deformgen}
We introduce a general notion of bimodule algebras in order to provide a uniform framework for several structures considered in this paper.

Let $\calp$ be a binary quadratic operad~\mcite{LV}. For convenience, we also let $\calp$ denote the
category of $\calp$-algebras.

\begin{defi}
Let $\calp$ be a binary quadratic operad with binary operations $\{\mu_i\}_{i\in I}$. Let $(A,\{\circ_i\}_{i\in I})$ and $(V,\{\cdot_i\}_{i\in I})$ be $\calp$-algebras. Let $l_i,r_i: A\rightarrow \textrm{End}_{\bfk}(V), i\in I$ be linear maps. If the vector space direct sum $A\oplus V$ is again a $\calp$-algebra with the multiplications $\odot_i$ defined by
\begin{equation}\mlabel{eq:pdimod}
    (x,u)\odot_i(y,v)=(x\circ_i y,l_i(x)v+r_i(y)u+u\cdot_i v), \quad  x, y\in A, u, v\in V, i\in I,
\end{equation}
then the quadruple $(V,\{\cdot\}_i,\{l_i\}_i,\{r_i\}_i)$ is called an $A$-\textbf{bimodule $\calp$-algebra} and the $\calp$-algebra $(A\oplus V,\{\odot_i\})$ is denoted by $A\ltimes_{l_i,r_i} V$. If the operations $\mu_i$ are either symmetric (meaning commutative) or antisymmetric, implying $r_i=\pm l_i, i\in I$,
%\lir{Is this true?} \sy{It is true.}
then the triple $(V,\{\cdot\}_i,\{l_i\}_i)$ is called an $A$-\textbf{module $\calp$-algebra}. The prefixes $A$- or $\calp$- will be suppressed if the context is clear.

Starting with a vector space $V$, equip $V$ with a $\calp$-algebra structure by taking $\cdot_i$ to be the zero multiplication. Then the resulting $A$-bimodule
algebra $(V,\{\cdot_i\},\{l_i\},\{r_i\})$ is called an
$A$-\textbf{bimodule} and is simply denoted by
$(V,\{l_i\},\{r_i\})$. In the case when each $\circ_i$ is either
symmetric or antisymmetric, then $V$ is also called an
$A$-\textbf{module}. The resulting $\calp$-algebra $(A\oplus
V,\{\odot\}_i)$ is still denoted by $A\ltimes_{l_i,r_i} V$.
\mlabel{de:pbimod}
\end{defi}

\vspd

When $\calp$ is the operad of associative algebras, we recover the notion of a bimodule algebra in~\mcite{BGN2}, originally formulated by relations among the binary operations and maps as follows. See also~\mcite{Me,Lue2} for a similar construction. %\lir{Introduction also mentioned \mcite{Lue}} \sy{The reference is added.}

\begin{pro}{\rm{\mcite{BGN2}}}\mlabel{abimodag-eqv}
Let $(A,\circ)$ and $(V,\cdot)$ be associative algebras. Let $l,r: A\rightarrow \End_{\bfk}(V)$ be linear maps. Then $(V,\cdot,l,r)$ is an $A$-bimodule algebra if and only if the following conditions are satisfied.
\vspb
\begin{align}
 l(x\circ y) =& l(x)l(y), &l(x)r(y) &= r(y)l(x), &r(x\circ y)&= r(y)r(x), \mlabel{abimod}\\
 l(x)(u\cdot v) =& (l(x)u)\cdot v, &(r(x)u)\cdot v &= u\cdot(l(x)v), &r(x)(u\cdot v) &= u\cdot(r(x)v),\  x, y\in A, u, v\in V. \mlabel{abimoda}
\end{align}
\end{pro}

Let $A$ be an associative algebra and let $V$ be a vector space
equipped with a zero operation $\cdot$. Then the resulting
$A$-bimodule algebra $(V,\cdot, l,r)$ coincides with the usual
$A$-\textbf{bimodule} defined by Eq.~\meqref{abimod}.

Now apply the general formulation to the operad $\calp$ of
commutative associative algebras. As noted above,
$l(x)=r(x):=\rho_\circ(x)$ for $x\in A$. Then (\mref{abimod}) and
(\mref{abimoda}) reduce to
\vspb

\begin{equation}\mlabel{cba}
  \rho_{\circ}(x\circ y) =\rho_{\circ}(x)\rho_{\circ}(y)\quad\textrm{and}\quad\rho_{\circ}(x)(u\cdot v) =(\rho_{\circ}(x)u)\cdot
  v ,\;\; x,y\in A, u,v\in V,
\end{equation}
respectively. So in this case an
{\bf $A$-module algebra} is simply the usual commutative $A$-algebra. For notational consistency in the paper, we will use the former term as in~\mcite{NB}.

For the operad of Lie algebras, the notion of module
$\calp$-algebras will be detailed in Section\,\mref{ss:ybeopdend}
(see the paragraph below Theorem~\ref{tri-taybe}).
At this moment, for the sake of later references, we fix some notations when $\calp$ is the
operad of Poisson algebras. Recall that a \textbf{Poisson algebra}
$(P,\{,\},,\circ)$ consists of a Lie algebra $(P,\{,\})$ and a
commutative associative algebra $(P,\circ)$ such that
\vspb
\begin{equation}\mlabel{PostP1}
    \{x,y\circ z\}=\{x,y\}\circ z+y\circ\{x,z\},\;\; x,y,z\in P.
\end{equation}
\begin{defi}\rm{\mcite{NB}}
Let $(P,\{,\},\circ)$ and $(Q,[,],\cdot)$ be Poisson algebras, and let $\rho_{\{,\}},\rho_{\circ}: P\rightarrow \mathrm{End}_{\bfk}(Q)$ be linear maps. If the
direct sum of vector spaces $P\oplus Q$ is turned into a Poisson algebra with the bracket $\llbracket,\rrbracket$  and the multiplication $\odot$ respectively defined by
\vspb
\begin{eqnarray}
\llbracket(x,u),(y,v)\rrbracket&:=&(\{x, y\},\rho_{\{,\}}(x)v-\rho_{\{,\}}(y)u+[u,
  v]),\mlabel{lrbrp}\\
(x,u)\odot(y,v)&:=&(x\circ y,\rho_{\circ}(x)v+\rho_{\circ}(y)u+u\cdot v),
\quad  x, y\in P, u, v\in Q,\mlabel{odotp}
\end{eqnarray}
then we call $(Q,[,],\cdot,\rho_{\{,\}},\rho_{\circ})$ a
$P$-\textbf{module Poisson algebra} and denote the resulting
Poisson algebra $(P\oplus Q,\llbracket,\rrbracket,\odot)$ by
$P\ltimes_{\rho_{\{,\}},\rho_{\circ}} Q$.

The same notion of $P$-\textbf{module} is used if $Q$ is only a vector space but equipped with the zero multiplications as in the general construction in Definition~\mref{de:pbimod}.
\mlabel{Pma}
\end{defi}
\vspb
\begin{rmk}
There is also an equivalent characterization of a $P$-module Poisson algebra as well as a $P$-module in \mcite{NB} in terms of a system of equations similar to \eqref{abimod} and \eqref{abimoda}. We will not list them since they are not needed in the rest of this paper.
\end{rmk}

We next recall the needed notations for (formal) deformations as follows \cite{ES}. Let
$$\mathbf{K}:=\bfk[[h]]$$
be the formal power series algebra and let $V$ be a vector space. Consider the set of formal series
\vspb
\[V_h:=V[[h]]:= \Big\{\sum_{i=0}^{\infty}v_{i}~h^{i}~\Big|~v_{i}\in V\Big\}.
\vspb
\]
For $v^{\prime}=\sum\limits_{i=0}^{\infty}v_{i}~h^{i},w^{\prime}=\sum\limits_{i=0}^{\infty}w_{i}~h^{i}\in V[[h]]$ and
$\lambda^{\prime}=\sum\limits_{i=0}^{\infty}\lambda_{i}~h^{i}\in\mathbf{K}$, set
\vspc
\[v^{\prime}+w^{\prime}:=\sum_{i=0}^{\infty}(v_{i}+w_{i})h^{i} , \quad \lambda^{\prime} v^{\prime}:=\sum_{i=0}^{\infty}\Big(\sum_{m+n=i}\lambda_{m}~v_{n}\Big)~h^{i}.
\vspb
\]
Then $V[[h]]$ becomes a $\mathbf{K}$-module.

For a fixed real number $C>1$, define the \textbf{$h$-adic norm} $\|~\|$ on $V[[h]]$ by
\vspb
\[\Big\|\sum_{i=0}^{\infty}v_{i}~h^{i}\Big\|:= C^{-m},\]
where $m$ is the smallest integer such that $v_{m}\neq 0$. The topology induced by this norm is called the \textbf{$h$-adic topology}.
A $\mathbf{K}$-module that is $\bf K$-linearly isomorphic to $V[[h]]$ for a vector space $V$ is called a \textbf{topologically free $\mathbf{K}$-module}. In particular, $V_h$ is a topologically  free $\bfK$-module.

 Let $A_{h}$ be a topologically free $\mathbf{K}$-module. If there is a $\mathbf{K}$-bilinear operation $\cdot_{h}$ on $A_{h}$ such that
 \vspb
  \[(x\cdot_{h} y)\cdot_{h} z=x\cdot_{h}(y\cdot_{h} z),\;\; x,y,z\in A_h,\]
  then $(A_{h},\cdot_{h})$ is called a \textbf{topologically free associative algebra}.
\begin{defi}
  A \textbf{deformation of an associative algebra} $(A,\cdot)$ is a topologically free associative algebra $(A_{h},\cdot_{h})$ such that
    \begin{equation}
       x\cdot_{h}y \equiv x\cdot y\pmod h,\;\;x,y\in A.\mlabel{defa2}
    \end{equation}
\end{defi}
\begin{rmk}
  \begin{enumerate}
 \item In \mcite{ES},
the condition (\mref{defa2}) is replaced by
the equivalent condition
 $A=A_{h}/hA_{h}$.
 \item
To describe the associativity for a topologically free associative algebra, we need the notion of a \textbf{topological tensor product}.
Let $V[[h]]$ and $W[[h]]$ be topologically free $\mathbf{K}$-modules. The topological tensor product $V[[h]]\hat{\otimes}W[[h]]$ is the completion of
$V[[h]]\otimes_{\mathbf{K}}W[[h]]$ under the $h$-adic norm. So the topological tensor product is the completion of the $\mathbf{K}$-bilinear extension of the usual (algebraic) tensor product.
\end{enumerate}
\end{rmk}

A deformation $(A_{h},\cdot_{h})$ of an associative algebra
$(A,\cdot)$ is also called an \textbf{associative algebra
deformation} or simply an \textbf{associative deformation} of
$(A,\cdot)$.

\begin{defi}
Let $(A_{h},\cdot_{h})$ be an associative deformation of a commutative associative algebra $(A,\cdot)$. Let $[,]$ be an operation on $A$ such that
\[[x,y]\equiv\frac{x\cdot_{h}y-y\cdot_{h}x}{h}\mmod{h},\;\;  x, y\in A.\]
Then $(A,[,],\cdot)$ is called the
\textbf{quasiclassical limit (\qcl) of the associative
deformation} $(A_{h},\cdot_{h})$.
\end{defi}

The following result is well known.

\begin{thm}{\rm\mcite{ES}}
Let $(A_{h},\cdot_{h})$ be an associative deformation of a commutative associative algebra $(A,\cdot)$. Then its \qcl $(A,[,],\cdot)$ is a Poisson algebra.
\end{thm}

We next extend the associative deformations to bimodule algebras.
\begin{defi}
Let $(A_{h},\circ_{h})$ and $(V_{h},\cdot_{h})$ be
topologically free associative algebras. Let $l_{h},r_{h}:
A_{h}\rightarrow \mathrm{End}_{\mathbf{K}}(V_{h})$ be
$\mathbf{K}$-linear maps.
If $(V_{h},\cdot_{h},l_{h},r_{h})$ is an $A_{h}$-bimodule algebra, then we call $(V_{h},\cdot_{h},l_{h},r_{h})$ a \textbf{topologically free} $A_{h}$-\textbf{bimodule algebra} and denote the topologically free associative algebra structure on $A_{h}\oplus V_{h}$ by $A_{h}\ltimes_{l_{h},r_{h}} V_{h}$.

Let $(A,\circ)$ be an associative algebra and $(V,\cdot,l,r)$ be
an $A$-bimodule algebra. Let $(A_{h},\circ_{h})$ be a
topologically free associative algebra and
$(V_{h},\cdot_{h},l_{h},r_{h})$ be a topologically free
$A_{h}$-bimodule algebra. If $A_{h}\ltimes_{l_{h},r_{h}} V_{h}$
is an
associative deformation of
$A\ltimes_{l,r} V$, then we call
$(V_{h},\cdot_{h},l_{h},r_{h})$ an $A_{h}$-\textbf{bimodule
algebra deformation of the} $A$-\textbf{bimodule algebra }
$(V,\cdot,l,r)$.
\end{defi}

\begin{rmk}\label{ids}
    Let $\imath:A_{h}\oplus V_{h}\to(A\oplus V)_{h}$ be the map defined by
    \[\imath\,\Big(\sum_{i=0}^{\infty}x_{i}~h^{i},\sum_{i=0}^{\infty}v_{i}~h^{i}\Big):=\sum_{i=0}^{\infty}(x_{i},v_{i})~h^{i},\ \ i\in\NN,x_{i}\in A, v_{i}\in V.\]
    Then $\imath$ is a $\bf K$-linear isomorphism.
With $A_{h}\oplus V_{h}$ identified with $(A\oplus V)_{h}$ this way, it makes sense to say that $A_{h}\ltimes_{l_{h},r_{h}} V_{h}$ is an associative deformation of $A\ltimes_{l,r} V$.
\end{rmk}
\begin{pro}\label{edm}
    Let $(A,\circ)$ be an associative algebra and $(V,\cdot,l,r)$ be
    an $A$-bimodule algebra. Let $(A_{h},\circ_{h})$ be a
    topologically free associative algebra and
    $(V_{h},\cdot_{h},l_{h},r_{h})$ be a topologically free
    $A_{h}$-bimodule algebra. Then
    $(V_{h},\cdot_{h},l_{h},r_{h})$ is an $A_{h}$-bimodule
        algebra deformation of the $A$-bimodule algebra
    $(V,\cdot,l,r)$ if and only if the following congruences
    hold$:$
    \begin{eqnarray}
    &&x\circ_h y\equiv x\circ y\mmod{h},\quad    u\cdot_h v\equiv u\cdot v\mmod{h},\label{e1}\\
    &&l_h(x)v\equiv l(x)v\mmod{h},\quad r_{h}(y)u\equiv r(y)u\mmod{h}, \quad x, y\in A, u, v\in V.
    \label{e2}
    \end{eqnarray}
\end{pro}
\begin{proof}
Let $\odot_h$ denote the associative product
    of the topologically free associative algebra $A_{h}\ltimes_{l_{h},r_{h}} V_{h}$. Then $A_{h}\ltimes_{l_{h},r_{h}} V_{h}$
    is an associative deformation of the associative algebra $A\ltimes_{l,r} V$
    if and only if
    \begin{equation*}
         (x,u)\odot_{h}(y,v)\equiv (x,u)\odot(y,v)\mmod{h}, \quad x,y\in A, u,v\in V .
    \end{equation*}
This holds if and only if
    \begin{equation}\label{em}
          (x\circ_h y,l_h (x)v+r_h (y)u+u\cdot_h v)\equiv (x\circ y,l(x)v+r(y)u+u\cdot v)\mmod{h}, \quad x,y\in A, u,v \in V.
    \end{equation}
By suitably choosing $x$, $y$, $u$ and $v$, we find that \eqref{em} holds if and only if
    (\ref{e1}) and (\ref{e2}) hold, yielding the conclusion. \end{proof}

\begin{rmk}
Condition (\ref{e1}) holds if and only if
    $(A_{h},\circ_{h})$ is an
    associative deformation of $(A,\circ)$ and  $(V_{h},\cdot_{h})$ is an
    associative deformation of $(V,\cdot)$.
\end{rmk}

\begin{defi}
  Let $(A,\circ)$ be a commutative associative algebra and $(V,\cdot,\rho_{\circ})$ be an $A$-module algebra.
Let  $(A_{h},\circ_{h})$ and $(V_{h},\cdot_{h})$ be the
associative deformations of $(A,\circ)$ and $(V,\cdot)$
respectively, and $(A,\{,\},\circ)$ and $(V,[,],\cdot)$ be the
corresponding \qcls. Let $(V_{h},\cdot_{h},l_{h},r_{h})$ be an $A_{h}$-bimodule
algebra deformation of the $A$-module algebra
$(V,\cdot,\rho_{\circ})$. Let $\rho_{\{,\}}: A\rightarrow
\mathrm{End}_{\bfk}(V)$ be a linear map satisfying
  \begin{equation}\mlabel{cmod}
    \rho_{\{,\}}(x)(u)\equiv\frac{l_{h}(x)u-r_{h}(x)u}{h}\mmod{h}, \;\; x\in A, u\in
    V.
  \end{equation}
We call $(V,[,],\cdot,\rho_{\{,\}},\rho_{\circ})$ the
\textbf{quasiclassical limit (\qcl) of the} $A_{h}$-\textbf{bimodule
algebra deformation} $(V_{h},\cdot_{h},l_{h},r_{h})$.
\vspb
\end{defi}
\begin{rmk}\label{qclbd} \begin{enumerate}
\item\label{wdf} The fraction $\frac{l_{h}(x)u-r_{h}(x)u}{h}$ is a
well-defined element in $V_{h}$ for each $x\in A$ and $u\in V$,
since $l_{h}\equiv r_{h}\mmod{h}$. Moreover, $\rho_{\{,\}}(x)(u)$
is the zeroth order term of $\frac{l_{h}(x)u-r_{h}(x)u}{h}$ and
the linear map $\rho_{\{,\}}$ in \meqref{cmod} is well defined.
\item In the same way or by taking $\cdot$ and
$\cdot_h$ being null operations in the above definitions, we can
give the notions of a \textbf{topologically free bimodule}, an
$A_{h}$-\textbf{bimodule deformation} and the
\textbf{quasiclassical limit (\qcl) of the}
$A_{h}$-\textbf{bimodule deformation}.
\end{enumerate}
\end{rmk}

\begin{thm}
Let $(A,\circ)$ be a commutative associative algebra,
$(A_{h},\circ_{h})$ be an associative deformation of $(A,\circ)$
and $(A,\{,\},\circ)$ be the Poisson algebra as the \qcl.
Let $(V,\cdot,\rho_{\circ})$ be an $A$-module
algebra and $(V_{h},\cdot_{h},l_{h},r_{h})$ be an $A_{h}$-bimodule
algebra
deformation. Then the corresponding \qcl $(V,[,],\cdot,\rho_{\{,\}},\rho_{\circ}
 )$ is an $(A,\{,\},\circ)$-module Poisson algebra. Moreover, the \qcl of the associative deformation
$A_{h}\ltimes_{l_{h},r_{h}} V_{h}$ of the commutative associative algebra $A\ltimes_{\rho_{\circ}} V$ is exactly the Poisson algebra $A\ltimes_{\rho_{\{,\}},\rho_{\circ}} V$.
\mlabel{Thmdefmoda}
\end{thm}

\begin{proof}
As the \qcl of $(V_{h},\cdot_{h})$, the triple $(V,[,],\cdot)$ is a Poisson algebra. Define
 operations $\llbracket,\rrbracket$ and $\odot$ on the direct
sum of vector spaces $A\oplus V$ by \eqref{lrbrp} and
\eqref{odotp} respectively, that is,
\vspb
  \begin{eqnarray*}
    \llbracket(x,u),(y,v)\rrbracket&=:&(\{x, y\},\rho_{\{,\}}(x)v-\rho_{\{,\}}(y)u+[u, v]), \\
   (x,u)\odot(y,v)&=:&(x\circ y,\rho_{\circ}(x)v+\rho_{\circ}(y)u+u\cdot v), \quad  x, y\in A, u, v\in V.
\vspa
  \end{eqnarray*}
Then $(A\oplus V,\odot)$ is a commutative associative algebra. Let
$(A_{h}\oplus V_{h},\odot_{h}):=A_{h}\ltimes_{l_{h},r_{h}} V_{h}$
be the associative deformation. For $x,y\in A$ and $u,v\in V$, we
have \vspb
\begin{equation*}
  \begin{split}
      & (x,u)\odot_{h}(y,v)-(y,v)\odot_{h}(x,u)\\
       &=(x\circ_{h}y-y\circ_{h}x,l_{h}(x)v-r_{h}(x)v-l_{h}(y)u+r_{h}(y)u+u\cdot_{h}v-v\cdot_{h}u) \\
       &\equiv(\{x, y\},\rho_{\{,\}}(x)v-\rho_{\{,\}}(y)u+[u,
       v])h\mmod{h^2}.
  \end{split}
\vspa
\end{equation*}
Then $(A\oplus V,\llbracket,\rrbracket,\odot)$ is the \qcl of the
associative deformation $(A_{h}\oplus V_{h},\odot_{h})$.
Therefore, $(A\oplus V,\llbracket,\rrbracket,\odot)$ is a Poisson
algebra and hence $(V,[,],\cdot,\rho_{\{,\}},\rho_{\circ}
 )$ is an $(A,\{,\},\circ)$-module Poisson algebra by Definition
 \mref{Pma}. The last statement holds since the two Poisson algebras $A\ltimes_{\rho_{\{,\}},\rho_{\circ}} V$ and $(A\oplus V,\llbracket,\rrbracket,\odot)$ coincide.
\end{proof}

Regarding a bimodule as a bimodule algebra where the multiplication is zero, we have the following consequence of Theorem \mref{Thmdefmoda}.
\vspb
\begin{cor}
    Let $(A,\circ)$ be a commutative associative algebra and $(V,\rho_{\circ})$ be an $A$-module.  Let $(A_{h},\circ_{h})$ be an associative deformation of $(A,\circ)$ and
    $(A,\{,\},\circ)$ be the \qcl Poisson algebra. Let $(V_{h},l_{h},r_{h})$ be an $A_{h}$-bimodule deformation of $(V,\rho_{\circ})$.
    Then the corresponding \qcl Poisson algebra $(V,\rho_{\{,\}},\rho_{\circ})$ is an $(A,\{,\},\circ)$-module. Moreover, the associative deformation
    $A_{h}\ltimes_{l_{h},r_{h}} V_{h}$ of the commutative associative algebra $A\ltimes_{\rho_{\circ}} V$
    is exactly the Poisson algebra $A\ltimes_{\rho_{\{,\}},\rho_{\circ}} V$.
\mlabel{cor-defmod}
\end{cor}

\vspb
 \subsection{Scalar deformations of $\mathcal{O}$-operators}
 \mlabel{ss:deformop}

We give a unified notion of $\calo$-operators on various algebras.

\begin{defi}
Let $\calp$ be a binary quadratic operad with binary operations
$\{\mu_i\}_{i\in I}$. Let $(A,\{\circ_i\}_{i\in I})$ and
$(V,\{\cdot_i\}_{i\in I})$ be $\calp$-algebras. Let $l_i,r_i:
A\rightarrow \textrm{End}_{\bfk}(V), i\in I,$ be linear maps such
that the quadruple $(V,\{\cdot\}_i,\{l_i\}_i,\{r_i\}_i)$ is an
$A$-bimodule $\calp$-algebra as given in
Definition~\mref{de:pbimod}. A linear map $T:V\rightarrow A$ is
called an \textbf{$\mathcal{O}$-operator of weight} $\lambda$
\textbf{ on the $\calp$-algebra} $(A,\circ)$ \textbf{associated to
the $A$-bimodule $\calp$-algebra} $(V,\cdot,l,r)$ if $T$ satisfies
\vspb
\begin{equation}\mlabel{ARRBO}
    T(u)\circ_i T(v)=T(l_i(T(u))v)+T(r_i(T(v))u)+\lambda T(u\cdot_i v),\;\; u,v\in V.
\vspa
\end{equation}
When an $A$-bimodule $(V,l_i,r_i)$ is regarded as an $A$-bimodule $\calp$-algebra equipped with the zero multiplications $\cdot_i$, the weight is irrelevant, and $T$ is called an {\bf $\mathcal
    O$-operator on the $\calp$-algebra $(A,\{\circ_i\}_{i\in I})$ associated to the $A$-bimodule $(V,\{l_i\}_i,\{r_i\}_i)$}.
\mlabel{de:oop}
\end{defi}

When $\calp$ is the operad of associative algebras, that is, when $(A,\circ)$ is an associative algebra and $(V,\cdot,l,r)$ is an $A$-bimodule algebra, we recover the notion of an $\mathcal{O}$-operator of weight $\lambda$ on $(A,\circ)$ associated to $(V,\cdot,l,r)$~\mcite{BGN4}.

\begin{rmk} For an associative algebra $(A,\circ)$, the regular representation gives an $A$-bimodule algebra $(A,\circ,L_\circ,R_\circ)$. Then an $\mathcal{O}$-operator of weight $\lambda$ on the associative algebra $(A,\circ)$ associated to this $A$-bimodule algebra gives the notion of a \textbf{Rota-Baxter operator} of the same weight, defined by
\vspb
  \begin{equation}\mlabel{ARBO}
    T(x)\circ T(y)=T(T(x)\circ y)+T(x\circ T(y))+\lambda  T(x\circ y),\;\; x,y\in A.
\vspa
\end{equation}
Thus an $\calo$-operator of weight $\lambda$ is also
called a \name{relative Rota-Baxter operator of weight $\lambda$}.
\end{rmk}

For later reference of notations, we recall the details when $\calp$ is the operad of Poisson algebras.
\begin{defi}\rm{\mcite{NB}}
  Let $(P,\{,\},\circ)$ be a Poisson algebra, $(Q,[,],\cdot,\rho_{\{,\}},\rho_{\circ})$ be a $P$-module Poisson algebra and $\lambda\in \bfk$. A linear map $T:Q\rightarrow P$ is called an \textbf{$\mathcal{O}$-operator} \textbf{of weight} $\lambda$ \textbf{ on %the Poisson algebra}
  $(P,\{,\},\circ)$} \textbf{associated to %the $P$-module Poisson algebra
  $(Q,[,],\cdot,\rho_{\{,\}},\rho_{\circ})$} if $T$ satisfies
\vspb
  \begin{eqnarray}
    \{T(u),T(v)\} &=& T(\rho_{\{,\}}(T(u))v)-T(\rho_{\{,\}}(T(v))u)+\lambda T([u,v]) \mlabel{rpl}\\
    T(u)\circ T(v) &=& T(\rho_{\circ}(T(u))v)+T(\rho_{\circ}(T(v))u)+\lambda T(u\cdot v), \;\;\quad  u, v\in Q.\mlabel{rpa}
  \end{eqnarray}
  In particular, when the operations $[,]$ and $\cdot$ on $Q$ are zero, that
  is, $(Q,\rho_{\{,\}},\rho_{\circ})$ is a $P$-module,
  % the weight $\lambda$ can be arbitrary and
  $T$ is simply called an {\bf $\mathcal
    O$-operator on  %the Poisson algebra}
    $(P,\{,\},\circ)$} \textbf{associated to
    $(Q,\rho_{\{,\}},\rho_{\circ})$}.
\vspb
\end{defi}

\begin{rmk}
Similar to the case of associative algebras, a Rota-Baxter operator $T$ of weight $\lambda$ on a Poisson algebra $(P,\{,\},\circ)$ is defined by the following equations:
\vspb
\begin{eqnarray}
    \{T(x),T(y)\} &=& T(\{T(x),y\})+T(\{x,T(y)\})+\lambda  T(\{x,y\}), \label{rbl} \\
  T(x)\circ T(y) &=& T(T(x)\circ y)+T(x\circ T(y))+\lambda  T(x\circ y),\quad x,y\in P.
\vspb
\end{eqnarray}
\end{rmk}

Let $(A,\circ)$ be an associative algebra and $(V,\cdot,l,r)$ be an $A$-bimodule algebra.
Let $(A_{h},\circ_{h})$ be an associative deformation of $(A,\circ)$ and $(V_{h},\cdot_{h},l_{h},r_{h})$ be an $A_{h}$-bimodule algebra deformation of $(V,\cdot,l,r)$. For $T\in\mathrm{Hom}_{\bfk}(V,A)$, define the {\bf scalar deformation} of $T$ to be the $\bf K$-linear extension
\vspb
\begin{equation}\label{eq:TK}
T_\bfK: V_h\to A_h, \quad T_{\bf K}\Big(\sum_{i=0}^{\infty}v_{i}~h^{i}\Big):=\sum_{i=0}^{\infty}T(v_{i})~h^{i}, \quad \sum_{i=0}^\infty v_i~h^i\in V_h.
\vspb
\end{equation}
Then  $T_{\mathbf{K}}$ is an $\mathcal O$-operator of weight $\lambda$ on $(A_{h},\circ_{h})$ associated to  $(V_{h},\cdot_{h},l_{h},r_{h})$ means that
\vspb
\begin{equation}\label{eq:O-def} T_{\mathbf{K}}(u)\circ_{h} T_{\mathbf{K}}(v)=T_{\mathbf{K}}(l_{h}(T_{\mathbf{K}}(u))v)+T_{\mathbf{K}}(r_{h}(T_{\mathbf{K}}(v))u)+\lambda T_{\mathbf{K}}(u\cdot_{h} v), \;\; \quad  u, v\in V.
\vspb
\end{equation}
For $x,y\in A$ and $u\in V$, consider the expansions
\vspc
\[  x\circ_{h}y=\sum_{s=0}^{\infty}(x\circ_{s}y)h^{s},\;\;      l_{h}(x)u=\sum_{s=0}^{\infty}l_{s}(x)u~h^{s}\quad\textrm{and}\quad r_{h}(x)u=\sum_{s=0}^{\infty}r_{s}(x)u~h^{s},
\vspb
\]
with $x\circ_s y\in A$ and $l_{s}(x),r_{s}(x)\in
\mathrm{End}_{\bf k}(V)$. Then \eqref{eq:O-def} holds if and only if
\vspb
\begin{equation}\mlabel{infirrbo}
   T(u)\circ_{s} T(v)=T(l_{s}(T(u))v)+T(r_{s}(T(v))u)+\lambda T(u\cdot_{s} v),\;\; u,v\in V, s\geq 0.
\vspb
\end{equation}
When $(V,\cdot,l,r)$ is the $A$-bimodule algebra $(A,\circ,L_\circ,R_\circ)$, we replace $\cdot_{s}$ by $\circ_{s}$ and (\mref{infirrbo}) becomes
\vspb
\begin{equation}\mlabel{extrb}
   T(x)\circ_{s}T(y)=T(T(x)\circ_{s}y)+T(x\circ_{s}T(y))+\lambda T(x\circ_{s}y),\;\; x,y\in A, s\geq 0.
\vspb
\end{equation}

\begin{thm}
Let $(A,\circ)$ be a commutative associative algebra, and
let  $(A_{h},\circ_{h})$ be its associative deformation whose \qcl Poisson algebra is $(A,\{,\},\circ)$.
Also let $(V,\cdot,\rho_{\circ})$ be an $A$-module algebra, and let $(V_{h},\cdot_{h},l_{h},r_{h})$ be its $A_{h}$-bimodule algebra deformation whose \qcl  $(A,\{,\},\circ)$-module Poisson algebra is $(V,[,],\cdot,\rho_{\{,\}},\rho_{\circ})$. Let
$T\in\mathrm{Hom}_{\bfk}(V,A)$ be an $\mathcal{O}$-operator of weight $\lambda$ on $(A,\circ)$ associated to
$(V,\cdot,\rho_{\circ})$. If its scalar deformation $T_{\mathbf{K}}$ defined by (\ref{eq:TK}) is an $\calo$-operator of weight $\lambda$ on
$(A_{h},\circ_{h})$ associated to $(V_{h},\cdot_{h},l_{h},r_{h})$,
then $T$ is an $\mathcal{O}$-operator of weight $\lambda$ on the
Poisson algebra $(A,\{,\},\circ)$ associated to the
$(A,\{,\},\circ)$-module Poisson algebra
$(V,[,],\cdot,\rho_{\{,\}},\rho_{\circ})$.
\mlabel{invrb}
\end{thm}
\begin{proof}
Since $T$ is an $\mathcal{O}$-operator of weight $\lambda$ on $(A,\circ)$ associated to
$(V,\cdot,\rho_{\circ})$, it is direct to see that
(\mref{rpa}) is satisfied. Applying (\mref{infirrbo}) and exchanging $u$ and $v$, we obtain
\vspb
  \begin{equation}\mlabel{huan}
    T(v)\circ_{s} T(u)=T(l_{s}(T(v))u)+T(r_{s}(T(u))v)+\lambda T(v\cdot_{s}
    u),\;\; u,v\in V, s\geq 0.
\vspb
  \end{equation}
Taking $s=1$ and subtracting (\mref{huan}) from (\mref{infirrbo}),
we deduce that the equation
  \[T(u)\circ_{1} T(v)-T(v)\circ_{1} T(u)=T((l_{1}-r_{1})(T(u))v)-T((l_{1}-r_{1})(T(v))u)+\lambda T(u\cdot_{1} v-u\cdot_{1} v),\;\; u,v\in V,\]
holds if and only if (\mref{rpl}) holds. Thus $T$ is an
$\mathcal{O}$-operator of weight $\lambda$ on $(A,\{,\},\circ)$
associated to $(V,[,]$, $\cdot$, $\rho_{\{,\}}$, $\rho_{\circ})$.
\vspc
\end{proof}
\vspb

\section{Tridendriform deformations and scalar deformations of $\calo$-operators}
\label{S2}
In this section, we study tridendriform deformations of commutative tridendriform algebras in
terms of bimodule algebra deformations. This provides an application of Theorem~\ref{invrb} where the
deformations and \qcls with $\mathcal O$-operators can be constructed explicitly. We also give constructions of bimodule algebra deformations and tridendriform
deformations via derivations

\vspa

\subsection{Deformations of tridendriform algebras and the corresponding \qcls}
\vspb
\begin{defi}
\mcite{LR} A \textbf{tridendriform algebra} is a quadruple
$(A,\succ,\prec,\cdot)$ consisting of a vector space $A$ with three binary operations $\succ$, $\prec$ and $\cdot$ such that
\vspb
 \begin{align}
    (x\prec y)\prec z =&x\prec (y\circ z),
   &(x\succ y)\prec z =& x\succ(y\prec z),
   &(x\circ y)\succ z=& x\succ(y\succ z), \mlabel{Tri1}\\
    (x\succ y)\cdot z =& x\succ(y\cdot z),
    &(x\prec y)\cdot z =& x\cdot(y\succ z),
    &(x\cdot y)\prec z =& x\cdot(y\prec z), \\
& &  (x\cdot y)\cdot z =& x\cdot(y\cdot z),\;\; x,y,z\in A,\mlabel{Tri7}
\vspb
\end{align}
\vspc
where
\begin{equation}\label{eq:sum}
x\circ y:=x\prec y + x\succ y + x\cdot y,\;\; x,y\in A.
\end{equation}
If in addition, $x\succ y=y\prec x$ and $x\cdot y=y\cdot x$ for
$x,y\in A$, then we call the tridendriform algebra \textbf{commutative} and denote it by $(A,\succ,\cdot)$.
\end{defi}

\begin{rmk}
 We recall that a \textbf{dendriform algebra} is a triple $(A,\succ,\prec)$ consisting of a vector space $A$ with binary operations $\succ$ and $\prec$ such that the
  identities in (\mref{Tri1}) are satisfied for $x,y,z\in A$, where $x\circ y:=x\prec y + x\succ y$. Moreover, if $x\succ y=y\prec x$ for $x,y\in A$, we call it % $(A,\succ,\prec)$
  a \textbf{Zinbiel algebra}
  and denote it by $(A,\succ)$. So we may regard a dendriform algebra $(A,\succ,\prec)$ as a tridendriform
  algebra $(A,\succ,\prec,\cdot)$ in which the multiplication $\cdot$ is zero.
\end{rmk}
There is a bimodule algebra characterization of a tridendriform algebra.
\begin{pro}{\rm{\mcite{BGN4}}}\label{pro-trieqv}
Let $\succ$, $\prec$ and $\cdot$ be binary operations on a
vector space $A$. Define a binary operation $\circ$ on $A$ by
\meqref{eq:sum}.
Then $(A,\succ,\prec,\cdot)$ is a tridendriform  algebra if and
only if $(A,\circ)$ is an associative algebra and
$(A,\cdot,L_{\succ},R_{\prec})$ is an $(A,\circ)$-bimodule
algebra. In this case, $\mathrm{id}_{A}$ is an
$\mathcal{O}$-operator of weight $1$ on $(A,\circ)$ associated to $(A,\cdot,L_{\succ},R_{\prec})$. Moreover, $(A,\succ,\cdot)$ is a commutative tridendriform algebra if and
only if $(A,\circ)$ is a commutative  associative algebra and
 $(A,\cdot,L_{\succ})$ is an $(A,\circ)$-module
algebra.
\mlabel{pro-ct}
\end{pro}

The notion of post-Lie algebras arose from the study of
operads~\mcite{V} with applications to numerical
analysis~\mcite{LM} and specializes to {\bf pre-Lie algebras} ~\cite{Bai2} when the Lie algebras are abelian.

\begin{defi}
   \rm{\mcite{V}} A \textbf{post-Lie algebra} is a triple $(\mathfrak{g},[,],\triangleright)$ consisting of a Lie algebra $(\mathfrak{g},[,])$ and a binary operation $\triangleright$ such that
   \vspb
\begin{eqnarray}
  x\triangleright[y,z] &=& [x\triangleright y,z]+[y,x\triangleright z],\mlabel{PostL1} \\
  \ [x,y]\triangleright z &=& x\triangleright(y\triangleright z)-(x\triangleright y)\triangleright z-y\triangleright(x\triangleright z)+(y\triangleright x)\triangleright z,\;\; x,y,z\in \mathfrak g. \mlabel{PostL2}
\end{eqnarray}
\vspb
 \end{defi}
\vspd

 \begin{defi}
  \rm{\mcite{BBGN}}  A \textbf{ post-Poisson algebra} is a triple $(P,[,],\triangleright,\succ,\cdot)$ consisting of a post-Lie algebra $(P,[,], \triangleright)$ and a commutative tridendriform algebra $(P, \succ,\cdot)$
  such that (\mref{PostP1}) and the following equations hold.
 \vspb
  \begin{align}
    [x,y\succ z] =& y\succ[x,z]-z\cdot(y\triangleright x),
      &x\triangleright(y\cdot z) =& (x\triangleright y)\cdot z+y\cdot(x\triangleright z),\\
    (x\circ y)\triangleright z =& x\succ(y\triangleright z)+y\triangleright(z\prec x),
      &x\triangleright(y\succ z) =& y\succ(x\triangleright z)+\{x,y\}\succ z,\mlabel{PostP5}
  \end{align}
  for $x,y,z\in P$, where
\begin{equation}\label{eq:summ}
x\circ y:=x\succ y+y\succ x+x\cdot y,\;\; \{x,y\}:=x\triangleright
y-y\triangleright x+[x,y],\;\;
  x,y\in P.
\end{equation}
 \end{defi}
\begin{rmk}
Introduced by Aguiar~\mcite{A2}, a \textbf{pre-Poisson algebra}
is a triple $(A,\triangleright,\succ)$ consisting of a Zinbiel algebra
$(A,\succ)$ and a pre-Lie algebra $(A,\triangleright)$
 such that the
   identities in (\mref{PostP5}) are satisfied for $x,y,z\in A$, where $\{x,y\}:=x\triangleright y-y\triangleright x$ and $x\circ y:=x\succ y+y\succ x$.
   So a pre-Poisson algebra is a post-Poisson algebra where the multiplications $\cdot$ and $[,]$ are zero.
 \end{rmk}

Similar to Proposition\,\ref{pro-ct}, the following statement holds.

\begin{pro}{\rm{\mcite{NB}}}
Let $\triangleright,[,],\succ$ and $\cdot$ be four binary
operations on a vector space $P$. Define  binary operations
$\circ$ and $\{,\}$ on $P$ by \meqref{eq:summ}.
Then $(P,\triangleright,[,],\succ,\cdot)$ is a post-Poisson
algebra if and only if $(P,\{,\},\circ)$ is a Poisson algebra and
$(P,[,],\cdot,L_{\triangleright},L_{\succ})$ is a
$(P,\{,\},\circ)$-module Poisson algebra. Moreover, in this case,
$\mathrm{id}_{P}$ is an $\mathcal{O}$-operator of weight 1 on
$(P,\{,\},\circ)$ associated to
$(P,[,],\cdot,L_{\triangleright},L_{\succ})$.
\mlabel{pro-cp}
\vspc
\end{pro}

\begin{defi}
  A \textbf{topologically free tridendriform algebra} is a quadruple $(A_{h},\succ_{h},\prec_{h},\cdot_{h})$ consisting of a topologically free $\mathbf{K}$-module $A_{h}$,  $\mathbf{K}$-bilinear operations
  $\succ_{h},\prec_{h}$ and $\cdot_{h}$ such that
   $(A_h,\succ_h$, $\prec_h$, $\cdot_h)$ is a tridendriform algebra.

A \textbf{deformation of a tridendriform algebra} $(A,\succ,\prec,\cdot)$ is a topologically free tridendriform algebra $(A_{h},\succ_{h},\prec_{h},\cdot_{h})$ such that
\begin{equation} \mlabel{eq:dt2}
 x\succ_{h} y\equiv x\succ y\mmod{h}, \;x\prec_{h}y \equiv x\prec y\mmod{h},\; x\cdot_{h}y \equiv x\cdot y\mmod{h},\;x,y\in A.
\end{equation}
We also call  $(A_{h},\succ_{h},\prec_{h},\cdot_{h})$ a \textbf{tridendriform algebra deformation} or simply a \textbf{tridendriform deformation} of $(A,\succ,\prec,\cdot)$.
\end{defi}

Let $(A_{h},\succ_{h},\prec_{h},\cdot_{h})$ be a tridendriform deformation of $(A\succ,\prec,\cdot)$. For each $\ast_{h}\in \{\succ_{h},\prec_{h},\cdot_{h},\circ_{h}\}$, we can write
\vspc
\begin{equation}\mlabel{expan}
 x\ast_{h}y=\sum_{s=0}^{\infty}(x\ast_{s}y)h^{s},\;\;
 x,y\in A,
\end{equation}
where $x\ast_{s}y\in A$. By \meqref{eq:dt2}, we have
$x\ast_{0}y=x\ast y$ for   $x,y\in A$.
\begin{defi}
   Let $(A_{h},\succ_{h},\prec_{h},\cdot_{h})$ be a tridendriform deformation of a commutative tridendriform algebra $(A,\succ,\cdot)$. Let $[,]$ and $\triangleright$ be the binary operations on $A$ such that
   \begin{equation}
     \ [x,y] \equiv \frac{x\cdot_{h}y-y\cdot_{h}x}{h}\mmod{h},
     \quad
     x\triangleright y \equiv \frac{x\succ_{h}y-y\prec_{h}x}{h}\mmod{h},
     \quad  x, y\in A. \mlabel{brak}
   \end{equation}
Then we call $(A,[,],\triangleright,\succ,\cdot)$ the \textbf{quasiclassical limit (\qcl) of the tridendriform deformation} $(A_{h},\succ_{h},\prec_{h},\cdot_{h})$.
\end{defi}

\begin{rmk}
By setting $\cdot=\cdot_h=0$ in the above definitions, we define the notions of a {\bf topologically free dendriform algebra},
a {\bf dendriform deformation} and the {\bf quasiclassical limit of a dendriform deformation}.
\end{rmk}

For the \qcl $(A,[,],\triangleright,\succ,\cdot)$,
first we note that, as  the \qcl of the
associative deformation $(A_h,\cdot_{h})$, the triple
$(A,[,],\cdot)$ is a Poisson algebra. Define a binary operation on $A_h$ by
\vspb
\begin{equation}\label{eq:sumh}
    x_{h}\circ_{h}y_{h}:=x_{h}\succ_{h} y_{h}+x_{h}\prec_{h}y_{h} + x_{h}\cdot_{h}y_{h},\quad x_{h},y_{h}\in A_{h}.
\vspa
\end{equation}
Next, with the binary
operation $\{,\}$ on $A$ defined by
\vspb
\begin{equation}
   \{x,y\}\equiv \frac{x\circ_{h}y-y\circ_{h}x}{h}\ (\textrm{mod}~h), \quad  x, y\in A,
\vspa
 \end{equation}
and $\circ$ defined by \meqref{eq:sum}, the triple $(A,\{,\},\circ)$ is the \qcl of the
associative deformation $(A_h,\circ_{h})$ and hence is also a
Poisson algebra. Let $x,y\in A$. Using the expansion in
\meqref{expan},
 we have
\vspb
$$  x\triangleright y = x\succ_{1}y-y\prec_{1}x, \
  \ [x,y] = x\cdot_{1}y-y\cdot_{1}x, \ \
  \{x,y\} = x\circ_{1}y-y\circ_{1}x.
\vspa
$$
  It is direct to see that
\vspc
\begin{equation}\label{eq-br}
    \{x,y\}=x\triangleright y-y\triangleright x+[x,y].
\vspa
\end{equation}
Hence the operation $\{,\}$ also splits into three operations.

\begin{thm}{\rm{\mcite{OPV}}}
Let $(A_{h},\succ_{h},\!\prec_{h},\cdot_{h})$ be a tridendriform deformation of a commutative tridendriform algebra $(A,\succ,\!\cdot)$. Then the corresponding \qcl $(A,[,],\triangleright,\succ,\!\cdot)$ is a post-Poisson algebra.
\mlabel{mainthm}
\vspd
\end{thm}

By taking $\cdot$ and $\cdot_{h}$ in Theorem \ref{mainthm} to be
zero multiplications, we obtain

\begin{cor}{\rm\mcite{A2}}
  Let $(A_{h},\succ_{h},\prec_{h})$ be a dendriform deformation of a Zinbiel algebra $(A,\succ)$. Then the corresponding \qcl $(A,\triangleright,\succ)$ is a pre-Poisson algebra.
\mlabel{cor-defdendri}
\end{cor}

We give the following example of tridendriform deformations
constructed from an arbitrary associative deformation.

\begin{ex}\mlabel{exps}
Let $(A^{n},\cdot)$ be the direct product of $n$ copies of an associative algebra $(A,\diamond)$, equipped with the product
\vspb
\begin{equation}\mlabel{cdot}
a\cdot b=(a_{1}\diamond b_{1},\ldots, a_{i}\diamond b_{i},\ldots, a_{n}\diamond b_{n}), \quad a=(a_{1},\ldots,a_{n}), b=(b_{1},\ldots,b_{n})\in A^{n}.
\vspb
\end{equation}
Set
\vspc
\begin{equation}
   a\succ b:=
    (0,a_{1}\diamond b_{2},\ldots, \sum_{j=1}^{n-1}a_{j}\diamond b_{n}),  \quad
   a\prec b := (0,a_{2}\diamond b_{1},\ldots,  \sum_{j=1}^{n-1}a_{n}\diamond b_{j}).
   \label{succ}
\vspb
\end{equation}
Then $(A^n,\succ,\prec,\cdot)$ is a tridendriform algebra. If $(A_{h},\diamond_{h})$ is an associative deformation of $(A,\diamond)$, denote
\[a_{i}\diamond_{h}b_{j}=\sum_{s=0}^{\infty}(a_{i}\diamond_{s}b_{j})h^{s}, 1\leq i,j\leq n.\]
Then the following equations define a tridendriform deformation $(A^{n}_{h},\succ_{h},\prec_{h},\cdot_{h})$ of $(A^n,\succ,\prec,\cdot)$:
\vspa
    \begin{eqnarray}
   & a\succ_{h}b := \sum\limits_{s=0}^{\infty}(0,a_{1}\diamond_{s} b_{2},\ldots, \sum\limits_{j=1}^{n-1}a_{j}\diamond_{s} b_{n})h^{s}, \ \ 
    a\prec_{h}b := \sum\limits_{s=0}^{\infty}(0,a_{2}\diamond_{s} b_{1},\ldots, \sum\limits_{j=1}^{n-1}a_{n}\diamond_{s} b_{j})h^{s},&\\
    &a\cdot_{h}b := \sum\limits_{s=0}^{\infty}(a_{1}\diamond_{s}b_{1},..., a_{i}\diamond_{s}b_{i},..., a_{n}\diamond_s b_{n})h^{s}.&
    \end{eqnarray}
\vspa
    If $(A,\diamond)$ is commutative and $(A,[,],\diamond)$ is the \qcl of $(A_{h},\diamond_{h})$, then the \qcl $(A^{n},[,]_n,\triangleright,\succ,\cdot)$ of $(A^{n}_{h},\succ_{h},\prec_{h},\cdot_{h})$ is determined by (\mref{cdot}), (\mref{succ}) and the equations
\vspb
$$
[a,b]_n = ([a_{1}, b_{1}], ..., [a_{i}, b_{i}], ..., [a_{n}, b_{n}]),  \quad
    a\triangleright b=
    (0,[a_{1}, b_{2}], ..., \sum_{j=1}^{i-1}[a_{j},b_{i}], ..., \sum_{j=1}^{n-1}[a_{j}, b_{n}]).
\vspc
$$
\end{ex}

We recall the following results on derived structures from $
\calo$-operators.
\begin{pro}{\rm{\mcite{BGN4,NB}}}
Let $(A,\circ)$ $($resp. $(P,\{,\},\circ)$$)$ be an associative algebra $($resp. a Poisson algebra$)$ and $(V,\cdot,l,r)$ $($resp. $(Q,[,],\cdot,\rho_{\{,\}},\rho_{\circ}))$ be an $A$-bimodule algebra $($resp. a $P$-module Poisson algebra$)$. Let
$T:V\rightarrow A$ $($resp. $T:Q\rightarrow P$$)$ be an $\mathcal{O}$-operator of weight $\lambda$
on $(A,\circ)$ associated to $(V,\cdot,l,r)$ $($resp. on $(P,\{,\},\circ)$ associated to  $(Q,[,],\cdot,\rho_{\{,\}},\rho_{\circ}))$. Set
\[u\succ v:=l(T(u))v , \quad u\prec v:=r(T(v))u ,\quad u,v\in V.\]
$$ (\text{resp. } u\triangleright v:=\rho_{\{,\}}(T(u))v,\quad u\succ v:=\rho_{\circ}(T(u))v,\quad u,v\in Q).$$
Then $(V,\succ,\prec,\lambda\cdot)$ is a tridendriform algebra $($resp. $(Q, \lambda[,],\triangleright,\succ,\lambda\cdot)$ is a post-Poisson algebra$)$.
\mlabel{THTRT} \mlabel{PPRT}
\end{pro}

Let $(A,\circ)$ be a commutative associative algebra and $(V,\cdot,\rho_{\circ})$ be an $A$-module
 algebra.  Let $(A_{h},\circ_{h})$ be an associative deformation of $(A,\circ)$ and $(V_{h},\cdot_{h},l_{h},r_{h})$ be an $A_{h}$-bimodule algebra deformation of $(V,\cdot,\rho_\circ)$.
Let $T$ be an
$\mathcal{O}$-operator of weight $\lambda$
on $(A,\circ)$ associated to $(V,\cdot,\rho_{\circ})$.
By Proposition \mref{THTRT}, $(V,\succ,\lambda\cdot)$
is a commutative tridendriform algebra, where
$u\succ v:=\rho_{\circ}(T(u))v, u,v\in V$.
Suppose that the operator
$T_{\mathbf{K}}$ defined by (\ref{eq:TK}) is an $\calo$-operator of weight
$\lambda$ on $(A_{h},\circ_{h})$ associated to
$(V_{h},\cdot_{h},l_{h},r_{h})$.
Set
\vspb
 \[u_{h}\succ_{h} v_{h}:=l_{h}(T_{\bf K}(u_{h}))v_{h},\quad u_{h}\prec_{h}v_{h}:=r_{h}(T_{\bf K}(v_{h}))u_{h},\;\; u_{h},v_{h}\in V_{h}.\]
    By Proposition
    \mref{THTRT} again, $(V_{h},\succ_{h},\prec_{h},\lambda\cdot_{h})$ is a
    tridendriform algebra. Since $(V_{h},\cdot_{h})$ is an associative deformation of $(V,\cdot)$, we deduce that $\lambda x\cdot_{h}y\equiv \lambda x\cdot y\mmod{h}$. Because $(V_{h},\cdot_{h},l_{h},r_{h})$ is an $A_{h}$-bimodule algebra deformation of $(V,\cdot,\rho_\circ)$, by Proposition \ref{edm} we deduce that
\vspb
    \begin{eqnarray*}
        u\succ_{h} v&=&l_{h}(T_{\bf K}(u))v\equiv \rho_{\circ}(T(u))v\mmod{h}  \equiv u\succ v\mmod{h},\\
        u\prec_{h} v&=&r_{h}(T_{\bf K}(v))u\equiv \rho_{\circ}(T(v))u\mmod{h}\equiv v\succ u\mmod{h}\equiv u\prec v\mmod{h},  u,v\in V.
    \end{eqnarray*}
    Thus  we  conclude that $(V_{h},\succ_{h},\prec_{h},\lambda\cdot_{h})$ is a tridendriform deformation of
    $(V,\succ,\lambda\cdot)$. By Theorem~\ref{mainthm}, let
$(V,[,]',\triangleright',\succ,\lambda\cdot)$  be its
\qcl which is a post-Poisson algebra, where
\vspb
$$[u,v]':=\lambda (u\cdot_1 v-v\cdot_1
u),\;\; u\triangleright' v:=u\succ_{1}v-v\prec_{1}u,\;\;
u,v\in V.
\vspb
$$
Here we use the expansion in \meqref{expan}.

On the other hand, suppose that $(A,\{,\},\circ)$, $(V,[,],\cdot)$
and $(V,[,],\cdot,\rho_{\{,\}},\rho_{\circ})$ are the
\qcls of  $(A_{h},\circ_{h})$, $(V_{h},\cdot_{h})$
and $(V_{h},\cdot_{h},l_{h},r_{h})$ respectively. By
Theorems~\ref{Thmdefmoda} and~\mref{invrb}, $T$ is an
$\mathcal{O}$-operator of weight $\lambda$ on the Poisson algebra
$(A,\{,\},\circ)$ associated to $(V,[,],\cdot,\rho_{\{,\}},\rho_{\circ})$. By
Proposition~\ref{PPRT}, there is a post-Poisson algebra $(V,
\lambda[,],\triangleright,\succ,\lambda\cdot)$. Note that
\vspb
$$
[u,v]=u\cdot_1v-v\cdot_1 v, \quad u\succ_{1}v-v\prec_{1}u=l_{1}(T(u))v-r_{1}(T(u))v=\rho_{\{,\}}(T(u))v,\;\;
u,v\in V.
\vspb
$$
Hence $[u,v]'=\lambda [u,v]$ and
$u\triangleright' v=u\triangleright v$ for $u,v\in V$. That
is, the two post-Poisson algebras
$(V,[,]',\triangleright',\succ,\lambda\cdot)$ and $(V,
\lambda[,],\triangleright,\succ,\lambda\cdot)$ coincide. In summary, we obtain the following conclusion.

\begin{pro}\label{pro-diagotri}
With the above notations, starting from an $\calo$-operator $T$, the topologically free tridendriform
algebra $(V_{h}$, $\succ_{h}$, $\prec_{h}$, $\lambda\cdot_{h})$
induced by the scalar deformation $\calo$-operator $T_{\mathbf{K}}$ of the $\calo$-operator $T$
coincides with the tridendriform
deformation of the commutative tridendriform algebra $(V,\succ,\lambda\cdot)$ induced by the $\calo$-operator $T$.
Moreover, the \qcl post-Poisson algebra of the deformation tridendriform algebra $(V_{h}$, $\succ_{h}$, $\prec_{h}$, $\lambda\cdot_{h})$ coincides with the post-Poisson
algebra induced by the $\calo$-operator $T$ on the \qcl Poisson algebra $(A,\{,\},\circ)$.
In other words, the following  diagram commutes.
\vspa
\begin{displaymath}
  \xymatrixcolsep{6pc}\xymatrixrowsep{3pc}
  \xymatrix{
  *\txt{\small $\mathcal{O}$-operator $T$ on \\
    \small $(A,\circ)$ associated to\\
    \small $(V,\cdot,\rho_{\circ})$}
   \ar[d]|-{
    \txt{\tiny Proposition \mref{THTRT}}
    }\ar[r]^{\txt{\tiny deformation}} & *\txt{\small $\mathcal{O}$-operator $T_\mathbf{K}$ on \\ \small $(A_{h},\circ_{h})$ associated to \\\small $(V_{h},\cdot_{h},l_{h},r_{h})$ }\ar[r]^{\txt{ \tiny QCL}}_{\txt{\tiny Theorem~\mref{invrb}
        }}\ar[d]|-{\txt{\tiny Proposition \mref{THTRT}}}
   & *\txt{\small $\mathcal{O}$-operator $T$ on \\\small $(A,\{,\},\circ)$ associated to\\ \small $(V,[,],\cdot,\rho_{\{,\}},\rho_{\circ})$}\ar[d]|-{\txt{\tiny Proposition \mref{PPRT} }}\\
   *\txt{\small commutative \\ \small tridendriform algebra\\ \small $(V,\succ,\lambda\cdot)$}
   \ar[r]^{\txt{\tiny deformation}}& *\txt{\small topologically free\\ \small tridendriform algebra\\$(V_{h},\succ_{h},\prec_{h},\lambda\cdot_{h})$}\ar[r]^{\txt{\tiny \qcl}}_{\txt{\tiny Theorem \mref{mainthm}}}
   & *\txt{\small post-Poisson\\ \small  algebra\\ \small $(V,\lambda [,],\triangleright,\succ,\lambda\cdot)$.
}}
\vspd
 \end{displaymath}
 \mlabel{diag-r1}
\vspb
 \end{pro}
Conversely, from a tridendriform deformation of a commutative tridendriform algebra, we obtain the same $\calo$-operator on a Poisson algebra associated to its module Poisson algebra by two ways as follows.

 Let  $(A,\succ,\cdot)$ be a commutative tridendriform algebra, $(A_{h},\succ_{h},\prec_{h},\cdot_{h})$ be a tridendriform deformation and $(A,[,],\triangleright,\succ,\cdot)$ be the \qcl post-Poisson algebra. Define binary operations $\circ$ and $\{,\}$ on $A$ by
(\ref{eq:summ}). Define a binary operation $\circ_{h}$ on $A_h$ by
(\ref{eq:sumh}). By Propositions~\ref{pro-ct} and~\ref{pro-cp} respectively, we obtain
the following data.
    \begin{itemize}
    \item a commutative associative algebra $(A,\circ)$, an $(A,\circ)$-module algebra $(A,\cdot,L_{\succ})$ and an
    $\mathcal{O}$-operator $\mathrm{id}_{A}$ on $(A,\circ)$ of weight $1$ associated to $(A,\cdot,L_{\succ})$;
    \item a topologically free associative algebra $(A_{h},\circ_{h})$, a topologically free $(A_{h},\circ_{h})$-bimodule algebra $(A_{h},\cdot
    _{h},L_{\succ_{h}},R_{\prec_{h}})$ and an $\calo$-operator $\mathrm{id}_{A_{h}}$ of weight $1$ on $(A_{h},\circ_{h})$ associated to\\ $(A_{h},\cdot
    _{h},L_{\succ_{h}},R_{\prec_{h}})$;
    \item a Poisson algebra $(A,\{,\},\circ)$,
    an $(A,\{,\},\circ)$-module Poisson algebra
    $(A,[,],\cdot,L_{\triangleright},L_{\succ})$ and an $\mathcal{O}$-operator $\mathrm{id}_{A}$ of weight $1$ on $(A,\{,\},\circ)$ associated to $(A,[,],\cdot,L_{\triangleright},L_{\succ})$.
\end{itemize}

\begin{thm}\label{pro-diaid}
Keeping the notations and assumptions as above, we have the following results.
    \begin{enumerate}
        \item  The topologically free $(A_{h},\circ_{h})$-bimodule algebra $(A_{h},\cdot
        _{h},L_{\succ_{h}},R_{\prec_{h}})$ is an $A_h$-bimodule deformation of the $(A,\circ)$-module algebra $(A,\cdot,L_{\succ})$ and the \qcl is the
        $(A,\{,\},\circ)$-module Poisson algebra
        $(A,[,],\cdot,L_{\triangleright},L_{\succ})$.\label{pro-diaid1}
        \item Starting from the commutative tridendriform algebra $(A,\succ,\cdot)$, the $\calo$-operator $\mathrm{id}_{A_{h}}$ on $(A_{h},\circ_{h})$ associated to $(A_{h},\cdot
        _{h},L_{\succ_{h}},R_{\prec_{h}})$ induced by the tridendriform deformation $(A_{h},\succ_{h}$, $\prec_{h},\cdot_{h})$ of $(A,\succ,\cdot)$ coincides with
        the scalar deformation $(\id_A)_{\bf K}$ of the $\mathcal{O}$-operator $\mathrm{id}_{A}$ on $(A,\circ)$ associated to $(A,\cdot,L_{\succ})$ induced by
        $(A,\succ,\cdot)$, that is, $\mathrm{id}_{A_{h}}=(\id_A)_{\bf
        K}$.
 \label{pro-diaid2}
            \item Starting from the topologically tridendriform algebra $(A_{h},\succ_{h},\prec_{h},\cdot_{h})$, the $\mathcal{O}$-operator $\mathrm{id}_{A}$ on $(A,\{,\},\circ)$ associated to $(A,[,],\cdot,L_{\triangleright},L_{\succ})$ induced by the \qcl post-Poisson algebra $(A,[,],\triangleright,\succ,\cdot)$
            of $(A_{h},\succ_{h},\prec_{h},\cdot_{h})$ coincides with the $\mathcal{O}$-operator induced by $\mathrm{id}_{A_{h}}=(\id_A)_{\bf
        K}$ via Theorem~\mref{invrb}.\label{pro-diaid3}
    \end{enumerate}
     Then we have the following commutative diagram.
\vspb
     {\small
            \begin{displaymath}
            \xymatrixcolsep{6pc}\xymatrixrowsep{3pc}  \xymatrix{
                *\txt{\small commutative \\ \small tridendriform algebra\\ \small $(A,\succ,\cdot)$}
                \ar[r]^{\txt{\tiny deformation}}\ar[d]|-{\txt{\tiny Proposition~\mref{pro-ct}}} & *\txt{\small topologically free\\ \small tridendriform algebra\\ \small $(A_{h},\succ_{h},\prec_{h},\cdot_{h})$}\ar[r]^{\txt{\tiny \qcl}}_{\txt{\tiny Theorem~\mref{mainthm}}} \ar[d]|-{\txt{\tiny Proposition~\mref{pro-ct}}}
                & \txt{\small post-Poisson\\ \small algebra\\ \small $(A,[,],\triangleright,\succ,\cdot)$}\ar[d]|-{\txt{\tiny Proposition~\mref{pro-cp}}}\\
                *\txt{\small $\mathcal{O}$-operator $\mathrm{id}_{A}$ on \\ \small $(A,\circ)$ associated to\\ \small $(A,\cdot,L_{\succ})$}
                \ar[r]^{\txt{\tiny deformation}} & *\txt{\small  $\mathcal{O}$-operator $\mathrm{id}_{A_{h}}=(\id_A)_{\bf K}$ \\ \small on $(A_{h},\circ_{h})$ associated to \\ \small  $(A_{h},\cdot
                    _{h},L_{\succ_{h}},R_{\prec_{h}})$}\ar[r]^{\txt{\tiny \qcl}}_{\txt{\tiny Theorems~\mref{invrb},\mref{Thmdefmoda}}}
                & *\txt{\small $\mathcal{O}$-operator $\mathrm{id}_{A}$ on \\$(A,\{,\},\circ)$ associated to \\ \small  $(A,[,],\cdot,L_{\triangleright},L_{\succ})$.}
            }
            \end{displaymath}
    }
In particular, if the multiplications $\cdot, \cdot_h$ and $[\,,\,]$ are zero, then we have 
 \begin{equation}\notag%\label{maindiag0}
	\begin{split}
		\xymatrixcolsep{5pc}\xymatrixrowsep{3pc}
		\xymatrix{
			*\txt{\small Zinbiel algebra\\ \small $(A,\succ)$}
			\ar[r]^{\txt{\tiny deformation}}\ar[d]|-{\txt{\tiny
					Proposition~\ref{pro-ct}}}& *\txt{\small topologically free\\
				\small dendriform algebra\\ \small
				$(A_{h},\succ_{h},\prec_{h})$}\ar[r]^{\txt{ \tiny
					\qcl}}_{\txt{\tiny
					Corollary~\mref{cor-defdendri}}}\ar[d]|-{\txt{\tiny
					Proposition~\ref{pro-ct}}}
			& *\txt{\small pre-Poisson algebra\\\small $(A,\triangleright,\succ)$}\ar[d]|-{\txt{\tiny  Proposition~\ref{pro-cp}}}\\
			*\txt{\small $\mathcal{O}$-operator $\mathrm{id}_{A}$ on\\ \small $(A,\circ)$ associated\\ \small to $(A,L_{\succ})$}\ar[r]^{\txt{\tiny  deformation}}
			&*\txt{\small $\mathcal{O}$-operator $\mathrm{id}_{A_h}$ on\\ \small $(A_h,\circ_h)$ associated\\ \small to $(A_h,L_{\succ_h},R_{\prec_h})$}\ar[r]^{\txt{\tiny QCL}}_{\txt{\tiny  Theorem \mref{invrb}
			}}&*\txt{\small $\mathcal{O}$-operator $\mathrm{id}_{A}$ on\\ \small $(A,\{,\},\circ)$ associated\\ \small to $(A,L_{\triangleright},L_{\succ})$}
}
\end{split}
\end{equation}
\end{thm}
\begin{proof}
    (\ref{pro-diaid1}). By (\ref{eq:dt2}), we deduce
\vspc
     \begin{eqnarray*}
        x\cdot_{h}y&\equiv& x\cdot y\mmod{h},\\
        x\circ_{h}y &=&x_{h}\succ_{h} y_{h}+x_{h}\prec_{h}y_{h} + x_{h}\cdot_{h}y_{h} \equiv    x\circ y\mmod{h},\\
    L_{\succ_{h}}(x)y&=& x\succ_{h} y \equiv x\succ y\mmod{h} \equiv    L_{\succ}(x)y\mmod{h},\\
    R_{\prec_{h}}(y)x &=& x\prec_{h}y\equiv x\prec y\mmod{h}\equiv  R_{\prec}(y)x\mmod{h},  x,y\in A.
\vspb
 \end{eqnarray*}
Thus the $(A_{h},\circ_{h})$-bimodule algebra $(A_{h},\cdot
    _{h},L_{\succ_{h}},R_{\prec_{h}})$ is an $A_h$-bimodule deformation of the $(A,\circ)$-module algebra $(A,\cdot,L_{\succ})$ by Proposition~\ref{edm}.

    By (\ref{eq-br}) and (\ref{brak}), the \qcls of $(A_{h},\circ_{h})$ and $(A_{h},\cdot_{h})$ are the Poisson algebras $(A,\{,\},\circ)$ and $(A,[,],\circ)$ respectively. Note that
\vspb
\begin{equation*}
\frac{L_{\succ_{h}}(x)y-R_{\prec_{h}}(x)y}{h}=\frac{x\succ_{h}y-y\prec_{h}x}{h}\equiv x\triangleright y \mmod{h}\equiv L_{\triangleright}(x) y \mmod{h},\ x, y\in A.
\end{equation*}
Thus the conclusion follows.

 Because
\vspb
\[\mathrm{id}_{A_{h}}\Big(\sum_{i=0}^{\infty}x_{i}~h^{i}\Big)=\sum_{i=0}^{\infty}x_{i}~h^{i}=\sum_{i=0}^{\infty}\mathrm{id}_{A}(x_{i})~h^{i}, \quad \sum_{i=0}^\infty x_i~h^i\in A_h,
\vspa
\]
we obtain that $\mathrm{id}_{A_{h}}$ coincides with the scalar
deformation $(\id_A)_{\bf K}$ of $\mathrm{id}_{A}$. Hence
(\ref{pro-diaid2}) and (\ref{pro-diaid3}) hold.
\end{proof}

\begin{rmk}\label{rmk-pro-diaid}
As an
application of Theorem~\ref{pro-diaid}, we give another proof of Theorem~\ref{mainthm} as follows.
    By Theorem~\ref{pro-diaid}~(\ref{pro-diaid1}),  $(A,[,],\cdot,L_{\triangleright},L_{\succ})$ is exactly the \qcl of
    $(A_{h},\cdot_{h},L_{\succ_{h}},R_{\prec_{h}})$. Then
 $(A,[,],\cdot,L_{\triangleright},L_{\succ})$ is an $(A,\{,\},\circ)$-module Poisson algebra by Theorem~\ref{Thmdefmoda}. Hence by Proposition~\mref{pro-cp}, $(A,[,],\triangleright,\succ,\cdot)$ is a post-Poisson algebra.

Also note that Theorem~\ref{pro-diaid} provides the scalarly deformed $\calo$-operator $\mathrm{id}_{A_h}$ as an explicit example to illustrate Theorem~\ref{invrb}.
\end{rmk}
\vspd
\subsection{Constructions of tridendriform deformations by derivations}

We recall a deformation construction for associative algebras induced by derivations.
\begin{lem}{\rm{\mcite{G3}}}
Let $\dera$ and $\derb$ be  commuting derivations on an associative
algebra $(A,\circ)$. Set
\vspb
  \begin{equation}\mlabel{UDF}
    x\circ_{h}y:=\sum_{s=0}^{\infty}(\dera^{s}(x)\circ\derb^{s}(y))\frac{h^{s}}{s!}, \quad  x, y\in A,
\vspb
  \end{equation}
and extend it to $A_{h}$ by $\mathbf{K}$-bilinearity. Then $(A_h,\circ_h)$ is
an associative deformation of $(A,\circ)$. \mlabel{lemG}
\end{lem}
We extend this construction to $A_{h}$-bimodule algebra deformations.
\begin{pro}
 Let $(A,\circ)$ be an associative algebra and $(V,\cdot,l,r)$ an $A$-bimodule algebra.
 %Then $(A,\circ)$ and $(V,\cdot)$ are subalgebras of the semi-direct product algebra $A\ltimes_{l,r}V$.
Let $\dera$ and $\derb$ be commuting derivations on the
associative algebra $A\ltimes_{l,r}V$ such that $d_i(A)\subset A$ and $d_i(V)\subset V$, $i=1,2$. Define
 \begin{align*}
   x\circ_{h}y &:=\sum_{s=0}^{\infty}(\dera^{s}(x)\circ\derb^{s}(y))\frac{h^{s}}{s!}, &u\cdot_{h}v&:=\sum_{s=0}^{\infty}(\dera^{s}(u)\circ\derb^{s}(v))\frac{h^{s}}{s!}, \\
   l_{h}(x)v &:= \sum_{s=0}^{\infty}(l(\dera^{s}(x))\derb^{s}(v))\frac{h^{s}}{s!},
   &r_{h}(x)v &:=\sum_{s=0}^{\infty}(r(\dera^{s}(x))\derb^{s}(v))\frac{h^{s}}{s!}, \quad  x, y \in A, u, v\in V.
 \end{align*}
Extend $\mathbf{K}$-bilinearly the above operations to $A_{h}$ and $V_h$. Then  $(V_{h},\cdot_{h},l_{h},r_{h})$ is an
$A_{h}$-bimodule algebra deformation of $(V,\cdot,l,r)$.
\mlabel{pp:der}
\end{pro}
\begin{proof} Obviously, $\dera$ and $\derb$ are commuting
derivations on the subalgebras $(A,\circ)$ and $(V,\cdot)$
respectively. By Lemma~\mref{lemG}, $(A_{h},\circ_{h})$ and
$(V_{h},\cdot_{h})$
    are associative deformations of $(A,\circ)$ and $(V,\cdot)$ respectively.
  We only need to show that $(A_{h}\oplus V_{h},\odot_{h})$ is an associative deformation of $A\ltimes_{l,r}V$.
Note that
\vspb
  \begin{eqnarray*}
 (x,u)\odot_{h}(y,v)
    &=&(x\circ_{h} y,l_{h}(x)v+r_{h}(y)u+u\cdot_{h} v) \\
    & =&\sum_{s=0}^{\infty}(\dera^{s}(x)\circ\derb^{s}(y),l(\dera^{s}(x))\derb^{s}(v)+r(\dera^{s}(y))\derb^{s}(u)+\dera^{s}(u)\circ\derb^{s}(v))\frac{h^{s}}{s!} \\
    & =&
    \sum_{s=0}^{\infty}\dera^{s}((x,u))\odot\derb^{s}((y,v))\frac{h^{s}}{s!},\;\;
    x,y\in A, u,v\in V.
\vspb
  \end{eqnarray*}
By Lemma~\mref{lemG} again,  $(A_{h}\oplus V_{h},\odot_{h})$ is an
associative deformation of $A\ltimes_{l,r}V$. Thus
$(V_{h},\cdot_{h},l_{h},r_{h})$ is an $A_{h}$-bimodule algebra
deformation of $(V,\cdot,l,r)$.
\end{proof}

Let $(A,\succ,\prec,\cdot)$ be a tridendriform algebra. If $d\in
\mathrm{End}_{\bfk}(A)$ satisfies
\vspb
$$d(x\ast y)=d(x)\ast y+x\ast
d(y),\;\; x,y\in A, \ast\in\{\succ,\prec,\cdot\},
\vspb
$$ then $d$ is called a \textbf{derivation on the tridendriform algebra}
$(A,\succ,\prec,\cdot)$. Note that in this case, the operator $d\dotplus d\in \mathrm{End}_{\bfk}(A\oplus A)$ defined by $d\dotplus d((x,y)):=(d(x),d(y)), x,y\in A,$ is a derivation on the associative algebra $A\ltimes_{L_{\succ},R_{\prec}}A$. Thus applying Proposition~\mref{pp:der} to the
$(A,\circ)$-bimodule algebra $(A,\cdot,L_{\succ},R_{\prec})$ gives the following consequence.
\begin{cor}\label{cordt}
If $\dera$ and $\derb$ are  commuting derivations on a tridendriform
algebra $(A,\succ,\prec,\cdot)$, then the operations
\vspc
  \begin{equation}
    x\ast_{h}y:=\sum_{s=0}^{\infty}(\dera^{s}(x)\ast\derb^{s}(y))\frac{h^{s}}{s!}, \quad  x, y\in A,
\vspa
  \end{equation}
for each $\ast\in\{\succ,\prec,\cdot\}$, give a tridendriform
deformation $(A_{h},\succ_{h},\prec_{h},\cdot_{h})$ of
$(A,\succ,\prec,\cdot)$.
\vspb
\end{cor}

\begin{pro}
With the assumptions in Proposition\,\mref{pp:der}. Let $T$ be an
$\mathcal{O}$-operator of weight $\lambda$ on $(A,\circ)$
associated to $(V,\cdot,l,r)$. If $T\dera|_{V}=\dera T$ and
$T\derb|_{V}=\derb T$, then $T_\mathbf{K}$ defined by (\ref{eq:TK}) is
an $\mathcal{O}$-operator of weight $\lambda$ on
$(A_{h},\circ_{h})$ associated to $(V_{h},\cdot_{h},l_{h},r_{h})$.
\mlabel{pp:dero}
\end{pro}
\begin{proof}Let $u,v\in V$. Then for $s\in \mathbb{N}$, we have
\vspa
{\small  \begin{eqnarray*}
    T(u)\circ_{s}T(v)&=& \frac{\dera^{s}(T(u))\circ\derb^{s}(T(v))}{s!} \\
    &=& \frac{T(\dera^{s}(u))\circ T(\derb^{s}(v))}{s!} \\
    &=& \frac{T(l(T\dera^{s}(u))\derb^{s}(v))}{s!}+\frac{T(r(T\derb^{s}(v))\dera^{s}(u))}{s!}+\lambda\frac{T(\dera^{s}(u)\cdot\derb^{s}(v))}{s!} \\
    &=& \frac{T(l(\dera^{s}(T(u))\derb^{s}(v))}{s!}+\frac{T(r(\derb^{s}(T(v)))\dera^{s}(u))}{s!}+\lambda\frac{T(\dera^{s}(u)\cdot\derb^{s}(v))}{s!}  \\
    &=& T(l_{s}(T(u))v)+T(r_{s}(T(v)))+\lambda T(u\cdot_{s}v).
  \end{eqnarray*}
}
 Therefore $T_\mathbf{K}$ is an $\mathcal{O}$-operator of weight $\lambda$ on $(A_{h},\circ_{h})$ associated to $(V_{h},\cdot_{h},l_{h},r_{h})$.
\end{proof}

\begin{pro}
  Let $\dera$ and $\derb$ be  commuting derivations on an associative algebra $(A,\cdot)$ and $T$ be a Rota-Baxter operator of weight $1$ on $(A,\cdot)$ that commutes with $\dera$ and $\derb$. Define
\vspb
  \begin{eqnarray*}
    &x\succ y := T(x)\cdot y, \quad x\prec y := x\cdot T(y), &\\
   & x\succ_{h} y := \sum_{s=0}^{\infty}(\dera^{s}(T(x))\cdot\derb^{s}(y))\frac{h^{s}}{s!}, \quad
    x\prec_{h} y := \sum_{s=0}^{\infty}(\dera^{s}(x)\cdot\derb^{s}(T(y)))\frac{h^{s}}{s!}, &\\
    &x\cdot_{h}y =\sum_{s=0}^{\infty}(\dera^{s}(x)\cdot\derb^{s}(y))\frac{h^{s}}{s!}, \quad x, y\in A.&
  \end{eqnarray*}
Extend $\mathbf{K}$-bilinearly the above operations $\succ_{h}$, $\prec_{h}$ and $\cdot_{h}$ to
  $A_{h}$. Then $(A_{h},\succ_{h},\prec_{h},\cdot_{h})$ is a tridendriform
deformation of the tridendriform algebra $(A,\succ,\prec,\cdot)$.
If in addition $(A,\cdot)$ is commutative, then the \qcl of the tridendriform deformation is the post-Poisson algebra $(A,\triangleright,[,],\succ,\cdot)$, where $\triangleright$ and $[,]$ are respectively defined by
\vspb
  \[x\triangleright y:= \dera(T(x))\cdot\derb(y)- \dera(y)\cdot\derb(T(x)),\quad
  [x,y]:= \dera(x)\cdot\derb(y)-\dera(y)\cdot\derb(x),\;\; x,y\in A.
 \vspb
 \]
  \mlabel{pp:derp}
\vspb
\end{pro}

\begin{proof}
Applying Proposition~\ref{pp:dero} to the $\calo$-operator $T$ of weight $1$ on $(A,\cdot)$ associated to the $A$-bimodule algebra $(A,\cdot,L_\cdot, R_\cdot)$, we find that $T_{\bf K}$ is an Rota-Baxter operator of weight $1$ on $(A_{h},\cdot_{h})$. Note that (\ref{eq:dt2}) also holds. Thus $(A_{h},\succ_{h},\prec_{h},\cdot_{h})$ is a tridendriform
deformation of $(A,\succ,\prec,\cdot)$. The last conclusion follows from a direct verification.
\end{proof}

We apply the above construction to give an example of the tridendriform deformation.

Let $(A,\cdot)$ be the
(nonunital) subalgebra of $\bfk[x_{1},x_{2}]$ generated by
$x_{1}$ and $x_{2}$, that is,
$A=x_{1}\bfk[x_{1},x_{2}]+x_{2}\bfk[x_{1},x_{2}]$. It is known~\mcite{Gu} that, if $q\in \bfk$ is not a root of unit, then a Rota-Baxter operator of weight $1$ on the algebra $x_{1}\bfk[x_{1}]$ is given by
\vspb
\[T(x_{1}^{i_{1}})=\frac{q^{i_{1}}x_{1}^{i_{1}}}{1-q^{i_{1}}}, \quad i_{1}\in \mathbb{Z}^{+}.
\vspb
\]

Let $q_1, q_2\in \bfk$ be such that
$q_{1}^{i_{1}}q_{2}^{i_{2}}\neq1$ for $i_{1}, i_{2}\geq 0$
with $i_1i_2\neq 0$. Define a linear map
\vspb
$$T_{q_1,q_2}:A\rightarrow A,\quad  T_{q_1,q_2}(x_{1}^{i_{1}}x_{2}^{i_{2}}) := \frac{q_{1}^{i_{1}}q_{2}^{i_{2}}}{1-q_{1}^{i_{1}}q_{2}^{i_{2}}}x_{1}^{i_{1}}x_{2}^{i_{2}},\quad i_{1},i_{2}\in \mathbb{N}, i_{1}i_{2}\neq0.
\vspb
$$
Noting that
{\small
\begin{equation*}
    \frac{q_{1}^{i_{1}+j_{1}}q_{2}^{i_{2}+j_{2}}}{(1-q_{1}^{i_{1}}q_{2}^{i_{2}})(1-q_{1}^{j_{1}}q_{2}^{j_{2}})}=\frac{q_{1}^{i_{1}}q_{2}^{i_{2}}}{1-q_{1}^{i_{1}}q_{2}^{i_{2}}}\frac{q_{1}^{i_{1}+j_{1}}q_{2}^{i_{2}+j_{2}}}{1-q_{1}^{i_{1}+j_{1}}q_{2}^{i_{2}+j_{2}}}+\frac{q_{1}^{j_{1}}q_{2}^{j_{2}}}{1-q_{1}^{j_{1}}q_{2}^{j_{2}}}\frac{q_{1}^{i_{1}+j_{1}}q_{2}^{i_{2}+j_{2}}}{1-q_{1}^{i_{1}+j_{1}}q_{2}^{i_{2}+j_{2}}}+\frac{q_{1}^{i_{1}+j_{1}}q_{2}^{i_{2}+j_{2}}}{1-q_{1}^{i_{1}+j_{1}}q_{2}^{i_{2}+j_{2}}},
\end{equation*}
}
 where $i_{1},i_{2},j_{1},j_{2}\in \mathbb{N}$, $i_{1}i_{2}\neq0$, $j_{1}j_{2}\neq0$, we deduce that
\vspb
 \begin{eqnarray*}
    &&T_{q_1,q_2}(x_{1}^{i_{1}}x_{2}^{i_{2}})\cdot T_{q_1,q_2}(x_{1}^{j_{1}}x_{2}^{j_{2}})\\
    &\;&=T_{q_1,q_2}(T_{q_1,q_2}(x_{1}^{i_{1}}x_{2}^{i_{2}})\cdot x_{1}^{j_{1}}x_{2}^{j_{2}})+T_{q_1,q_2}(x_{1}^{i_{1}}x_{2}^{i_{2}}\cdot T_{q_1,q_2}(x_{1}^{j_{1}}x_{2}^{j_{2}}))+T_{q_1,q_2}(x_{1}^{i_{1}}x_{2}^{i_{2}}\cdot x_{1}^{j_{1}}x_{2}^{j_{2}}).
    \end{eqnarray*}
Hence $T_{q_1,q_2}$ is a Rota-Baxter operator of weight 1 on $A$.
 We also set
\vspb
\[\dera:=x_{1}\frac{\partial}{\partial x_{1}},\quad  \derb:=x_{2}\frac{\partial}{\partial x_{2}}.
\vspb
\]
Then $T_{q_1,q_2}, \dera$ and $\derb$ are pairwise commuting. Therefore we obtain the following example.
\vspb
\begin{ex}\label{ex-nonuni}
  Let $A=x_{1}\bfk[x_{1},x_{2}]+x_{2}\bfk[x_{1},x_{2}]$. Fix $q_1, q_2\in \bfk$ such that $q_{1}^{i_{1}}q_{2}^{i_{2}}\neq1$ for $i_{1}, i_{2}\geq 0$ with $i_1i_2\neq 0$. A commutative tridendriform algebra structure $(A,\succ,\cdot)$ is defined by
\vspb
{\small  \begin{eqnarray}
     x_{1}^{i_{1}}x_{2}^{i_{2}}\succ x_{1}^{j_{1}}x_{2}^{j_{2}}&:=& \frac{q_{1}^{i_{1}}q_{2}^{i_{2}}}{1-q_{1}^{i_{1}}q_{2}^{i_{2}}}x_{1}^{i_{1}+j_{1}}x_{2}^{i_{2}+j_{2}}, \mlabel{2,v1}\\
    x_{1}^{i_{1}}x_{2}^{i_{2}}\cdot x_{1}^{j_{1}}x_{2}^{j_{2}}&:=& x_{1}^{i_{1}+j_{1}}x_{2}^{i_{2}+j_{2}}\mlabel{2,v2},
    \ \   i_{1},i_{2},j_{1},j_{2}\in \mathbb{N}, i_{1}i_{2}\neq0, j_{1}j_{2}\neq0.
  \end{eqnarray}
}
  A tridendriform deformation $(A,\succ_{h},\prec_{h},\cdot_{h})$ of $(A,\succ,\cdot)$ is given by
\vspb
{\small
  \begin{eqnarray*}
&     x_{1}^{i_{1}}x_{2}^{i_{2}}\succ_{h} x_{1}^{j_{1}}x_{2}^{j_{2}}:= \sum_{s=0}^{\infty}\frac{i_{1}^{s}j_{2}^{s}q_{1}^{i_{1}}q_{2}^{i_{2}}}{1-q_{1}^{i_{1}}q_{2}^{i_{2}}}x_{1}^{i_{1}+j_{1}}x_{2}^{i_{2}+j_{2}}\frac{h^{s}}{s!},\mlabel{i1} \quad
    x_{1}^{i_{1}}x_{2}^{i_{2}}\prec_{h} x_{1}^{j_{1}}x_{2}^{j_{2}}:= \sum_{s=0}^{\infty}\frac{i_{1}^{s}j_{2}^{s}q_{1}^{j_{1}}q_{2}^{j_{2}}}{1-q_{1}^{j_{1}}q_{2}^{j_{2}}}x_{1}^{i_{1}+j_{1}}x_{2}^{i_{2}+j_{2}}\frac{h^{s}}{s!},\mlabel{i2} &\\
&x_{1}^{i_{1}}x_{2}^{i_{2}}\cdot_{h} x_{1}^{j_{1}}x_{2}^{j_{2}}:= \sum_{s=0}^{\infty}i_{1}^{s}j_{2}^{s}x_{1}^{i_{1}+j_{1}}x_{2}^{i_{2}+j_{2}}\frac{h^{s}}{s!}\mlabel{i3}.&
\end{eqnarray*}
}
The post-Poisson algebra $(A,[,],\triangleright,\succ,\cdot)$ for the corresponding \qcl is given by (\mref{2,v1}), (\mref{2,v2}) and
\vspb
{\small $$   \ \big[x_{1}^{i_{1}}x_{2}^{i_{2}}, x_{1}^{j_{1}}x_{2}^{j_{2}}\big]:=
   (i_{1}j_{2}-i_{2}j_{1})x_{1}^{i_{1}+j_{1}}x_{2}^{i_{2}+j_{2}},\ \
     x_{1}^{i_{1}}x_{2}^{i_{2}}\triangleright x_{1}^{j_{1}}x_{2}^{j_{2}}:=                    (i_{1}j_{2}-i_{2}j_{1})\frac{q_{1}^{i_{1}}q_{2}^{i_{2}}}{1-q_{1}^{i_{1}}q_{2}^{i_{2}}}x_{1}^{i_{1}+j_{1}}x_{2}^{i_{2}+j_{2}}.
$$
}
\vspb
\end{ex}
\vspe

 \section{Solutions of PYBE from scalar deformations of solutions of AYBE}
\mlabel{s:ybe}
In this section, we construct solutions of the Poisson Yang-Baxter equation (PYBE) by our general method of bimodule algebra deformations.
In Section~\mref{sub-triaybe}, let $(A,\circ)$ be a commutative associative algebra whose associative deformation
 $(A_{h},\circ_{h})$ has it \qcl Poisson
 algebra $(A,\{,\},\circ)$.
For a solution of the AYBE in $(A,\circ)$, we show that its
scalarly deformed solution of the AYBE in $(A_{h},\circ_{h})$ with
invariant symmetric part induces a solution of the PYBE in
$(A,\{,\},\circ)$. Some notational preparations on topological
dual bimodules are given in Section~\mref{tdb}. Then in
Section~\mref{css}, we characterize scalarly deformed solutions of
the AYBE by scalarly deformed $\calo$-operators for deformed
associative algebras which, under taking the \qcls, agrees with
the characterization of solutions of the PYBE by $\calo$-operators
for Poisson algebras. Finally in Sections~\mref{ss:ybeopdend} and
\mref{ss:ybetrid}, we obtain solutions of the PYBE from dendriform
deformations of Zinbiel algebras and tridendriform deformations of
commutative tridendriform algebras respectively.

\subsection{Scalar deformations of solutions of AYBE  in associative deformations}\mlabel{sub-triaybe}
We recall the notions of associative Yang-Baxter equations.

\begin{defi}\rm{\mcite{BGN2}}
    Let $(A,\circ)$ be an associative algebra and $r=\sum_{i}a_{i}\otimes b_{i}\in A\otimes A$.
Set
\vspb
\begin{equation}\mlabel{mul-1}
  r_{12}\circ r_{13}:=\sum_{i,j}a_{i}\circ a_{j}\ot b_{i}\otimes b_{j},\
  r_{13}\circ r_{23}:=\sum_{i,j}a_{i}\otimes  a_{j}\otimes b_{i}\circ b_{j},\
  r_{23}\circ r_{12}:=\sum_{i,j}a_{j}\otimes a_{i}\circ b_{j}\otimes b_{i}.
\vspb
\end{equation}
We say that $r$ is a solution of the
\textbf{associative Yang-Baxter equation (AYBE)} in $(A,\circ)$ if
\vspb
\begin{equation}\mlabel{aybe}
\mathbf{A}(r):=r_{12}\circ r_{13}+r_{13}\circ r_{23}-r_{23}\circ r_{12}=0.
\vspb
\end{equation}
%\vspb
\end{defi}
\begin{rmk}\label{rop}
In~\mcite{A1}, another form of the AYBE is defined by
\vspb
     \begin{equation}\mlabel{oaybe}
       \mathbf{A}^{\rm op}(r):=r_{13}\circ r_{12}+r_{23}\circ r_{13}-r_{12}\circ r_{23}=0,
\vspb
     \end{equation}
  where
\vspb
  \begin{equation}\mlabel{mul-2}
  r_{13}\circ r_{12}:=\sum_{i,j}a_{i}\circ a_{j}\ot b_{j}\otimes a_{j},\
  r_{23}\circ r_{13}:=\sum_{i,j} a_{j}\otimes a_{i} \ot b_{i}\circ b_{j},\
  r_{12}\circ r_{23}:=\sum_{i,j}a_{i}\otimes b_{i}\circ a_{j}\otimes b_{j}.
\vspb
  \end{equation}
By \cite[Lemma 3.4]{A3}, when the symmetric
part
\vspb
\begin{equation}
\beta:=\frac{1}{2}(r+\sigma(r))
\mlabel{eq:symm}
\vspb
\end{equation}
of $r$ is $(A,\circ)$-\textbf{invariant}, that is,
\vspb
\begin{equation}
(\mathrm{id}_{A}\otimes L_{\circ}(x)-R_{\circ}(x)\otimes
\mathrm{id}_{A})\beta=0, \quad  x\in A,
\mlabel{eq:inv}
\vspb
\end{equation}
then
     $\mathbf{A}^{\rm op}(r)=0$ if and only if $\mathbf{A}(r)=0$. Here $\sigma$ is the \textbf{twisting operator}
\vspb
$$\sigma:A\ot A\to A\ot A, \quad \sigma (x\otimes y)=y\otimes x, \quad x,y\in A.
\vspb
$$
\vspd
\end{rmk}
\vspb

The AYBE can be defined on a topologically free associative algebra by replacing the usual tensor product in the AYBE with the topological tensor product. See~\mcite{HBG,Ka} for details. In fact, let $V$ and $W$ be vector spaces and $V_{h}=V[[h]]$ and $W_{h}=W[[h]]$ be topologically free $\mathbf{K}$-modules. Then $V_{h}\hat{\otimes}W_{h}\cong (V\otimes W)[[h]]$.

Let $(A,\circ)$ be an associative algebra and $(A_h,\circ_h)$ be an associative deformation. Any element
$r=\sum_{i}a_{i}\otimes b_{i}\in A\otimes A$ is naturally an
element in $(A\otimes A)[[h]]$. Then under the isomorphism
$(A\otimes A)[[h]]\cong A_h\hat{\ot}A_h$, $r$ is also naturally an
element in $A_h\hat{\ot}A_h$ and so takes the form
$\sum_{i}a_{i}\hat{\ot} b_{i}$. We let $\hat r$ denote the image
of $r\in A\ot A$ under the linear embedding of $A\otimes A$ into
$A_{h}\hat{\otimes}A_{h}$ and call $\hat r$ the \textbf{scalar
deformation} of $r$.

We can also extend the twisting operator $\sigma:A\ot A\to A\ot A$ to the topological tensor product by
$\bf K$-bilinear extension to get the \textbf{topological twisting operator}
\vspb
$$\sigma_{\bf K}:A_h\hat{\ot} A_h\to A_h\hat{\ot} A_h, \quad \sigma_{\bf K}(x_h{\hat \otimes} y_h)=y_h{\hat \otimes} x_h,\;\; x_h,y_h\in A_h.
\vspb
$$

Let $(\mathfrak{g},[,])$ be a Lie algebra and $r=\sum_{i}y_{i}\otimes y_{i}^{\prime}\in \mathfrak{g}\otimes\mathfrak{g}$. Set
\vspb
 \begin{equation}
  [r_{12},r_{13}]:=\sum_{i,j}[y_{i},y_{j}]\otimes y_{i}^{\prime}\otimes y_{j}^{\prime},\ [r_{12},r_{23}]:=\sum_{i,j}y_{i}\otimes [y_{i}^{\prime}, y_{j}]\otimes y_{j}^{\prime},\
  [r_{13},r_{23}]:=\sum_{i,j}y_{i}\otimes  y_{j}\otimes[y_{i}^{\prime},y_{j}^{\prime}].
\vspc
\end{equation}
Recall that the equation
\vspb
 \begin{equation}\mlabel{cybe}
   \mathbf{C}(r):=[r_{12},r_{13}]+[r_{12},r_{23}]+[r_{13},r_{23}]=0
\vspb
 \end{equation}
is called the \textbf{classical Yang-Baxter equation (CYBE)}. For the adjoint representation of $\frakg$:
\vspb
$$\rad:\frak g\rightarrow {\rm End}_{\bf k}(\frak g), \quad \rad(x)(y)=[x,y], \quad x,y\in \frak g,
\vspb
$$
an element $r\in \mathfrak{g}\otimes\mathfrak{g}$ is called $(\mathfrak{g},[,])$-\textbf{invariant} if
\vspb
$$(\rad(x)\otimes \mathrm{id}+\mathrm{id}\otimes \rad(x))r=0, \quad x\in\mathfrak{g}.
\vspb
$$

 \begin{defi}\rm{\mcite{NB}}
   Let $(A,\{,\},\circ)$ be a Poisson algebra and $r\in A\otimes A$. The equation
\vspb
   \begin{equation}\mlabel{pybe}
     \mathbf{A}(r)=\mathbf{C}(r)=0
\vspb
   \end{equation}
is called the \textbf{Poisson Yang-Baxter equation (PYBE)} in $(A,\{,\},\circ)$.
\vspb
\end{defi}

\begin{thm}
   Let $(A,\circ)$ be a commutative associative algebra, $(A_{h},\circ_{h})$ be an associative deformation and $(A,\{,\},\circ)$ be the corresponding \qcl.
  Let $r=\sum_{i}a_{i}\otimes b_{i}\in A\otimes A$ be a solution of the AYBE in $(A,\circ)$ and $r+\sigma(r)$ be $(A,\circ)$-invariant. Let $\hr$ denote the image of $r$ under the linear embedding of $A\otimes A$ into $A_{h}\hat{\otimes}A_{h}$. Suppose that $\hr$ is a solution of the AYBE in $(A_{h},\circ_{h})$ and $\hr+\sigma_{\bf K}(\hr)$ is $(A_{h},\circ_{h})$-invariant. Then $r$ is a solution of the PYBE in $(A,\{,\},\circ)$ and $r+\sigma(r)$ is $(A,\{,\},\circ)$-invariant.
 \mlabel{inv-aybe}
\end{thm}
 \begin{proof} Since $\hr$ is a solution of the AYBE in $(A_{h},\circ_h )$ and $\hr+\sigma_{\bf K}(\hr)$ is $(A_{h},\circ_{h})$-invariant, similar to Remark~\ref{rop}, we have $\mathbf{A}^{\rm op}(\hr)=0$  in $(A_{h},\circ_{h})$. Note that
\vspb
{\small \begin{eqnarray*}
 \hr_{12} \circ_h \hr_{13} -\hr_{13} \circ_h \hr_{12}&=&\sum_{i,j}a_{i}\circ_h a_{j}\hat{\ot} b_{i}\hat{\ot} b_{j}-\sum_{i,j}a_{i}\circ_h a_{j}\hat{\ot} b_{j}\hat{\ot} b_{i}\\
 &=&\sum_{i,j}a_{i}\circ_h a_{j}\hat{\ot} b_{i}\hat{\ot} b_{j}-\sum_{i,j}a_{j}\circ_h a_{i}\hat{\ot} b_{i}\hat{\ot} b_{j}\\
 &=&\sum_{i,j}(a_{i}\circ_h a_{j}-a_{j}\circ_h a_{i})\hat{\ot} b_{i}\hat{\ot}b_{j}\\
 &\equiv& \sum_{i,j}\{a_{i},a_{j}\}\hat{\ot} b_{i} \hat{\ot} b_{j}~h\mmod{h^2}.
     \end{eqnarray*}
 }
  Similarly, we obtain
\vspb
  \begin{eqnarray*}
    \hr_{13} \circ_h \hr_{23} -\hr_{23} \circ_h \hr_{13}&\equiv& \sum_{i,j}a_{i}\hat{\ot} a_{j} \hat{\ot} \{b_{i}b_{j}\}~h\mmod{h^2},\\
        \hr_{12} \circ_h \hr_{23} -\hr_{23} \circ_h \hr_{12}&\equiv& \sum_{i,j}a_{i}\hat{\ot} \{b_{i},a_{j}\} \hat{\ot} b_{j} ~h\mmod{h^2}.
\vspd
 \end{eqnarray*}
   So
\vspb
$$\frac{\mathbf{A}(\hr)-\mathbf{A}^{\rm op}(\hr)}{h}\equiv \mathbf{C}(r)\mmod{h}.
\vspa
$$
Then we obtain $\mathbf{C}(r)=0$ in $(A,\{,\})$. Note that $r$ is a solution of the AYBE in $(A,\circ)$. Hence $r$ is also a solution of the PYBE in $(A,\{,\},\circ)$.

 Since $\hr+\sigma_{\bf K}(\hr)$ is $(A_{h},\circ_{h})$-invariant, we have
\vspa
\begin{equation}\mlabel{hiva}
(\mathrm{id}_{A_h} \hat{\otimes} L_{\circ_{h}}(x)-R_{\circ_{h}}(x)\hat{\otimes} \mathrm{id}_{A_h})(\hr+\sigma_{\bf K}(\hr))=0, \quad  x\in A.
\vspa
\end{equation}
Applying $\sigma_\mathbf{K}$  to both sides
of (\mref{hiva}), we deduce that
\vspa
   \begin{equation}\mlabel{ohiva}
     (L_{\circ_{h}}(x)\hat{\otimes} \mathrm{id}_{A_h}-\mathrm{id}_{A_h}\hat{\otimes} R_{\circ_{h}}(x))(\hr+\sigma_{\bf K}(\hr))=0,\;\; x\in A.
\vspa
   \end{equation}
So adding (\mref{hiva}) and (\mref{ohiva}) gives
\vspb
   \[((L_{\circ_{h}}(x)-R_{\circ_{h}}(x))\hat{\otimes} \mathrm{id}_{A_h}+\mathrm{id}_{A_h}\hat{\otimes}(L_{\circ_{h}}(x)-R_{\circ_{h}}(x)))(\hr+\sigma_{\bf K}(\hr))=0,\;\; x\in A.
   \vspb
   \]
   Since $\frac{L_{\circ_{h}}(x)-R_{\circ_{h}}(x)}{h}\equiv \rad_{\{,\}}(x) \mmod{h}$ for $x\in A$, we see that $r+\sigma(r)$ is also $(A,\{,\})$-invariant, implying
its $(A,\{,\},\circ)$-invariance due to the already assumed $(A,\circ)$-invariance.
\vspb
 \end{proof}

\begin{rmk}\label{rmkinvyb}
    In fact, without the assumption that $\hr$ is a solution of the AYBE in $(A_{h},\circ_{h})$, the conclusion involving the invariance is still available. Explicitly,
    for any symmetric and $(A,\circ)$-invariant $\beta\in A\ot A$, if $\hat {\beta}$ is $(A_{h},\circ_{h})$-invariant, then $\beta$ is $(A,\{,\},\circ)$-invariant.
\end{rmk}
\vspd

\subsection{Topological dual bimodules}\label{tdb}
For later needs, we fix notations and identifications on the interplay among linear duals, tensor products and deformations of vector spaces. We then prove the commutativity of the processes of taking the dual and taking the \qcls.

Let $V$ and $W$ be finite-dimensional vector spaces. Let $V^{\ast}=\textrm{Hom}_{\bfk}(V,\bfk)$ and
$W^{\ast}=\textrm{Hom}_{\bfk}(W,\bfk)$ be the dual space of $V$
and $W$ respectively. For $r\in W\otimes V$, let $r^{\sharp}$ be
the image under the linear isomorphism \vspb
\begin{equation}
    W\otimes V\cong \textrm{Hom}_{\bfk}(V^{\ast},W).
\mlabel{eq:tendual}
\vspb
\end{equation}
It is characterized by
\vspc
\begin{equation}\mlabel{map-tensor}
\langle w^{\ast},r^{\sharp}(v^{\ast})\rangle = \langle w^{\ast}\otimes v^{\ast},r \rangle,  \quad v^{\ast}\in V^{\ast}, w^{\ast}\in W^{\ast},
\vspb
\end{equation}
where $\langle\ ,\ \rangle$ is the natural pairing between the dual space of a vector space and the vector space.
For $f\in\textrm{Hom}_{\bfk}(V^{\ast},W)$, we also let $\tilde{f}$ denote its image in
$W\otimes V$.
For a vector space $\mathrm{A}$ and a linear map $\varrho:\mathrm{A}\rightarrow \textrm{End}_{\bfk}(V)$, we let $\varrho^{\ast}:\mathrm{A}\rightarrow \textrm{End}_{\bfk}(V^{\ast})$ be the unique linear map satisfying
\vspb
\begin{equation}
\langle\varrho^{\ast}(x)u^{\ast},v\rangle=-\langle u^{\ast},\varrho(x)v\rangle, \quad  x\in \mathrm{A}, u^{\ast}\in V^{\ast},v\in V.
\vspb
\end{equation}
Let $(A,\circ)$ be an associative algebra and $(V,l_{\circ},r_{\circ})$ be an $A$-bimodule. Then $(V,-r_{\circ}^{\ast},-l_{\circ}^{\ast})$ is an $A$-bimodule, called the \textbf{dual bimodule} of $(V,l_{\circ},r_{\circ})$.
Similarly, if $(P,[,],\cdot)$ is a Poisson algebra and $(W,\rho_{[,]},\rho_{\cdot})$ is a $(P,[,],\cdot)$-module, then
$(W^{\ast},\rho_{[,]}^{\ast},-\rho_{\cdot}^{\ast})$ is a $(P,[,],\cdot)$-module, called the \textbf{dual module} of the $(P,[,],\cdot)$-module $(W,\rho_{[,]},\rho_{\cdot})$.

For topologically free $\mathbf{K}$-modules $V_h$ and $W_h$, we have $\bf K$-module isomorphisms
\vspb
 $$\mathrm{Hom}_{\bfk}(V^{\ast},W)[[h]]\cong (W\otimes V)[[h]]\cong W_{h}\hat{\otimes} V_{h}.
 \vspb
 $$
On the other hand, by
\cite[\uppercase\expandafter{\romannumeral16}.7, Exercise 2]{Ka}, there is the $\bfK$-module isomorphism
\vspb
\begin{equation}
\Phi_{V,W}:\mathrm{Hom}_{\bfk}(V,W)[[h]]\to \mathrm{Hom}_{\mathbf{K}}(V_{h},W_{h}), \ \
    \Phi_{V,W}\Big(\sum_{s=0}^{\infty}\psi_{s}h^{s}\Big):=\sum_{s=0}^{\infty}(\psi_{s})_{\mathbf{K}}~h^{s},
\vspb
\end{equation}
where $\psi_{s}\in \mathrm{Hom}_{\bfk}(V,W)$ for $s\in \NN$, and
$(\psi_{s})_\mathbf{K}$
is defined as in (\ref{eq:TK}).
In particular, taking $W={\bf k}$, we obtain
\vspb
$$\textrm{Hom}_{\mathbf{K}}(V_{h},\mathbf{K})\cong\textrm{Hom}_{\bfk}(V,\bfk)[[h]]=(V^{\ast})_{h}.
\vspa
$$
Then the $\bf K$-linear map $\Phi_{V,\bf k}:\mathrm{Hom}_{\bfk}(V,\bfk)[[h]]\to \mathrm{Hom}_{\mathbf{K}}(V_{h},\bf K)$ is given by
\vspb
\begin{equation}
 \Phi_{V,\bfk}\Big(\sum_{s=0}^{\infty}v^{\ast}_{s}h^{s}\Big):=\sum_{s=0}^{\infty}(v^{\ast}_{s})_{\mathbf{K}}~h^{s},
\vspa
 \end{equation}
where $v^{\ast}_{s}\in \mathrm{Hom}_{\bfk}(V,\bfk)$ for any $s\in \NN$.
For clarity, we use the notation
\vspb
$$V^{\star}_{h}:=(V^{\ast})_{h}.
\vspb
$$
Then we obtain
\vspc
 \begin{equation}\label{is}
    \mathrm{Hom}_{\bfk}(V^{\ast},W)[[h]]\cong\mathrm{Hom}_{\bf K}(V^{\star}_{h},W_{h})\cong W_{h}\hat{\otimes} V_{h}\cong (W\otimes V)[[h]].
\vspb
 \end{equation}
Under these isomorphisms, the images of $f\in\mathrm{Hom}_{\bfk}(V^{\ast},W)$ in $\mathrm{Hom}_{\bf K}(V^{\star}_{h},W_{h})$, $W_{h}\hat{\otimes} V_{h}$ and $(W\otimes V)[[h]]$
   are denoted by $f_{\bf K}$, $\hat{\tilde{f}}$ and  $\tilde{f}$ respectively.
Also, the images
  of $r\in W\ot V$ in  $\mathrm{Hom}_{\bfk}(V^{\ast},W)[[h]]$, $\mathrm{Hom}_{\bf K}(V^{\star}_{h},W_{h})$ and $W_{h}\hat{\otimes} V_{h}$
 are denoted by $r^{\sharp}$, $(r^{\sharp})_{\bf K}$ and $\hr$
 respectively.

Next we define the pairing
\vspb
 \begin{equation}
\langle\ ,\ \rangle_{\bf K}: V_{h}^{\star}\hat{\ot}V_{h}\to \bfK,
\quad  \langle u^{\star},v^{\prime}\rangle_{\bf
K}:=\Phi_{V,\bfk}(u^{\star})(v^{\prime}),\;\;  u^{\star}\in
V_{h}^{\star}, v^{\prime}\in V_h.
\vspa
 \end{equation}
Since $\Phi_{V,\bfk}(v^{\ast})=(v^{\ast})_{\bf K},  v^{\ast}\in V^{\ast},$
 we get
\vspb
 \begin{equation}\label{pres}
    \langle v^{\ast},u\rangle_{\bf K}=(v^{\ast})_{\bf K}(u)=\langle v^{\ast},u\rangle,\quad v^{\ast}\in V^{\ast},u\in V.
\vspa
 \end{equation}
Hence, $\langle\ ,\ \rangle_{\bf K}$ is precisely the $\bf K$-bilinear extension of the usual pairing $\langle\ ,\ \rangle$.
Since there are natural isomorphisms of $\bfK$-modules
\vspa
$$W_{h}^{\star}\hat{\ot}V_{h}^{\star}\cong(W^{\ast}\otimes V^{\ast})[[h]]\;\textrm{and}\;W_{h}\hat{\ot}V_{h} \cong(W\otimes V)[[h]],
\vspa
$$
the pairing $\langle\ ,\ \rangle_{\bf K}$ can be turned into
a pairing between  $W_{h}^{\star}\hat{\ot}V_{h}^{\star}$ and $W_{h}\hat{\ot}V_{h}$ by
\vspa
\[\langle w^{\star}\hat{\ot}v^{\star},w_{h}\hat{\ot}v_{h} \rangle_{\bf K}=\langle w^{\star},w_{h}\rangle_{\bf K}\langle v^{\star},v_{h}\rangle_{\bf K},\quad w^{\star}\hat{\ot}v^{\star}\in W_{h}^{\star}\hat{\ot}V_{h}^{\star},w_{h}\hat{\ot}v_{h} \in W_{h}\hat{\ot}V_{h}.\]
 For any $r_h\in W_h\hat{\ot} V_h$, let $(r_h)^{\sharp}$ be the image of $r_h$ in the $\bfK$-linear isomorphism
\vspb
\begin{equation}\label{ih}
W_h\hat{\ot} V_h\cong \textrm{Hom}_{\bfK}(V^{\star}_h,W_h).
\vspb
\end{equation}
It is characterized by
\vspb
\begin{equation}
\langle w^{\star},r_{h}^{\sharp}(v^{\star})\rangle_{\bfK} = \langle  w^{\star}\hat{\ot} v^{\star},r_h \rangle_{\bfK},  \quad v^{\star}\in V^{\star}_h, w^{\star}\in W^{\star}_h.
\vspb
\end{equation}
For $f_h\in\textrm{Hom}_{\bfK}(V^{\star}_h,W_h)$, we also let $\tilde{f_h}$ denote its image in
$W_h\otimes V_h$. This agrees with the $\bfK$-linear isomorphism $W_h\hat{\ot} V_h\cong \textrm{Hom}_{\bfK}(V^{\star}_h,W_h)$ given in (\ref{is}) as follows.
\begin{pro}\label{pro-add}
Let $r$ be an element in $W\otimes V$. Then we have
\vspb
        \begin{equation}\label{eq-sh}
            (\hr)^{\sharp}=(r^{\sharp})_{\bfK}.
\vspb
        \end{equation}
Accordingly, for $f\in\Hom_{\bfk}(V,W)$, we have
    \begin{equation}\label{eq-tk}
\tilde{f_{\bfK}}=\hat{\tilde{f}}.
\end{equation}
\end{pro}
\begin{proof}
Let $r=\sum_{i}w_{i}\otimes v_{i}\in W\ot V$ with $v_i\in V$ and $w_i\in
    W$. Let $v^{\ast}\in V^{\ast}, w^{\star}\in W^{\star}_h$. Then
    we have
\vspb
    \begin{eqnarray*}
        \langle w^{\star},(\hr)^{\sharp}(v^{\ast})\rangle_{\bfK} &=& \langle  w^{\star}\hat{\ot} v^{\ast},\hr \rangle_{\bfK}
   =\sum_{i}\langle  w^{\star},w_i \rangle_{\bfK}\langle v^{\ast},v_i \rangle_{\bfK}
        =\sum_{i}\langle  w^{\star},\langle v^{\ast},v_i \rangle_{\bfK}w_i \rangle_{\bfK}\\
        &=&\langle  w^{\star},\sum_{i} \langle v^{\ast},v_i \rangle w_i \rangle_{\bfK}
        =\langle  w^{\star},r^{\sharp}(v^{\ast}) \rangle_{\bfK}.
\vspd
    \end{eqnarray*}
So $(\hr)^{\sharp}(v^{\ast})=r^{\sharp}(v^{\ast})$ and hence
\vspb
\[(\hr)^{\sharp}(\sum_{s=0}^{\infty}v_s^{\ast}h^s)=\sum_{s=0}^{\infty}r^{\sharp}(v_s^{\ast})h^s,\;\sum_{s=0}^{\infty}v_s^{\ast}h^s\in V^{\star}_h.
\vspb
\]
Thus $(\hr)^{\sharp}$ coincides with the $\bfK$-linear extension of $r^{\sharp}$, which proves (\ref{eq-sh}).
By (\ref{eq-sh}), we have $(\hat{\tilde{f}})^{\sharp}=((\tilde{f})^{\sharp})_{\bfK}.$
Then the last conclusion follows from
$(\tilde{f})^{\sharp}=f$.
\end{proof}
For topologically free $\bf K$-modules $\mathrm{A}_h$ and $V_h$, and a $\bf K$-linear map $\varrho_{h}:\mathrm{A}_h\to \textrm{End}_{\bf K}(V_{h})$, let $\varrho^{\star}_{h}:\mathrm{A}_{h}\to \textrm{End}_{\bf K}(V^{\star}_h)$ be the unique $\bf K$-linear map such that
\begin{equation}\mlabel{star}
\langle \varrho^{\star}_{h}(x^{\prime})u^{\star},v^{\prime}\rangle_{\bf K}=-\langle u^{\star},\varrho_{h}(x^{\prime})v^{\prime}\rangle_{\bf K},  x^{\prime}\in \mathrm{A}_h , v^{\prime}\in V_h , u^{\star}\in V^{\star}_{h}.
\end{equation}
Let $(A_{h},\circ_{h})$ be a topologically free associative algebra and $(V_{h}, l_{\circ_h} , r_{\circ_h})$ be a topologically free $(A_{h},\circ_{h})$-bimodule. Then $(V^{\star}_h,-r^{\star}_{\circ_ h},-l^{\star}_{\circ_ h})$ is a topologically free $(A_{h},\circ_{h})$-bimodule, called the \textbf{topological dual bimodule} of $(V_{h}, l_{\circ_h} , r_{\circ_h})$.

\begin{pro}\label{dedu}
Let $(A,\circ)$ be an associative algebra and $(V, l_{\circ}, r_{\circ})$ be an $(A,\circ)$-bimodule. Let $(A_{h},\circ_{h})$ be an associative deformation of  $(A,\circ)$ and $(V_{h}, l_{\circ_h} , r_{\circ_h})$ be an $(A_{h},\circ_{h})$-bimodule deformation of $(V, l_{\circ}, r_{\circ})$. Then  $(V^{\star}_h,-r^{\star}_{\circ_h},-l^{\star}_{\circ_ h})$ is an $(A_{h},\circ_{h})$-bimodule deformation of the $(A,\circ)$-bimodule $(V,-r_{\circ}^{\ast},-l_{\circ}^{\ast})$.
\end{pro}
\begin{proof}
 Let $x\in A, u_{h}\in V_h$ and $v^{\ast}\in V^{\ast}$. Then we have
 \begin{eqnarray*}
    \langle -r_{\circ_{h}}^{\star}(x)v^{\ast},u_{h}\rangle_{\bf K}
    &&=\langle v^{\ast},r_{\circ_{h}}(x)u_{h}\rangle_{\bf K}\\
    &&\equiv\langle v^{\ast},r_{\circ}(x)u\rangle\mmod{h}\\
    &&\equiv\langle -r_{\circ}^{\ast}(x)v^{\ast},u\rangle\mmod{h}\\
        &&\equiv\langle -r_{\circ}^{\ast}(x)v^{\ast},u_{h}\rangle_{\bf K}\mmod{h},
 \end{eqnarray*}
where $u$ is the element in $V$ such that $u_{h}\equiv u\quad(\mathrm{mod}~h)$.
Then we deduce that $-r_{\circ_{h}}^{\star}(x)\equiv-r_{\circ}^{\ast}(x)\mmod{h}$. Similarly, we can deduce that $-l_{\circ_{h}}^{\star}(x)\equiv-l_{\circ}^{\ast}(x)\mmod{h}$.
Thus $(V^{\star}_h,-r^{\star}_{\circ h},-l^{\star}_{\circ_ h})$ is an $(A_{h},\circ_{h})$-bimodule deformation of the $(A,\circ)$-bimodule $(V^{\ast},-r_{\circ}^{\ast},-l_{\circ}^{\ast})$.
\end{proof}
The following conclusion shows that the processes of taking the
dual modules and taking the \qcls commute.
\begin{pro}\label{cldu}
Let $(A,\circ)$ be a commutative associative algebra and $(V,\rho_{\circ})$ be an $(A,\circ)$-module.  Let $(A_{h},\circ_{h})$ be an associative deformation of $(A,\circ)$ and
$(A,\{,\},\circ)$ be the Poisson algebra as the \qcl. Let $(V_{h},l_{\circ_h},r_{\circ_h})$ be the $A_{h}$-bimodule deformation of $(V,\rho_{\circ})$ and $(V,\rho_{\{,\}},\rho_{\circ})$ be the $(A,\{,\},\circ)$-module as the corresponding \qcl. Then the \qcl of the
$A_{h}$-bimodule deformation $(V^{\star}_h,-r^{\star}_{\circ_ h},-l^{\star}_{\circ_ h})$ of the $(A,\circ)$-module $(V^{\ast},-\rho_{\circ}^{\ast})$ is exactly the $(A,\{,\},\circ)$-module $(V^{\ast},\rho_{\{,\}}^{\ast},-\rho_{\circ}^{\ast})$ from taking the dual module of the $(A,\{,\},\circ)$-module $(V,\rho_{\{,\}},\rho_{\circ})$
\end{pro}
\begin{proof}
    Let $x\in A, u^{\ast}\in V^{\ast},
    v_{h}\in V_h$. Then we have
\vspb
    \begin{eqnarray*}
 \langle -r_{\circ_{h}}^{\star}(x)u^{\ast}, v_{h}\rangle_{\bf K}-\langle -l_{\circ_{h}}^{\star}(x)u^{\ast}, v_{h}\rangle_{\bf K}
        &=&\langle u^{\ast},r_{\circ_{h}}(x) v_{h}-l_{\circ_{h}}(x)v_{h}\rangle_{\bf K}\\
        &\equiv&\langle u^{\ast},-\rho_{\{,\}}(x)v\rangle~h\mmod{h^2}\\
            &\equiv&\langle \rho_{\{,\}}^{\ast}(x)u^{\ast},v\rangle~h\mmod{h^2}\\
                &\equiv&\langle \rho_{\{,\}}^{\ast}(x)u^{\ast}, v_{h}\rangle_{\bf K}~h\mmod{h^2},
\vspb
    \end{eqnarray*}
where $v$ is the element in $V$ such that $ v_{h}\equiv v\quad(\mathrm{mod}~h)$.
    Hence we obtain \[\frac{-r_{\circ_{h}}^{\star}(x)-(-l_{\circ_{h}}^{\star}(x))}{h}\equiv \rho_{\{,\}}^{\ast}(x)\ (\mathrm{mod}~h).\]
    Thus we conclude that the \qcl of the
    $A_{h}$-bimodule deformation $(V^{\star}_h,-r^{\star}_{\circ_ h},-l^{\star}_{\circ_ h})$ is exactly the $(A,\{,\},\circ)$-module $(V^{\ast},\rho_{\{,\}}^{\ast},-\rho_{\circ}^{\ast})$.
\end{proof}
\vspb
\subsection{Scalarly deformed solutions of the AYBE and scalarly deformed $\calo$-operators}\label{css}
We first recall a basic equivalent relation between solutions of the AYBE and $\calo$-operators.
\vspb
 \begin{pro}{\rm{\cite [Corollary 3.6]{BGN2}}}
Let $(A,\circ)$ be a finite-dimensional associative algebra and $r\in A\otimes A$. Let
$\beta:=(r+\sigma(r))/2$ be $(A,\circ)$-invariant in the sense of
\meqref{eq:inv}. Let $r^\sharp$ and $\beta^\sharp$ be the images
of $r$ and $\beta$ under $A\otimes A\cong
\mathrm{Hom}_{\bfk}(A^{\ast},A)$ $($see \meqref{eq:tendual}$)$.
Then \vspb
   \begin{equation*}
    R_{\circ}^{\ast}(\beta^{\sharp}(a^{\ast}))b^{\ast}=L_{\circ}^{\ast}(\beta^{\sharp}(b^{\ast}))a^{\ast},\quad  a^{\ast},b^{\ast}\in A^{\ast}.
\vspa
   \end{equation*}
   Moreover, define a binary operation $\cdot$ on $A^{\ast}$ by
\vspb
    \begin{equation}
    a^{\ast}\cdot b^{\ast}:=2R_{\circ}^{\ast}(\beta^{\sharp}(a^{\ast}))b^{\ast},\quad  a^{\ast},b^{\ast}\in A^{\ast}.
\vspa
    \end{equation}
  Then $(A^{\ast},\cdot)$ is an associative algebra and $(A^{\ast},\cdot,-R_{\circ}^{\ast},-L_{\circ}^{\ast})$ is an $(A,\circ)$-bimodule algebra.
    Furthermore, $r$ is a solution of the AYBE in $(A,\circ)$ if and only if $r^{\sharp}$ is an $\mathcal{O}$-operator of weight $1$ on $(A,\circ)$ associated to $(A^{\ast},\cdot,-R_{\circ}^{\ast},-L_{\circ}^{\ast})$.
\mlabel{aw1}
\end{pro}
\begin{rmk}\label{cam}
    When the associative algebra $(A,\circ)$ in the above statement is commutative, then
    $(A^{\ast},\cdot)$ is commutative and $(A^{\ast},\cdot,-R_{\circ}^{\ast})$ is an $(A,\circ)$-module algebra. Note that in this case, $L_{\circ}^{\ast}=R_{\circ}^{\ast}$.
\end{rmk}

The following is the analogue of Proposition~\mref{aw1} for
topologically free associative algebras and can be proved in the
same way (see the proof of \cite [Corollary 3.6]{BGN2}).
\begin{pro}
Let $A$ be a finite-dimensional vector
space and $(A_{h},\circ_{h})$ be a topologically free associative
algebra. Let $r_h\in A_h\hat{\otimes} A_h$ and
$\beta_h:=(r_h+\sigma(r_h))/2$ be the symmetric part of $r_h$.
Denote the isomorphic images of $r_h$ and $\beta_h$ in
$\mathrm{Hom}_{\bfK}(A^{\star}_h,A_h)$ by $(r_h)^{\sharp}$ and
$(\beta_h)^{\sharp}$ respectively.  Suppose that $\beta_h$ is
$(A_{h},\circ_{h})$-invariant. Then \vspb
\begin{equation}\label{lr}
R_{\circ_{h}}^{\star}((\beta_h)^{\sharp}(a^{\star}))b^{\star}=L_{\circ_{h}}^{\star}((\beta_h)^{\sharp}(b^{\star}))a^{\star},\quad  a^{\star},b^{\star}\in A^{\star}_h.
\vspb
\end{equation}
Moreover, define a binary operation $\cdot_h$ on $A^{\star}_h$ by
\vspb
\begin{equation}
a^{\star}\cdot_h b^{\star}:=2R_{\circ_h}^{\star}((\beta_h)^{\sharp}(a^{\star}))b^{\star},\quad  a^{\star},b^{\star}\in A^{\star}_h.
\vspb
\end{equation}
Then $(A^{\star}_{h},\cdot_{h})$ is a topologically free associative algebra and  $(A^{\star}_{h},\cdot_{h},-R_{\circ_{h}}^{\star},-L_{\circ_{h}}^{\star})$ is a topologically free $(A_{h},\circ_{h})$-bimodule algebra.
Furthermore, $r_h$ is a solution of the AYBE in $(A_{h},\circ_{h})$ if and only if $(r_h)^{\sharp}$ is an $\mathcal{O}$-operator of weight $1$ on $(A_h ,\circ_h)$
associated to $(A^{\star}_{h},\cdot_{h},-R_{\circ_{h}}^{\star},-L_{\circ_{h}}^{\star})$.
\mlabel{taw1}
\end{pro}

The following is the PYBE analogue of Proposition~\mref{aw1}.
\begin{thm}{\rm{\mcite{NB}}}
Let $(A,\{,\},\circ)$ be a finite-dimensional Poisson algebra and $r\in A\otimes A$.
Suppose that $\beta:=(r+\sigma(r))/2$ is
$(A,\{,\},\circ)$-invariant. Then \vspb
\begin{equation}
[a^{\ast},b^{\ast}]:=-2\rad_{\{,\}}^{\ast}(\beta^{\sharp}(a^{\ast}))b^{\ast},\quad a^{\ast}\cdot b^{\ast}:=2R_{\circ}^{\ast}(\beta^{\sharp}(a^{\ast}))b^{\ast},\quad  a^{\ast},b^{\ast}\in A^{\ast},
\vspb
\end{equation}
define a Poisson algebra $(A^{\ast},[,],\cdot)$, and $(A^{\ast},[,],\cdot,\rad_{\{,\}}^{\ast},-R_{\circ}^{\ast})$ is an $(A,\{,\},\circ)$-module Poisson algebra.
Furthermore, $r$ is a solution of the PYBE in $(A,\{,\},\circ)$ if and only if $r^{\sharp}$ is an $\mathcal{O}$-operator of weight $1$ on $(A,\{,\},\circ)$ associated to the $(A,\{,\},\circ)$-module $(A^{\ast},[,],\cdot,\rad_{\{,\}}^{\ast},-R_{\circ}^{\ast})$.
\mlabel{pw1}
\vspb
\end{thm}
Let $(A,\circ)$ be a finite-dimensional commutative associative algebra, $(A_{h},\circ_{h})$ be an
associative deformation and $(A,\{,\},\circ)$ be the corresponding
\qcl. Let $r$ be in $A\otimes A$ and $\beta=(r+\sigma(r))/2$ be
its symmetric part.
 Denote the isomorphic images of $r$ and $\beta$ in $\mathrm{Hom}_{\bfk}(A^{\ast},A)$ by $r^{\sharp}$ and $\beta^{\sharp}$ respectively.
 Denote the isomorphic images of $r$ and $\beta$ under the linear embedding of $A\otimes A$ into $A_{h}\hat{\otimes}A_{h}$ by $\hr$ and $\hat{\beta}$ respectively.
Suppose that $\hat{\beta}$ is $(A_{h},\circ_{h})$-invariant.
 By Proposition \ref{taw1}, $(A^{\star}_{h},\cdot_{h})$ is a topologically free associative algebra and $(A^{\star}_{h},\cdot_{h},-R_{\circ_{h}}^{\star},-L_{\circ_{h}}^{\star})$ is a
 topologically free $(A_{h},\circ_{h})$-bimodule algebra, where
\vspc
 \[ a^{\star}\cdot_h b^{\star}:=2R_{\circ_h}^{\star}((\hat{\beta})^{\sharp}(a^{\star}))b^{\star},\quad  a^{\star},b^{\star}\in A^{\star}_h.
 \vspb
 \]
 By Remark~\ref{rmkinvyb}, $\beta$ is $(A,\{,\},\circ)$-invariant. Thus by Remark \ref{cam},
   $(A^{\ast},\cdot,-R_{\circ}^{\ast})$ is an $(A,\circ)$-module algebra, where
\vspc
\[ a^{\ast}\cdot b^{\ast}:=2R_{\circ}^{\ast}(\beta^{\sharp}(a^{\ast}))b^{\ast},\quad  a^{\ast},b^{\ast}\in A^{\ast}.
\vspb
\]
  By Theorem \ref{pw1},
   $(A^{\ast},[,],\cdot,\rad_{\{,\}}^{\ast},-R_{\circ}^{\ast})$ is an $(A,\{,\},\circ)$-module Poisson algebra, where
  \[ [a^{\ast},b^{\ast}]:=-2\rad_{\{,\}}^{\ast}(\beta^{\sharp}(a^{\ast}))b^{\ast},\quad a^{\ast}\cdot b^{\ast}:=2R_{\circ}^{\ast}(\beta^{\sharp}(a^{\ast}))b^{\ast},\quad  a^{\ast},b^{\ast}\in A^{\ast}.\]
 \begin{thm}
    Keep the assumptions and notations as above.
\begin{enumerate}
\item \label{w1-1} The topologically free $(A_{h},\circ_{h})$-bimodule algebra $(A^{\star}_{h},\cdot_{h},-R_{\circ_{h}}^{\star},-L_{\circ_{h}}^{\star})$ is an $(A_{h},\circ_{h})$-bimodule algebra deformation of the $(A,\circ)$-module algebra $(A^{\ast},\cdot,-R_{\circ}^{\ast})$, and its \qcl is exactly the $(A,\{,\},\circ)$-module Poisson algebra $(A^{\ast},[,],\cdot,\rad_{\{,\}}^{\ast},-R_{\circ}^{\ast})$.
\item \label{w1-2}
Take the downward arrows in the following diagram. Then starting from a solution $r$ of the AYBE
 in $(A,\circ)$, the left square in the diagram commutes.
Moreover, starting from a solution $\hr$ of the AYBE
 in $(A_{h},\circ_{h})$, the right square in the diagram commutes.
    \item \label{w1-3}
Take the upward arrows in the following diagram. Then starting from the $\mathcal{O}$-operator $r^{\sharp}$ of weight $1$ on $(A,\circ)$ associated to $(A^*,\cdot,-R^*_{\circ})$, the left square of the diagram commutes. Moreover, starting from the $\mathcal{O}$-operator $(\hr)^{\sharp}$ of weight $1$ on $(A_h ,\circ_h)$ associated to $(A^{\star}_{h},\cdot_{h},-R_{\circ_{h}}^{\star}, -L_{\circ_{h}}^{\star})$, the right square of the diagram
commutes.
\end{enumerate}
   \begin{equation}\label{diagw1}
   \begin{split}
  \xymatrixcolsep{5.5pc}\xymatrixrowsep{3pc}
  \xymatrix{
  *\txt{\small solution $r$ of\\ \small AYBE in $(A,\circ)$,\\
\small $(A,\circ)$-invariant $\beta$}\ar[r]^{\txt{\tiny deformation}}\ar@<-2pt>[d]|-{\txt{\tiny Proposition~\mref{aw1}}}&*\txt{\small solution $\hr$ of\\ \small AYBE in $(A_{h},\circ_{h})$,\\ \small $(A_{h},\circ_{h})$-invariant $\hat{\beta}$}\ar[r]^{\txt{\tiny QCL}}_{\txt{\tiny Theorem~\mref{inv-aybe}}}\ar@<-2pt>[d]|-{\txt{\tiny Proposition~\mref{taw1}}}&*\txt{\small solution $r$ of\\ \small PYBE in $(A,\{,\},\circ)$,\\ \small $(A,\{,\},\circ)$-invariant $\beta$}\ar@<-2pt>[d]|-{\txt{\tiny Theorem \mref{pw1}}}\\
  *\txt{\small $\mathcal{O}$-operator $r^{\sharp}$ on\\ \small $(A,\circ)$ associated\\ \small to $(A^{\ast},\cdot,-R_{\circ}^{\ast})$}\ar[r]^{\txt{\tiny deformation}}\ar@<-2pt>[u]|-{\phantom{\txt{\footnotesize Theorem}}}
   & *\txt{\small $\mathcal{O}$-operator $(\hat{r})^{\sharp}=(r^\sharp)_{\bf K}$ on\\ \small $(A_h,\circ_{h})$ associated\\ \small to $(A^{\star}_{h},\cdot_{h},-R_{\circ_{h}}^{\star},-L_{\circ_{h}}^{\star})$}\ar@<-2pt>[u]|-{\phantom{\txt{\footnotesize Theorem}}}\ar[r]^{\txt{\tiny QCL}}_{\txt{\tiny Theorem~\mref{invrb}}}
   & *\txt{\small  $\mathcal{O}$-operator $r^{\sharp}$ on\\ \small $(A,\{,\},\circ)$ associated\\ \small to $(A^{\ast},[,],\cdot,\rad_{\{,\}}^{\ast},-R_{\circ}^{\ast})$.}\ar@<-2pt>[u]|-{\phantom{\txt{\footnotesize Theorem}}}
}
\end{split}
   \end{equation}
\mlabel{pro-w1}
   \end{thm}
   \begin{proof}
    (\ref{w1-1}). Note that the topologically free $(A_h ,\circ_{h})$-bimodule  $(A^{\star}_{h},-R_{\circ_{h}}^{\star},-L_{\circ_{h}}^{\star})$ is
    the topological dual bimodule of the $(A_h ,\circ_{h})$-bimodule
  $(A_h ,L_{\circ_{h}},R_{\circ_h})$, and $(A_h ,L_{\circ_{h}},R_{\circ_h})$
  is an  $(A_h ,\circ_{h})$-bimodule deformation of the $(A,\circ)$-module
  $(A,L_\circ)$.
  Then by Proposition~\ref{dedu} and $R_{\circ}^{\ast}=L_{\circ}^{\ast}$, we
  obtain
\vspb
    $$-R_{\circ_{h}}^{\star}(x)\equiv-L_{\circ_{h}}^{\star}(x)\equiv -R_{\circ}^{\ast}(x)\mmod{h},  x\in A.
    \vspb
    $$
    By (\ref{eq-sh}), we have
$(\hat{\beta})^{\sharp}=(\beta^{\sharp})_{\bfK}.$
   Hence, for $a^{\ast},b^{\ast}\in A^{\ast}$, we have
\vspb
   \[a^{\ast}\cdot_h b^{\ast}=2R_{\circ_h}^{\star}((\hat{\beta})^{\sharp}(a^{\ast}))b^{\ast}=2R_{\circ_h}^{\star}((\beta^{\sharp})_{\bfK}(a^{\ast}))b^{\ast}\equiv 2R_{\circ}^{\ast}(\beta^{\sharp}(a^{\ast}))b^{\ast}\mmod{h}\equiv a^{\ast}\cdot b^{\ast}\mmod{h}.
   \vspb
   \]
     Thus the topologically free $(A,\circ_{h})$-bimodule algebra $(A^{\star}_{h},\cdot_{h},-R_{\circ_{h}}^{\star},-L_{\circ_{h}}^{\star})$ is an $(A,\circ_{h})$-bimodule algebra deformation of the $(A,\circ)$-module algebra $(A^{\ast},\cdot,-R_{\circ}^{\ast})$.

Let $(A^{\ast},[,]^{\prime},\cdot,\rho_{\{,\}}^{\prime},-R_{\circ}^{\ast})$ denote the \qcl of $(A^{\star}_{h},-R_{\circ_{h}}^{\star},-L_{\circ_{h}}^{\star})$.
    By Proposition \ref{cldu}, the \qcl
     of the $(A,\circ_{h})$-bimodule deformation $(A^{\star}_{h},-R_{\circ_{h}}^{\star},-L_{\circ_{h}}^{\star})$ is exactly
     the $(A,\{,\},\circ)$-module  $(A^{\ast},\cdot,\rad_{\{,\}}^{\ast},-R_{\circ}^{\ast})$. Thus $\rad_{\{,\}}^{\ast}=\rho_{\{,\}}^{\prime}$.
     By (\ref{lr}),
    $R_{\circ_{h}}^{\star}((\hat{\beta})^{\sharp}(a^{\star}))b^{\star}=L_{\circ_{h}}^{\star}((\hat{\beta})^{\sharp}(b^{\star}))a^{\star}$ for $a^{\star},b^{\star}\in A^{\star}_h$.
Then we deduce
\vspb
\begin{eqnarray*}
\frac{a^{\ast}\cdot_{h}b^{\ast}-b^{\ast}\cdot_{h}a^{\ast}}{h}&=&\frac{2R_{\circ_{h}}^{\star}((\hat{\beta})^{\sharp}(a^{\ast}))b^{\ast}-2R_{\circ_{h}}^{\star}((\hat{\beta})^{\sharp}(b^{\ast}))a^{\ast}}{h} \\
&=& \frac{2R_{\circ_{h}}^{\star}(\beta^{\sharp}(a^{\ast}))b^{\ast}-2L_{\circ_{h}}^{\star}(\beta^{\sharp}(a^{\ast}))b^{\ast}}{h} \\
&\equiv& -2\rad_{\{,\}}^{\ast}(\beta(a^{\ast}))b^{\ast}\mmod{h},\\
&\equiv& [a^{\ast},b^{\ast}]\mmod{h}\quad  a^{\ast},b^{\ast}\in A^{\ast}.
\vspb
\end{eqnarray*}
Hence $[a^{\ast},b^{\ast}]=[a^{\ast},b^{\ast}]^{\prime}$.
Therefore, the two $(A,\{,\},\circ)$-module Poisson algebras
$(A^{\ast},[,],\cdot$, $\rad_{\{,\}}^{\ast}$, $-R_{\circ}^{\ast})$
and
$(A^{\ast},[,]^{\prime},\cdot,\rho_{\{,\}}^{\prime},-R_{\circ}^{\ast})$
coincide.

Since $(\hr)^{\sharp}$ coincides with the scalar deformation $(r^{\sharp})_{\bf K}$ of $r^{\sharp}$ by Proposition~\ref{pro-add}, the statements in (\ref{w1-2}) and (\ref{w1-3}) follow.
\vspb
\end{proof}
%\sy{The skew-symmetric case is put here.}
\begin{rmk}
    Taking the tensor $r$ in Theorem~\ref{pro-w1} to be skew-symmetric leads to the following commutative diagram.
\begin{equation*}
\xymatrixcolsep{4pc}\xymatrixrowsep{2pc}
\xymatrix{
    *\txt{\small skew-symmetric\\ \small solution $r$ of\\ \small AYBE in $(A,\circ)$}\ar[r]^{\txt{\tiny deformation}}\ar@<-2pt>[d]&*\txt{\small skew-symmetric\\ solution $\hr$ of \\ AYBE in $(A_{h},\circ_{h})$}\ar[r]^{\txt{\tiny QCL}}\ar@<-2pt>[d]|-{\txt{\footnotesize }}&*\txt{\small skew-symmetric\\ \small solution $r$ of\\ \small PYBE in $(A,\{,\},\circ)$}\ar@<-2pt>[d]|-{\txt{\footnotesize }}\\
    *\txt{\small $\mathcal{O}$-operator $r^{\sharp}$ on\\ \small $(A,\circ)$ associated\\ \small to $(A^{\ast},-R_{\circ}^{\ast})$}\ar[r]^{\txt{\tiny deformation}}\ar@<-2pt>[u]
    & *\txt{\small  $\mathcal{O}$-operator $(\hr)^{\sharp}=(r^\sharp)_{\bf K}$ on\\ \small $(A_{h},\circ_{h})$ associated \\ \small to $(A^{\star}_{h},-R_{\circ_{h}}^{\star},-L_{\circ_{h}}^{\star})$}
    \ar[r]^{\txt{\tiny QCL}}\ar@<-2pt>[u]
    & *\txt{\small $\mathcal{O}$-operator $r^{\sharp}$ on \\ \small $(A,\{,\},\circ)$ associated\\ \small to $(A^{\ast},\rad_{\{,\}}^{\ast},-R_{\circ}^{\ast})$.}\ar@<-2pt>[u]
}
\end{equation*}
\vspb
\end{rmk}
\subsection{AYBE and PYBE via $\calo$-operators of weight zero and dendriform deformations}
\mlabel{ss:ybeopdend}
We first obtain solutions of PYBE induced by scalar deformations of $\mathcal{O}$-operators of weight zero. Then through the relation between dendriform algebras and $\calo$-operators, we also obtain solutions of the PYBE from deformations of dendriform algebras.

For vector spaces $V$ and $W$, we have the canonical embedding  \vspb
\begin{equation}
\eta: V\ot W \to (V\oplus W)\ot (V\oplus W),  \quad \sum_{i}v_i\ot w_i \mapsto \sum_{i}(v_i,0)\ot (0,w_i)\in (V\oplus W)\ot (V\oplus W).
\mlabel{eq:sumembed}
\vspb
\end{equation}
Similarly, we can  embed $V_h\hat{\ot}W_h$ into
$(V_h\oplus W_h)\hat{\ot}(V_h\oplus W_h)$. Then
identifying $V_h\oplus W_h$ with $(V\oplus W)_h$ by
Remark~\ref{ids}, we obtain the following $\bfK$-linear embedding:
\vspb
\begin{equation}\label{eq:hsumembed}
    \eta_h:V_h\hat{\ot}W_h\to (V\oplus W)_h\hat{\ot}(V\oplus W)_h,\quad \sum_{s\geq0}v_sh^s\hat{\ot}\sum_{l\geq0}w_lh^l \mapsto\sum_{s,l\geq0}(v_s,0)h^s\hat{\ot}(0,w_l)h^{l}.
\vspb
\end{equation}
\vspb

 \begin{pro}{\rm{\mcite{BGN2}}}
   Let $(A,\circ)$ be a finite-dimensional  associative algebra and $(V,l_{\circ},r_{\circ})$ be a finite-dimensional  $A$-bimodule. Let $T:V \rightarrow A$ be a linear map. Denote its isomorphic image in $A\ot V^{\ast}$ by $\tilde{T}$ and embed $\tilde{T}$
   into $(A\oplus V^{\ast})^{\ot 2}$ by $\eta$ defined by {\rm (\ref{eq:sumembed})}. Then $r=\eta(\tilde{T})-\sigma(\eta(\tilde{T}))$ is a skew-symmetric solution of the AYBE in the associative algebra $A \ltimes_{-r_{\circ}^{\ast},-l_{\circ}^{\ast}}V^{\ast}$
if and only if $T$ is an $\mathcal{O}$-operator on the associative
algebra $(A,\circ)$ associated to the $(A,\circ)$-bimodule $(V,l_{\circ},r_{\circ})$.
\mlabel{aybe-semi}
 \end{pro}
\vspb

 The topologically free associative algebra analogue of the above
 conclusion is given as follows by omitting the proof.
\begin{pro}
Let $A$ and $V$ be  finite-dimensional  vector spaces. Let $(A_{h},\circ_{h})$ be a topologically free associative algebra and $(V_{h},l_{\circ_{h}},r_{\circ_{h}})$ be a topologically free $A_{h}$-bimodule.
Let $T_h\in \mathrm{Hom}_{\bfK}(V_h,A_h)$ be a $\bfK$-linear map.
Embed $\tilde{T_h}\in A_h\otimes V^{\star}_h$
into $(A\oplus V^{\ast})_{h}\hat{\ot}(A\oplus V^{\ast})_{h}$ by $\eta_h$ defined by \eqref{eq:hsumembed}. Set that $r_h=\eta_h(\tilde{T_h})-\sigma_{\bfK}(\eta_h(\tilde{T_h}))$. Then $r_h$ is a skew-symmetric solution
    of AYBE in $A_{h} \ltimes_{-r_{\circ_{h}}^{\star},-l_{\circ_{h}}^{\star}} V_{h}^{\star}$ if and only if $T_{h}$ is an $\mathcal{O}$-operator on  $(A_{h},\circ_{h})$ associated to $(V_{h},l_{\circ_{h}},r_{\circ_{h}})$.
\mlabel{taybe-semi}
\end{pro}
There is also  a Poisson algebra analogue of Proposition
\ref{aybe-semi} as follows.

\begin{thm}{\rm{\mcite{LBS}}}
Let $(A,\{,\},\circ)$ be a finite-dimensional
 Poisson algebra and
$(V,\rho_{\{,\}},\rho_{\circ})$ be a finite-dimensional
 $(A,\{,\},\circ)$-module. Let $T:V
\rightarrow A$ be a linear map. Denote its isomorphic image in
$A\ot V^{\ast}$ by $\tilde{T}$ and embed $\tilde{T}$ into
$(A\oplus V^{\ast})^{\ot 2}$ by $\eta$ defined by
\eqref{eq:sumembed}. Then
$r=\eta(\tilde{T})-\sigma(\eta(\tilde{T}))$ is a skew-symmetric
solution of the PYBE in the Poisson algebra $A
\ltimes_{\rho_{\{,\}}^{\ast},-\rho_{\circ}^{\ast}}V^{\ast}$ if and
only if $T$ is an $\mathcal{O}$-operator on the Poisson algebra
$(A,\{,\},\circ)$ associated to the $(A,\{,\},\circ)$-module
$(V,\rho_{\{,\}},\rho_{\circ})$. \mlabel{pybe-semi}
\end{thm}
Let $(A,\circ)$ be a finite-dimensional commutative associative algebra and $(V,\rho_{\circ})$ be a finite
dimensional  $(A,\circ)$-module. Let
$(A_{h},\circ_{h})$ be an associative deformation of $(A,\circ)$
and $(A,\{,\},\circ)$ be the corresponding \qcl. Let
$(V_{h},l_{\circ_{h}},r_{\circ_{h}})$ be an $A_{h}$-bimodule
deformation of $(V,\rho_{\circ})$ and
$(V,\rho_{\{,\}},\rho_{\circ})$ be the corresponding \qcl, which
is an $(A,\{,\},\circ)$-module by Corollary \mref{cor-defmod}.
Then we have the following data:
\begin{itemize}
\item the $(A,\circ)$-module $(V^{\ast},-\rho_{\circ}^{\ast})$
which is the dual module of the $(A,\circ)$-module
$(V,\rho_{\circ})$ and a commutative associative algebra $A
\ltimes_{-\rho_{\circ}^{\ast}}V^{\ast}$; \item the topologically free
$(A_{h},\circ_{h})$-bimodule
$(V_{h}^{\star},-r_{\circ_{h}}^{\star},-l_{\circ_{h}}^{\star})$
which is the topological dual module of the
$(A_{h},\circ_{h})$-bimodule $(V_{h},l_{\circ_{h}},r_{\circ_{h}})$
and a topologically free associative algebra \\ $A_{h}
\ltimes_{-r_{\circ_{h}}^{\star},-l_{\circ_{h}}^{\star}}
V_{h}^{\star}$; \item the
 $(A,\{,\},\circ)$-module  $(V^{\ast},\rho_{\{,\}}^{\ast},-\rho_{\circ}^{\ast})$ which is the dual module of the  $(A,\{,\},\circ)$-module  $(V,\rho_{\{,\}},\rho_{\circ})$.
 and a Poisson algebra $A\ltimes_{\rho_{\{,\}}^{\ast},-\rho_{\circ}^{\ast}}V^{\ast}$.
\end{itemize}
 Let $T\in \mathrm{Hom}_{\bf k}(V,A)$ and $T_{\mathbf{K}}$ be defined by  (\ref{eq:TK}).
 Embed $\tilde{T}\in A\otimes V^{\ast}$
 into $(A\oplus V^{\ast})^{\ot 2}$ via $\eta$ defined by (\ref{eq:sumembed}). Then set $r:=\eta(\tilde{T})-\sigma(\eta(\tilde{T}))$. So $r$ is skew-symmetric.
 Denote the isomorphic images of $\tilde T$ and $r$ in $(A\oplus V^{\ast})_{h}\hat{\ot}(A\oplus V^{\ast})_h$ by $\hat{\tilde{T}}$ and $\hr$ respectively. We obtain
$\hr=\widehat{\eta(\tilde{T})}-\widehat{\sigma(\eta(\tilde{T}))}.$
 Embed $\widetilde{T_\bfK}\in A_h\otimes V^{\star}_h$
    into $(A\oplus V^{\ast})_{h}\hat{\ot}(A\oplus V^{\ast})_{h}$ via $\eta_h$ defined by (\ref{eq:hsumembed}).
Set
\vspb
$$r_h:=\eta_h(\widetilde{T_{\bfK}})-\sigma_{\bf K}(\eta_h(\widetilde{T_{\bfK}})).
\vspb
$$
 \begin{thm}\label{pro-skews}
    Keep the above assumptions and notations. We have the following results.
    \begin{enumerate}
        \item\label{skews1} The topologically free associative algebra $A_{h} \ltimes_{-r_{\circ_{h}}^{\star},-l_{\circ_{h}}^{\star}} V_{h}^{\star}$
        is an associative deformation of the commutative associative algebra
        $A \ltimes_{-\rho_{\circ}^{\ast}}V^{\ast}$ and the corresponding \qcl is exactly the Poisson algebra $A\ltimes_{\rho_{\{,\}}^{\ast},-\rho_{\circ}^{\ast}}V^{\ast}$.
       \item\label{skews2} Assume that $T$ and $T_{\bf K}$ are $\mathcal{O}$-operators on $(A,\circ)$ and $(A_{h},\circ_{h})$ associated to $(V,l_{\circ},r_{\circ})$ and $(V_{h},l_{\circ_{h}},r_{\circ_{h}})$ respectively. Starting with an $\mathcal{O}$-operator $T$ on $(A,\circ)$ associated to $(V,l_{\circ},r_{\circ})$, the left square of the following diagram commutes.
Moreover, starting from the $\mathcal{O}$-operator $T_{\bf K}$ on
$(A_{h},\circ_{h})$ associated to
$(V_{h},l_{\circ_{h}},r_{\circ_{h}})$, the right square of the
diagram commutes.
\vspb
 \end{enumerate}
{{\begin{displaymath}
  \xymatrixcolsep{5pc}\xymatrixrowsep{4pc}
  \xymatrix{
  *\txt{\small $\mathcal{O}$-operator $T$ on\\ \small $(A,\circ)$ associated \\ \small to $(V,\rho_{\circ})$}\ar[d]|-{\txt{\tiny Proposition \mref{aybe-semi}\\\tiny $r=\eta(\tilde{T})-\sigma(\eta(\tilde{T}))$}}\ar[r]^{\txt{\tiny deformation}}
    & *\txt{\small   $\mathcal{O}$-operator $T_{\mathbf{K}}$ on\\ \small $(A_{h},\circ_{h})$ associated\\ \small to $(V_{h},l_{h},r_{h})$}\ar[r]^{\txt{\tiny QCL}}_{\txt{\tiny Theorem \mref{invrb}
        }}\ar[d]|-{\txt{\tiny Proposition \mref{taybe-semi}\\\tiny$r_h=\eta_h(\widetilde{T_{\bfK}})-\sigma_{\bf K}(\eta_h(\widetilde{T_{\bfK}}))$ }}
   & *\txt{\small $\mathcal{O}$-operator $T$ on\\ \small $(A,\{,\},\circ)$ associated\\ \small to $(V,\rho_{\{,\}},\rho_{\circ})$}\ar[d]|-{\txt{\tiny Theorem \mref{pybe-semi}\\\tiny$r=\eta(\tilde{T})-\sigma(\eta(\tilde{T}))$}}\\
    *\txt{\small skew-symmetric\\ \small  solution $r$ of \\ \small AYBE in $A\ltimes_{-\rho_{\circ}^{\ast}}V^{\ast}$}\ar[r]^{\txt{\tiny deformation}}
   &*\txt{\small skew-symmetric\\ \small solution $r_h$ of\\ \small AYBE in $A_{h} \ltimes_{-r_{\circ_{h}}^{\star},-l_{\circ_{h}}^{\star}} V_{h}^{\star}$}\ar[r]^{\txt{\tiny QCL}}_{\txt{\tiny Theorem~\mref{inv-aybe}}}
   &*\txt{\small skew-symmetric\\ \small solution $r$ of\\ \small PYBE in $A\ltimes_{\rho_{\{,\}}^{\ast},-\rho_{\circ}^{\ast}}V^{\ast}$}
}
 \end{displaymath}}}
 \end{thm}
%\sy{The proof is revised.}
\begin{proof}
(\ref{skews1}).  Since the topologically free $(A_{h},\circ_{h})$-bimodule $(V_{h},l_{\circ_{h}},r_{\circ_{h}})$ is an $A_h$-bimodule deformation of the $(A,\circ)$-module $(V,\rho_{\circ})$, the topologically free
$(A_{h},\circ_{h})$-bimodule
$(V_{h}^{\star},-r_{\circ_{h}}^{\star},-l_{\circ_{h}}^{\star})$ is an $A_h$-bimodule deformation of the $(A,\circ)$-module $(V^{\ast},-\rho_{\circ}^{\ast})$ by Proposition~\ref{dedu}.
Thus $A_{h}
\ltimes_{-r_{\circ_{h}}^{\star},-l_{\circ_{h}}^{\star}}
V_{h}^{\star}$ is an associative deformation of
$A
\ltimes_{-\rho_{\circ}^{\ast}}V^{\ast}$.

By Proposition \ref{cldu},
  the \qcl corresponding to the $A_{h}$-bimodule deformation $(V_{h}^{\star},-r_{\circ_{h}}^{\star},-l_{\circ_{h}}^{\star})$ is exactly the $(A,\{,\},\circ)$-module $(V,\rho_{\{,\}},\rho_{\circ})$. Then
  we conclude that the corresponding \qcl of $A_{h}
\ltimes_{-r_{\circ_{h}}^{\star},-l_{\circ_{h}}^{\star}}
V_{h}^{\star}$ is exactly the Poisson algebra $A\ltimes_{\rho_{\{,\}}^{\ast},-\rho_{\circ}^{\ast}}V^{\ast}$ by Corollary \ref{cor-defmod}.

(\ref{skews2}). Since $T_{\bfK}$ is the scalar deformation of $T$,
we obtain
$\widetilde{T_{\bfK}}=\hat{\tilde{T}}$
by (\ref{eq-tk}).
Let $\{e_{1},\ldots,e_{n}\}$ be a basis of $V$ and
$\{e_{1}^{\ast},\ldots,e_{n}^{\ast}\}$ be the
dual basis of $V^{\ast}$, where $n=\mathrm{dim}\,V$.
So $\tilde{T}=\sum^{n}_{i=1}T(e_i)\otimes e_{i}^{\ast}$.
Then
\vspb
\begin{equation}\label{ehal}
\eta_h(\widetilde{T_{\bfK}})=\eta_h(\hat{\tilde{T}})=\sum^{n}_{i=1}(T(e_i),0) \hat{\otimes} (0,e_{i}^{\ast})=\widehat{\eta(\tilde{T})}.
\vspb
\end{equation}
 Also note that
\vspb
 \begin{equation}\label{sehal}
    \sigma_{\bfK}(\eta_h(\widetilde{T_{\bfK}}))=\sigma_{\bfK}(\eta_h(\hat{\tilde{T}}))=\sum^{n}_{i=1}  (0,e_{i}^{\ast})\hat{\otimes}(T(e_i),0) =\widehat{\sigma(\eta(\tilde{T}))}.
\vspb
 \end{equation}
Thus by (\ref{ehal}) and (\ref{sehal}), we obtain
\vspb
\[\eta_h(\widetilde{T_{\bfK}})=\widehat{\eta(\tilde{T})}\;\textrm{and}\;\sigma_{\bf K}(\eta_h(\widetilde{T_{\bfK}}))=\widehat{\sigma(\eta(\tilde{T}))},
\vspb
\]
respectively. Thus $r_h$ coincides with  the scalar deformation
$\hr$ of $r$, proving the conclusion.
\end{proof}
\vspb

Let $(A,\succ)$ be a finite-dimensional Zinbiel algebra, $(A_{h},\succ_{h},\prec_{h})$ be a dendriform deformation and $(A,\triangleright,\succ)$ be the \qcl pre-Poisson algebra. Define binary operations $\circ$ and $\{,\}$ on $A$ by
\vspb
$$x\circ y:=x\succ y+y\succ x,\;\; \{x,y\}:=x\triangleright
y-y\triangleright x,\;\; x,y\in A.$$ Define a binary operation
$\circ_{h}$ on $A_h$ by
$x_{h}\circ_{h}y_{h}:=x_{h}\succ_{h}
y_{h}+x_{h}\prec_{h}y_{h}$ for $x_{h},y_{h}\in A_{h}$. Denote the isomorphic image of $\mathrm{id}_{A}$ in $A\ot
A^{\ast}$ by $\widetilde{\mathrm{id}_A}$ and embed it into $(A\oplus
A^{\ast})^{\ot 2}$ by $\eta$ defined by (\ref{eq:sumembed}). Denote the isomorphic image of
$\mathrm{id}_{A_h}$ in $A_h\ot A^{\star}_h$ by
$\widetilde{\mathrm{id}_{A_h}}$ and embed
$\widetilde{\mathrm{id}_{A_h}}$ into $(A\oplus
A^{\ast})_{h}\hat{\ot}(A\oplus A^{\ast})_{h}$ by $\eta_h$ defined
by \eqref{eq:hsumembed}. Set
$$r:=\eta(\widetilde{\mathrm{id}_A})-\sigma(\eta(\widetilde{\mathrm{id}_A})),\;\;r_h:=\eta_h(\widetilde{\mathrm{id}_{A_h}})-\sigma_{\bf
K}(\eta_h(\widetilde{\mathrm{id}_{A_h}})).$$ Note that
$\mathrm{id}_{A_h}$ is exactly the scalar deformation of
$\mathrm{id}_{A}$.
\begin{cor}\label{cor:summary}
Keep the notations and assumptions as above.
    \begin{enumerate}
        \item\label{m1} The topologically free associative algebra  $A_{h} \ltimes_{-R_{\prec_{h}}^{\star},-L_{\succ_{h}}^{\star}} A_{h}^{\star}$ is an associative deformation of the commutative associative
        algebra $A\ltimes_{-L_{\succ}^{\ast}}A^{\ast}$, and the corresponding \qcl is exactly the Poisson algebra $A\ltimes_{L_{\triangleright}^{\ast},-L_{\succ}^{\ast}}A^{\ast}$.
        \item\label{m2} Starting from a Zinbiel algebra $(A,\succ)$, the left half of the following diagram commutes. Moreover, starting from the topologically free dendriform algebra $(A_{h},\succ_{h},\prec_{h})$,
        the right half of the following diagram commutes.
    \end{enumerate}
 \begin{equation}\label{maindiag}
\begin{split}
\xymatrixcolsep{5pc}\xymatrixrowsep{3pc}
\xymatrix{
*\txt{\small Zinbiel algebra\\ \small $(A,\succ)$}
\ar[r]^{\txt{\tiny deformation}}\ar[d]|-{\txt{\tiny
Proposition~\ref{pro-ct}}}& *\txt{\small topologically free\\
\small dendriform algebra\\ \small
$(A_{h},\succ_{h},\prec_{h})$}\ar[r]^{\txt{ \tiny
\qcl}}_{\txt{\tiny
Corollary~\mref{cor-defdendri}}}\ar[d]|-{\txt{\tiny
Proposition~\ref{pro-ct}}}
& *\txt{\small pre-Poisson algebra\\\small $(A,\triangleright,\succ)$}\ar[d]|-{\txt{\tiny  Proposition~\ref{pro-cp}}}\\
*\txt{\small $\mathcal{O}$-operator $\mathrm{id}_{A}$ on\\ \small $(A,\circ)$ associated\\ \small to $(A,L_{\succ})$}\ar[r]^{\txt{\tiny  deformation}}\ar[d]|-{\txt{\tiny  Proposition~\mref{aybe-semi}\\ \tiny $r:=\eta(\widetilde{\mathrm{id}_A})-\sigma(\eta(\widetilde{\mathrm{id}_A}))$ }}&*\txt{\small $\mathcal{O}$-operator $\mathrm{id}_{A_h}$ on\\ \small $(A_h,\circ_h)$ associated\\ \small to $(A_h,L_{\succ_h},R_{\prec_h})$}\ar[r]^{\txt{\tiny QCL}}_{\txt{\tiny  Theorem \mref{invrb}
}}\ar[d]|-{\txt{\tiny  Proposition~\mref{taybe-semi}\\\tiny$r_h:=\eta_h(\widetilde{\mathrm{id}_{A_h}})-\sigma_{\bf K}(\eta_h(\widetilde{\mathrm{id}_{A_h}}))$}}&*\txt{\small $\mathcal{O}$-operator $\mathrm{id}_{A}$ on\\ \small $(A,\{,\},\circ)$ associated\\ \small to $(A,L_{\triangleright},L_{\succ})$}\ar[d]|-{\txt{\tiny  Theorem~\mref{pybe-semi}\\\tiny  $r:=\eta(\widetilde{\mathrm{id}_A})-\sigma(\eta(\widetilde{\mathrm{id}_A}))$}}\\
*\txt{\small skew-symmetric\\\small solution  $r$ of\\ \small AYBE in $A\ltimes_{-L_{\succ}^{\ast}}A^{\ast}$}\ar[r]^{\txt{\tiny  deformation}}
&*\txt{\small skew-symmetric\\\small solution $r$ of\\\small AYBE in $A_{h} \ltimes_{-R_{\prec_{h}}^{\star},-L_{\succ_{h}}^{\star}} A_{h}^{\star}$}\ar[r]^{\txt{\tiny QCL}}_{\txt{\tiny  Theorem~\mref{inv-aybe}}}
&*\txt{\small skew-symmetric\\\small solution  $r$ of\\\small PYBE in $A\ltimes_{L_{\triangleright}^{\ast},-L_{\succ}^{\ast}}A^{\ast}$.}}
\end{split}
\end{equation}
\end{cor}
\begin{proof}
     It follows from Theorems~\ref{pro-diaid}
and \ref{pro-skews}.
\end{proof}
\subsection{Solutions of the PYBE from tridendriform deformations}
\mlabel{ss:ybetrid}

Tridendriform algebras correspond to $A$-bimodule algebras equipped with
the $\calo$-operator $\id_A$ of weight 1 and hence give rise to
solutions of the AYBE as in the following result.

\begin{pro}{\rm{\cite[Corollary~6.19]{BGN3}}}
Let $(A,\succ,\prec,\cdot)$ be a finite-dimensional tridendriform algebra and $A^{\oplus 2}$
be the underlying vector space of the associative algebra $(\hat
A:=A\ltimes_{L_{\succ},R_{\prec}}A,\odot)$. Define the operators
$\alpha_1,\alpha_2,\alpha_3,\alpha_4\in
\mathrm{End}_{\bfk}(A^{\oplus 2})$ by \vspb
\begin{equation}
\alpha_{1}(x,y)=(-x+y,0), \ \alpha_{2}(x,y)=-(y,y),\ \alpha_{3}(x,y)=-(x,x),\ \ \alpha_{4}(x,y)=(0,x-y),\  x, y\in A.
\mlabel{alpha}
\vspb
\end{equation}
Denote the identity operator on $A^{\oplus 2}$ by $\mathrm{id}$.
For $i\in \{1,2,3,4\}$, embed $\tilde{\alpha_i}$ and
$\widetilde{\mathrm{id}}$ into $(A^{\oplus 2}\oplus (A^{\oplus
2})^{\ast})^{\otimes^2}$ by the embedding $\eta$ defined
by \meqref{eq:sumembed}. Then
$\eta(\tilde{\alpha_i})-\sigma(\eta(\tilde{\alpha_i}))+
\eta(\widetilde{\mathrm{id}})$ and
$\eta(\tilde{\alpha_i})-\sigma(\eta(\tilde{\alpha_i}))-
\sigma(\eta(\widetilde{\mathrm{id}}))$
    are solutions of the AYBE in the associative algebra $\hat A\ltimes_{-R_{\odot}^{\ast},-L_{\odot}^{\ast}}(A^{\oplus 2})^{\ast}$.
\mlabel{tri-aybe}
\end{pro}
\begin{rmk}\label{rmk-rbaybe}
The above constructions of solutions of the AYBE are special cases of more general
    results where the constructions of solutions of the AYBE utilize Rota-Baxter operators (see \cite[Proposition~3.9(e)]{BGN3}).
    In fact, since $\mathrm{id}_A$ is an invertible $\mathcal{O}$-operator on $(A,\circ)$ of weight $1$ associated to $(A,\cdot,L_{\succ},R_{\prec})$,
    each $\alpha_i$ is a Rota-Baxter operator of weight $1$ on $(\hat A=A\ltimes_{L_{\succ},R_{\prec}}A,\odot)$ (see  \cite[Propositions~6.1 and 6.8]{BGN3}).
\end{rmk}
The following result gives the topologically free tridendriform
algebra analogue of Proposition \ref{tri-aybe} and can be proved in a
similar manner (see the proof of \cite[Corollary 6.19]{BGN3}).
\begin{pro}\label{tri-taybe}
Let $A$ be a finite-dimensional vector
space and $(A_h,\succ_{h},\prec_{h},\cdot_{h})$ be a topologically
free tridendriform algebra. Denote the underlying topologically
free $\bf K$-module of the topologically free associative algebra
$\hat A_h:=(A_{h}\ltimes_{L_{\succ_h},R_{\prec_h}}A_{h},\odot_h)$
by $(A^{\oplus 2})_{h}$. Let
$\alpha_{1,h},\alpha_{2,h},\alpha_{3,h},\alpha_{4,h}\in
\mathrm{End}_{\bfK}((A^{\oplus 2})_{h})$ be the operators defined
by \vspb
\begin{eqnarray*}
&&\alpha_{1,h}(x_h,y_h)=(-x_h+y_h,0), \ \alpha_{2,h}(x_h,y_h)=-(y_h,y_h),\\ &&\alpha_{3,h}(x_h,y_h)=-(x_h,x_h),\ \ \alpha_{4,h}(x_h,y_h)=(0,x_h-y_h),\  x_h, y_h\in A_h.
\mlabel{alpha_h}
\vspb
\end{eqnarray*}
Denote the identity operator on $(A^{\oplus 2})_{h}$ by
$\mathrm{id}_h$. For $i\in \{1,2,3,4\}$, embed
$\widetilde{\alpha_{i,h}}$ and $\widetilde{\mathrm{id}_h}$ into
$(A^{\oplus 2}\oplus (A^{\oplus
2})^{\ast})_h\hat{\otimes}(A^{\oplus 2}\oplus (A^{\oplus
2})^{\ast})_{h}$ by the embedding $\eta_h$ defined by
\eqref{eq:hsumembed}. Then \vspb
$$\eta_h(\widetilde{\alpha_{i,h}})-\sigma_{\bf K}(\eta_h(\widetilde{\alpha_{i,h}}))+
\eta_h(\widetilde{\mathrm{id}_h})\; \textrm{and} \;\eta_h(\widetilde{\alpha_{i,h}})-\sigma_{\bf K}(\eta_h(\widetilde{\alpha_{i,h}}))-
\sigma_{\bf K}(\eta_h(\widetilde{\mathrm{id}_h}))
\vspb
$$
are solutions of the AYBE in $\hat
A_h\ltimes_{-R_{\odot_{h}}^{\star},-L_{\odot_{h}}^{\star}}(A^{\oplus
2})^{\star}_{h}$.
\end{pro}

Let $(\mathfrak{g},\{,\})$ be a Lie algebra. Taking the operad of Lie algebras in Definition~\mref{de:pbimod}, we recover the notion of a $\mathfrak{g}$-module Lie algebra~\cite{BGN1,Lue}.
%\sy{The reference has been corrected.}
Let $(\mathfrak{h},[,],\rho_{\{,\}})$ be a $\mathfrak{g}$-module Lie algebra. Then taking the operad of Lie algebras in Definition~\mref{de:oop}, we recover the notion of an $\mathcal{O}$-operator of weight $\lambda$
on $\mathfrak{g}$ associated to
$(\mathfrak{h},[,],\rho_{\{,\}})$ \cite{BGN1}. A linear map $T\in\mathrm{End}_{\bfk}(\mathfrak{g})$ is called a Rota-Baxter operator of weight $\lambda$ on $\mathfrak{g}$ if it satisfies (\ref{rbl}).

\begin{thm}
Let $(A,[,],\triangleright,\succ,\cdot)$ be a finite-dimensional post-Poisson algebra. Denote the
underlying vector space of the Poisson algebra
$(\check{A}:=A\ltimes_{L_{\triangleright},L_{\succ}}A,\llbracket,\rrbracket,\odot)$
by $A^{\oplus 2}$. Let
$\alpha_i$,$\mathrm{id},\tilde{\alpha_i},\widetilde{\mathrm{id}}$
and $\eta$ be as defined in Proposition \ref{tri-aybe}. Then
$\eta(\tilde{\alpha_i})-\sigma(\eta(\tilde{\alpha_i}))+
\eta(\widetilde{\mathrm{id}})$ and
$\eta(\tilde{\alpha_i})-\sigma(\eta(\tilde{\alpha_i}))-
\sigma(\eta(\widetilde{\mathrm{id}}))$ are solutions of the PYBE
in the Poisson algebra
$\check{A}\ltimes_{\rad_{\llbracket,\rrbracket}^{\ast},-R_{\odot}^{\ast}}(A^{\oplus
2})^{\ast}$. \mlabel{post-pybe}
\end{thm}
  \begin{proof}
    Since $(A,\succ,\cdot)$ is a finite-dimensional tridendriform algebra, Proposition \ref{tri-aybe} shows that $\eta(\tilde{\alpha_i})-\sigma(\eta(\tilde{\alpha_i}))+ \eta(\widetilde{\mathrm{id}})$ and $\eta(\tilde{\alpha_i})-\sigma(\eta(\tilde{\alpha_i}))- \sigma(\eta(\widetilde{\mathrm{id}}))$ are solutions of the AYBE in $(A\ltimes_{L_{\succ}}A)\ltimes_{-R_{\odot}^{\ast}}(A^{\oplus 2})^{\ast}$.

 By Proposition \ref{pro-cp},
$(A,[,],L_{\triangleright})$ is an $(A,\{,\})$-module Lie algebra for $\{,\}$ defined in (\ref{eq:summ}), and $\mathrm{id}_{A}$
is an $\mathcal{O}$-operator of weight $1$ on $(A,\{,\})$
associated to $(A,[,],L_{\triangleright})$. Thus by
\cite[Proposition 4.1]{BGN1}, $\alpha_{1}$ and $\alpha_{2}$
in \meqref{alpha} are Rota-Baxter operators of weight $1$ on the Lie algebra
$A\ltimes_{L_{\triangleright}}A$ respectively. It is
straightforward to show that $\alpha_{3}$ and $\alpha_{4}$ are
also Rota-Baxter
 operators of weight $1$ on the Lie algebra
 $A\ltimes_{L_{\triangleright}}A$ respectively. That is, $\alpha_i$ is a Rota-Baxter
 operator of weight $1$ on the Lie algebra
 $A\ltimes_{L_{\triangleright}}A$ for any
 $\alpha_i\in\{\alpha_1,\alpha_2,\alpha_3,\alpha_4\}$.
 Thus by \cite[Corollary 2.16(d)]{BGN1},
$\eta(\tilde{\alpha_i})-\sigma(\eta(\tilde{\alpha_i}))+ \eta(\widetilde{\mathrm{id}})$ and $\eta(\tilde{\alpha_i})-\sigma(\eta(\tilde{\alpha_i}))- \sigma(\eta(\widetilde{\mathrm{id}}))$
 are also solutions of the CYBE in  $(A\ltimes_{L_{\triangleright}}A)\ltimes_{\rad_{\llbracket,\rrbracket}^{\ast}}(A^{\oplus 2})^{\ast}$.
\end{proof}

By a straightforward calculation, we obtain the following result.
\begin{lem}
Let $A$ be a finite-dimensional vector space and
$(A_{h},\circ_{h})$ be a topologically free associative algebra.
Denote the isomorphic image of $\mathrm{id}_{A}$ in $A^{\ast}\ot
A$ by  $\widetilde{\mathrm{id}_A}$ and embed it to $(A\oplus
A^{\ast})^{\otimes {2}}$ via $\eta$ defined by
\meqref{eq:sumembed}. Denote the isomorphic image of
$\eta(\widetilde{\mathrm{id}_A})$ under the linear embedding of
$(A\oplus A^{\ast})^{\otimes 2}$ into $(A\oplus
A^{\ast})_h\hat{\otimes}(A\oplus A^{\ast})_{h}$ by
$\widehat{\eta(\widetilde{\mathrm{id}_{A}})}$. Then the symmetric
part of $\widehat{\eta(\widetilde{\mathrm{id}_{A}})}$ is
$(A_{h}\ltimes_{-R_{\circ_{h}}^{\star},-L_{\circ_{h}}^{\star}}
A^{\star}_{h})$-invariant. \mlabel{idinv}
\end{lem}
\vspb
Let $(A,\succ,\cdot)$ be a finite-dimensional commutative tridendriform algebra,
$(A,\succ_{h},\prec_{h},\cdot_{h})$ be a tridendriform
deformation and $(A,\triangleright,[,],\succ,\cdot)$ be its \qcl
which is a post-Poisson algebra. Define binary operations $\circ$
and $\{,\}$ on $A$ by (\ref{eq:summ}). Define a binary operation
$\circ_{h}$ on $A_h$ by (\ref{eq:sumh}). Thus by
Theorem~\ref{pro-diaid}, the topologically free
$(A_{h},\circ_{h})$-bimodule algebra $(A_{h},\cdot
    _{h},L_{\succ_{h}},R_{\prec_{h}})$ is a deformation of the $(A,\circ)$-module algebra $(A,\cdot,L_{\succ})$ and its \qcl is the
$(A,\{,\},\circ)$-module Poisson algebra
$(A,[,],\cdot,L_{\triangleright},L_{\succ})$. So we have the
following data:
\begin{itemize}
\item  two commutative associative algebras
$(\hat{A}:=A\ltimes_{L_{\succ}}A,\odot)$ and
$\hat{A}\ltimes_{-R_{\odot}^{\ast}}(A^{\oplus 2})^{\ast}$;
\item
two topologically free associative algebras
$(\hat{A}_{h}:=A_{h}\ltimes_{L_{\succ_{h}},R_{\prec_{h}}}A_{h},\odot_{h})$
and
$\hat{A}_{h}\ltimes_{-R_{\odot_{h}}^{\star},-L_{\odot_{h}}^{\star}}(A^{\oplus
2})^{\star}_{h}$;
\item two Poisson algebras
$(\check{A}:=A\ltimes_{L_{\triangleright},L_{\succ}}A,\llbracket,\rrbracket,\odot)$
and
$\check{A}\ltimes_{\rad_{\llbracket,\rrbracket}^{\ast},-R_{\odot}^{\ast}}(A^{\oplus
2})^{\ast}$,
\end{itemize}
where $A^{\oplus 2}$ is the underlying vector space of
$A\ltimes_{L_{\succ}}A$. Fix an $i\in\{1,2,3,4\}$.
 Let $\alpha_i$, $\mathrm{id}$ and $\eta$ be the linear maps and $\tilde{\alpha_i}$ and
 $\widetilde{\mathrm{id}}$ be tensors given in Proposition
\ref{tri-aybe}. Let $\alpha_{i,h}$, $\mathrm{id}_h$ and $\eta_h$ be the $\bfK$-linear maps and $\widetilde{\alpha_{i,h}}$ and
$\tilde{\mathrm{id}_h}$ be tensors given in Proposition
\ref{tri-taybe}. 

\begin{thm}
Keep the above assumptions and notations.
\begin{enumerate}
\item\label{w2.1}  The topologically free associative algebra $\hat{A}_{h}\ltimes_{-R_{\odot_{h}}^{\star},-L_{\odot_{h}}^{\star}}(A^{\oplus 2})^{\star}_{h}$ is an associative deformation of the commutative associative algebra $\hat{A}\ltimes_{-R_{\odot}^{\ast}}(A^{\oplus 2})^{\ast}$ and the corresponding \qcl is exactly the Poisson algebra $\check{A}\ltimes_{\rad_{\llbracket,\rrbracket}^{\ast},-R_{\odot}^{\ast}}(A^{\oplus 2})^{\ast}$.
\item
Starting with a commutative tridendriform algebra $(A,\succ,\cdot)$, the left square of the following diagram commutes.
%\mlabel{w2.2}
%\item
Moreover, starting from a topologically free tridendriform algebra $(A_{h},\succ_{h},\prec_{h},\cdot_{h})$, the right square of the following diagram commutes.
\mlabel{w2.3}
\end{enumerate}
\vspa
{\begin{displaymath}
 \xymatrixcolsep{6pc}\xymatrixrowsep{5pc}
 \xymatrix{
  *\txt{\small commutative\\ \small tridendriform\\ \small algebra $(A,\succ,\cdot)$}
   \ar[r]^{\txt{\tiny deformation}}\ar[d]|-{\txt{\tiny Proposition \mref{tri-aybe}\\\tiny $r=\eta(\tilde{\alpha_i})-\sigma(\eta(\tilde{\alpha_i}))+
\eta(\widetilde{\mathrm{id}})$\\\tiny($r=
\eta(\tilde{\alpha_i})-\sigma(\eta(\tilde{\alpha_i}))-
\sigma(\eta(\widetilde{\mathrm{id}})))$}}& *\txt{\small topologically free\\ \small tridendriform algebra\\ \small $(A_{h},\succ_{h},\prec_{h},\cdot_{h})$}\ar[r]^{\txt{\tiny \qcl}}_{\txt{\tiny Theorem \mref{mainthm}}}\ar[d]|-{\txt{\tiny Proposition \mref{tri-taybe}\\\tiny$r_h=\eta_h(\widetilde{\alpha_{i,h}})-\sigma_{\bf K}(\eta_h(\widetilde{\alpha_{i,h}}))+
\eta_h(\widetilde{\mathrm{id}_h})$\\\tiny($r_h=\eta_h(\widetilde{\alpha_{i,h}})-\sigma_{\bf K}(\eta_h(\widetilde{\alpha_{i,h}})-
\eta_h(\widetilde{\mathrm{id}_h}))$)}}
   & *\txt{\small post-Poisson\\ \small algebra $(A,\{,\},\triangleright,\succ,\cdot)$}\ar[d]|-{\txt{\tiny Theorem \mref{post-pybe}\\\tiny$r=\eta(\tilde{\alpha_i})-\sigma(\eta(\tilde{\alpha_i}))+
\eta(\widetilde{\mathrm{id}})$\\\tiny($r=
\eta(\tilde{\alpha_i})-\sigma(\eta(\tilde{\alpha_i}))-
\sigma(\eta(\widetilde{\mathrm{id}})))$}}\\
   *\txt{\small solution $r$ of \\ \small AYBE in \\ \small $\hat{A}\ltimes_{-R_{\odot}^{\ast}}(A^{\oplus 2})^{\ast}$}\ar[r]^{\txt{\tiny deformation}}
   &*\txt{\small  solution $r_h$ of\\ \small AYBE in \\ \small  $\hat{A}_{h} \ltimes_{-R_{\odot_{h}}^{\star},-L_{\odot_{h}}^{\star}}(A^{\oplus 2})^{\star}_{h}$}\ar[r]^{\txt{\tiny QCL}}_{\txt{\tiny Lemma \ref{idinv}\\\tiny  Theorem~\mref{inv-aybe}}}
    &*\txt{\small solution $r$ of\\ \small PYBE in\\ \small  $\check{A}\ltimes_{\rad_{\llbracket,\rrbracket}^{\ast},-R_{\odot}^{\ast}}(A^{\oplus 2})^{\ast}$.}}
 \end{displaymath}}
\mlabel{pro-w2}
\vspb
\end{thm}
%\sy{The proof is revised.}
\begin{proof}(\ref{w2.1}). The proof is similar to the one in the proof of Theorem~\ref{pro-skews}~(\ref{skews1}).

Note that
\vspb
$$\alpha_{i,h}((x,y))=\alpha_{i}((x,y))=(\alpha_{i})_\bfK((x,y)),\;\forall x,y\in A.$$
Since $\alpha_{i,h}$ and $(\alpha_{i})_\bfK$ are in $\mathrm{End}_{\bfK}((A^{\oplus 2})_h)$ and their restrictions to $A^{\oplus2}$ coincide, we conclude that
 $\alpha_{i,h}$ coincides with $(\alpha_{i})_\bfK$.
Thus we obtain
$\widetilde{\alpha_{i,h}}=\hat{\tilde{\alpha_{i}}}$
by (\ref{eq-tk}).
Then similar to (\ref{ehal}), we have
\vspb
\begin{equation}\label{ehi}
\eta_h(\widetilde{\alpha_{i,h}})=\widehat{\eta(\tilde{\alpha_i})},\;\;\eta_h(\widetilde{\mathrm{id}}_h)=\widehat{\eta(\widetilde{\mathrm{id}})}.
\vspb
\end{equation}
 Also similar to (\ref{sehal}), we obtain
\vspb
 \begin{equation}\label{sehi}
    \sigma_{\bfK}(\eta_h(\widetilde{\alpha_{i,h}}))=\widehat{\sigma(\eta(\tilde{\alpha_i}))},\;\;\sigma_{\bfK}(\eta_h(\widetilde{\mathrm{id}}_h))=\widehat{\sigma(\eta(\tilde{\mathrm{id}}))}.
    \end{equation}
\vspb
 Thus we have
\vspb
 \begin{eqnarray*}
     \eta_h(\widetilde{\alpha_{i,h}})-\sigma_{\bf K}(\eta_h(\widetilde{\alpha_{i,h}}))+
\eta_h(\widetilde{\mathrm{id}_h})&=&\widehat{\eta(\widetilde{\alpha_i})}-\widehat{\sigma(\eta(\widetilde{\alpha_i}))}+
\widehat{\eta(\tilde{\mathrm{id}})},\\ \eta_h(\widetilde{\alpha_{i,h}})-\sigma_{\bf K}(\eta_h(\widetilde{\alpha_{i,h}}))-
\sigma_{\bf K}(\eta_h(\widetilde{\mathrm{id}_h}))&=&
\widehat{\eta(\tilde{\alpha_i})}-\widehat{\sigma(\eta(\tilde{\alpha_i}))}-
\widehat{\sigma(\eta(\tilde{\mathrm{id}}))}.
\vspb
 \end{eqnarray*}
We conclude that $r_h$ coincides with the scalar deformation $\hat r$ of $r$. 
Denote the symmetric part of $r_h$ by $\beta_h$. Then by the second equations in
(\ref{ehi}) and (\ref{sehi}), we have
\vspb
$$\beta_h=\pm\frac{\eta_h(\widetilde{\mathrm{id}_h})+\sigma_{\bf
        K}(\eta_h(\widetilde{\mathrm{id}_h}))}{2}=\pm\frac{\widehat{\eta(\widetilde{\mathrm{id}})}+\sigma_{\bf K}(\widehat{\eta(\tilde{\mathrm{id}})})}{2}.
\vspb
$$
So we deduce that the symmetric part of $r_h$ is $\hat{A}_{h}\ltimes_{-R_{\odot_{h}}^{\star},-L_{\odot_{h}}^{\star}}(A^{\oplus 2})^{\star}_{h}$-invariant by
applying Lemma \mref{idinv} to the topologically free associative algebra  $\hat{A}_{h}$.
 Then it is clear that (\ref{w2.3}) holds.
\end{proof}

\noindent {\bf Acknowledgements.} This work is supported by NSFC
(11931009, 12271265, 12261131498), Fundamental Research
Funds for the Central Universities and Nankai Zhide Foundation.
%, 12326319

\noindent
{\bf Declaration of interests. } The authors have no conflicts of interest to disclose.

\noindent
{\bf Data availability. } Data sharing is not applicable as no new data were created or analyzed.

\vspb
%%%%%%%%%%%%%%%%%%%%%%%%%%%%%%%%%%%%%%%%%%%%%%%%%%%%%%%%%%%%%%%%%%%%%%%%%%%%%%%%


\begin{thebibliography}{9}
\vspa
\mbibitem{A1}
M. Aguiar, Infinitesimal Hopf algebras. \textit{Contemp. Math.} \textbf{267}, 1-29, 2000.

\mbibitem{A2}
M. Aguiar, Pre-Poisson algebras. \textit{Lett. Math. Phys.} \textbf{54} (2000), 263-277.

\mbibitem{A3}
M. Aguiar, On the associative analog of Lie bialgebras. \textit{J. Algebra.} \textbf{244} (2001), 492-532.

\mbibitem{AZ}
K. Atalikov and A. Zotov, Higher rank $1+1$ integrable Landau-Lifshitz field theories from the associative Yang-Baxter equation. \textit{JETP Lett.}\textbf{115} (2022), 757-762.

\mbibitem{Bai0} C. Bai, A unified algebraic approach to the classical Yang-Baxter equation, {\em J. Phys. A} {\bf 40} (2007), 11073–11082.

\mbibitem{Bai2}  C. Bai, An introduction to pre-Lie algebras. In:
Algebra and Applications 1: Nonssociative Algebras and Categories, Wiley Online Library (2021), 245-273.

\mbibitem{BBGN}
 C. Bai, O. Bellier, L. Guo and X. Ni, Splitting of operations, Manin products and Rota-Baxter operators. \textit{Int. Math. Res. Not.} \textbf{3} (2013), 485-524.

\mbibitem{BGN1} C. Bai, L. Guo and X. Ni, Generalizations of the
classical Yang-Baxter equation and $\mathcal O$-operators.
\textit{J. Math. Phys.} \textbf{52} (2011), 063515.

\mbibitem{BGN2}
C. Bai, L. Guo and X. Ni, $\mathcal{O}$-operators on associative algebras and associative Yang-Baxter equation. \textit{Pacific J. Math.} \textbf{256} (2012), 257-289.

\mbibitem{BGN3}
C. Bai, L. Guo and X. Ni, $\mathcal{O}$-operators on associative algebras, associative Yang-Baxter equations and
dendriform algebras. In: Quantized Algebra and Physics. World Scientific, 2012, 10-51.

\mbibitem{BGN4}
C. Bai, L. Guo and X. Ni, Relative Rota-Baxter operators and tridendriform algebras. \textit{J. Algebra Appl.} \textbf{12} (2013), 1350027.

\mbibitem{D}
 A. Das, Cohomology and deformations of weighted Rota-Baxter operators. \textit{J. Math. Phys.} \textbf{63} (2022), 091703.

\bibitem{Do-PBW}
V. Dotsenko, Functorial PBW theorems for post-Lie algebras. \textit{Comm. Algebra.} \textbf{48} (2020), 2072-2080.

\mbibitem{ES}
P. Etingof and O. Schiffmann, Lectures on Quantum Groups, 2nd ed. International Press, 2010.

\mbibitem{G1}
M. Gerstenhaber, On the deformation of rings and algebras. \textit{Ann. Math.} \textbf{79} (1964), 59-103.

\mbibitem{G3}
M. Gerstenhaber, On the deformation of rings and algebras: \uppercase\expandafter{\romannumeral3}. \textit{Ann. Math.} \textbf{88} (1968), 1-34.

\mbibitem{Gu}
L. Guo, An Introduction to Rota-Baxter Algebra. International Press, Somerville, MA, 2012.

\mbibitem{HBG} Y. Hong, C. Bai and L. Guo, Infinite-dimensional Lie bialgebras via affinization of Novikov bialgebras and Koszul duality. \textit{Comm. Math. Phys.} \textbf{401} (2023), 2011-2049.

\mbibitem{Ka}
C. Kassel, Quantum groups. Springer-Verlag, 1995.

\mbibitem{Ko}
M. Kontsevich, Deformation quantization of Poisson manifolds \uppercase\expandafter{\romannumeral1}. \textit{Lett. Math. Phys.} \textbf{66} (2003), 157-216.

\mbibitem{Ku}
B. A. Kupershmidt, What a classical r-matrix really is. \textit{J. Nonlinear Math. Phy.} \textbf{6} (1999), 448-488.

\mbibitem{LBS}
 J. Liu, C. Bai and Y. Sheng, Noncommutative Poisson bialgebras. \textit{J. Algebra.} \textbf{556} (2020), 35-66.

\bibitem{LS}
 J. Liu and Y. Sheng, Cohomologies of PoiMod pairs and compatible structures on Poisson algebras. \textit{J. Geom. Phys.} \textbf{173} (2022), 104449.

\mbibitem{Lo}
J.-L. Loday, Dialgebras. In: Dialgebras and Related Operads. \textit{Lect. Notes in Math} \textbf{1763}, 7-66, 2002.

\mbibitem{LR}
J.-L. Loday and M. Ronco, Trialgebras and families of polytopes,. \textit{Contemp. Math.} \textbf{346}, 369-398, 2004.

\mbibitem{LV}
J.-L. Loday and B. Vallette, Algebraic Operads. Springer, 2012.

\mbibitem{Lue}
S.-T. Lue, Crossed homomorphisms of Lie algebras.
{\em Proc. Cambridge Philos. Soc.} {\bf 62} (1966),  577-581.

\mbibitem{Lue2} S.-T. Lue, Non-abelian cohomology of associative algebras, {\em J. Math. Oxford} {\bf 19} (1968), 159-180.

\mbibitem{LM}
A. Lundervold and H. Munthe-Kaas, On post-Lie algebras, Lie-Butcher series and moving frames. \textit{Found. Comput. Math.} \textbf{13} (2013), 583-613.

\mbibitem{Me}
Z. Mesyan, The ideals of an ideal extension. {\em J. Algebra Its Appl.} {\bf 9} (2010), 407-431.

\mbibitem{MS}
 A.-V. Mikhailov and V.-V. Sokolov, Integrable ODEs on Associative Algebras. \textit{Comm. Math. Phys.} \textbf{211} (2000), 231-251.

\mbibitem{NB}
X. Ni and C. Bai, Poisson bialgebras. \textit{J. Math. Phys.} \textbf{54} (2013), 023515.

\mbibitem{OPV} C. Ospel, F. Panaite and P. Vanhaecke, Polarization
and deformations of generalized dendriform algebras. \textit{J.
Noncommut. Geom.} \textbf{16} (2022), 561-594.

\mbibitem{P1}
A. Polishchuk, Classical Yang-Baxter equation and the $A_{\infty}$-constraint. \textit{Adv. Math.} \textbf{168} (2002), 56-95.

\mbibitem{P2}
 A. Polishchuk, Massey products on cycles of projective lines and trigonometric solutions of the Yang-Baxter equations. In: Algebra, arithmetic, and geometry: in honor of Yu. I. Manin. Vol. \uppercase\expandafter{\romannumeral2}. \textit{Prog. Math.} \textbf{270}, 573-617, 2009.

\mbibitem{U}
 K. Uchino, Quantum analogy of Poisson geometry, related dendriform algebras and Rota-Baxter operators. \textit{Lett.
Math. Phys.} \textbf{85} (2008), 91-109.

\mbibitem{V}
B. Vallette, Homology of generalized partition posets. \textit{J. Pure Appl. Algebra.} \textbf{208} (2007), 699-725.

\end{thebibliography}
\end{document}